\documentclass[10pt]{amsart}
\usepackage{geometry}                % See geometry.pdf to learn the layout options. There are lots.
\usepackage{graphicx}
\usepackage{dsfont}
\usepackage{amssymb}
\usepackage{amsmath}
\usepackage{epstopdf}
\usepackage{fullpage}
\usepackage{enumerate}
\usepackage{color}
\usepackage{amsthm}
\usepackage{placeins}
\usepackage{longtable}
\usepackage{array}
\usepackage{esint}
\numberwithin{equation}{section}

\let\oldtocsection=\tocsection

\let\oldtocsubsection=\tocsubsection

\let\oldtocsubsubsection=\tocsubsubsection

\renewcommand{\tocsection}[2]{\hspace{0em}\oldtocsection{#1}{#2}\textbf}
\renewcommand{\tocsubsection}[2]{\hspace{1em}\oldtocsubsection{#1}{#2}}
\renewcommand{\tocsubsubsection}[2]{\hspace{2em}\oldtocsubsubsection{#1}{#2}}

\newtheorem{theorem}{Theorem}[section]
\newtheorem{lemma}[theorem]{Lemma}
\newtheorem{proposition}[theorem]{Proposition}
\newtheorem{corollary}[theorem]{Corollary}
\newtheorem{definition}{Definition}[section]
\theoremstyle{definition}

\newtheorem{remark}{Remark}[section]

\theoremstyle{definition}

\begin{document}

\newenvironment{pr oof}[1][D\'emonstration]{\begin{trivlist}
\item[\hskip \labelsep {\bfseries #1}]}{\end{trivlist}}
\newenvironment{example}[1][Example]{\begin{trivlist}
\item[\hskip \labelsep {\bfseries #1}]}{\end{trivlist}}

\title{Homogenization of the eigenvalues of the Neumann-Poincar\'e operator}
\author{
\'E. Bonnetier\textsuperscript{1}, C. Dapogny\textsuperscript{1}, and F. Triki\textsuperscript{1}
}

\maketitle

\begin{center}
\emph{\textsuperscript{1} Laboratoire Jean Kuntzmann, CNRS, Universit\'e Grenoble-Alpes, BP 53, 38041 Grenoble Cedex 9, France}.
\end{center}

%%%%%%%%%%%%%%%%%%%%%%%%%%%%%%%%%%%%%%%%%%%%%%%%%%%%%%%%%%%%%%%%%%%%%%%%%%%%%%%%%%%%%%%%%%%%%%%%
\begin{abstract}
In this article, we investigate the spectrum of the Neumann-Poincar\'e operator associated 
to a periodic distribution of small inclusions with size $\varepsilon$, and its asymptotic behavior as the parameter $\varepsilon$ vanishes. 
Combining techniques pertaining to the fields of homogenization and potential theory, we prove that the limit spectrum 
is composed of the `trivial' eigenvalues $0$ and $1$, and of a subset which stays bounded away from $0$ and $1$ uniformly with respect to $\varepsilon$. 
This non trivial part is the reunion of the \textit{Bloch spectrum}, accounting for the collective resonances between collections of inclusions, 
and of the \textit{boundary layer spectrum}, associated to eigenfunctions which spend a not too small part of their energies near the boundary of the macroscopic device.
These results shed new light about the homogenization of the voltage potential $u_\varepsilon$ caused by a given source in a medium composed of a periodic distribution 
of small inclusions with an arbitrary (possible negative) conductivity $a$, surrounded by a dielectric medium, with unit conductivity. In particular, we prove that the limit behavior of $u_\varepsilon$ is strongly related to the (possibly ill-defined) homogenized diffusion matrix predicted by the homogenization theory in the standard elliptic case. 
Additionally, we prove that the homogenization of $u_\varepsilon$ is always possible when $a$ is either positive, or negative with a `small' or `large' modulus.
\end{abstract}
%%%%%%%%%%%%%%%%%%%%%%%%%%%%%%%%%%%%%%%%%%%%%%%%%%%%%%
\bigskip
\bigskip
%%%%%%%%%%%%%%%%%%%%%%%%%%%%%%%%%%%%%%%%%%%%%%%%%%%%%%%%%%%%%%%%%%%%%%%%%%%%%%%%%%%%%%%%%%%%%%%%

\hrule
\tableofcontents
\hrule

\bigskip 
\bigskip 

\def\a{{\alpha}}
\def\b{{\beta}}
\def\d{{\delta}}
\def\e{{\varepsilon}}
\def\R{{I \kern -.30em R}}
\def\ds{{\displaystyle}}

%%%%%%%%%%%%%%%%%%%%%%%%%%%%%%%%%%%%%%%%%%%%%%%%%%%%%%%
\section{Introduction}\label{secintro}
%%%%%%%%%%%%%%%%%%%%%%%%%%%%%%%%%%%%%%%%%%%%%%%%%%%%%%%

Perhaps the first historical evidence of plasmonic resonances 
is the Lygurcus cup, a piece of antique art modeled around the fourth century A.D.: 
when illuminated from the outside (i.e. the reflected rays are observed), the cup appears green; 
however, when illuminated from the inside (i.e. the transmitted rays are observed), it looks red.
The reason for this stupendous
behavior is that the cup is encrusted with a colloid of silver and gold nanoparticles, 
that cause a very strong scattering enhancement at wavelengths
close to $350$ nm; see for instance \cite{manley}.
An analogous (heuristic) consideration guided the fabrication
of the stained glass adorning medieval cathedrals. 

Over the last decades, plasmonic resonances have made possible 
several breakthroughs to face the
ever increasing need for imaging small features, and more generally for concentrating electromagnetic energy in very small regions. 
A surface plasmon is an electromagnetic wave which is confined to the interface between 
two media whose permittivities have opposite signs. 
This rather unusual situation arises for instance at the interface between a dielectric medium (such as void, air), and a metal at infrared or optical frequency.  
Surface plasmons have the unique ability to confine electromagnetic energy in
areas much smaller than the wavelength of the incident wave 
(and thus smaller that the diffraction limit which constrains conventional
dielectric media); 
this opens the door to many interesting applications, ranging from the design of efficient photovoltaic cells (whose light absorption has been enhanced thanks to the strong absorption bands of metal nanoparticles), to cancer therapy (gold nanoparticles being compatible with in vivo detection since they present no toxicity \cite{elsayed}).   
In biochemistry, the unique sensitivity of surface plasmons to the local shape of the dielectric-metal interface has inspired valuable enhancements of spectroscopic methods, 
allowing for instance very accurate observations of molecular adsorption phenomena in the context of DNA, polymers, or proteins \cite{patching,trivau}.
See for instance the monograph \cite{maier} for further details about plasmonic resonances and possible applications.\\

From the mathematical point of view, plasmonic resonances have attracted a great deal of attention - see \cite{grieser} for a review of some results.
To set ideas, consider a domain $D$ (representing a metallic particle or film, or a collection of such particles) 
with electric permittivity $\varepsilon_c := \varepsilon_c^\prime + i \varepsilon_c^{\prime\prime}$,
embedded in a matrix with permittivity $\varepsilon_m >0$ (both media have identical magnetic permeability $\mu >0$). 
In the considered applications, the real part $\varepsilon_c^\prime$ is negative (as is the case of metals in the infrared-visible 
light regime);
the imaginary part $\varepsilon_c^{\prime\prime} >0$ accounts for
the absorption of electromagnetic waves by the inclusions
and the dissipation of the corresponding energy as heat. 
In the setting of the time-harmonic Helmholtz equation for the transverse magnetic (TM) mode,
a surface plasmon is defined as a non trivial solution $u$ to the system governing the intensity of the magnetic field in the direction transverse to the propagation of the wave: 
% Several conventions in the literature as far as the transverse magnetic terminology is concerned.
% See for instance the book of Maier for this version of the TM mode, depending on the geometry at stake.
$$ \left\{Ê
\begin{array}{cl}
-\text{\rm div}(\frac{1}{\varepsilon_m} \nabla u) + \mu \omega^2 u = 0 & \text{outside } D,\\
-\text{\rm div}(\frac{1}{\varepsilon_c} \nabla u) + \mu \omega^2 u = 0 & \text{in } D,\\
\frac{1}{\varepsilon_c} \frac{\partial u^-}{\partial n} = \frac{1}{\varepsilon_m} \frac{\partial u^+}{\partial n} & \text{on } \partial D, \\
+ \text{ other homogeneous B.C.}
\end{array}
\right. $$
Such a solution may only exist provided that $\varepsilon_c^\prime < 0$, 
and that the absorption coefficient vanishes ($\varepsilon_c^{\prime\prime} = 0$), in the limit where the size of the particles goes to $0$. 
In our dimensionless setting where this size is fixed, this corresponds to the limit $\omega \to 0$. 
Hence, studying the asymptotic regime when both $\varepsilon^{\prime\prime}_c$ and $\omega$
vanish provides valuable information in terms of designing nearly-resonant
structures.

That this quasi-static approximation is indeed an accurate description
of the physical phenomenon of plasmonic resonances
has been mathematically justified in the work of H. Ammari and his collaborators for the Maxwell system; see \cite{AmmariDengMillien,AmmariRuizYuZhang}. 
The scalar case of the Helmholtz equation 
has been investigated further in \cite{AmmariMillienRuizZhang}.

In this spirit the present contribution to the study of plasmonic resonances focuses on the quasistatic scalar equation for the voltage potential $u$:
\begin{equation} \label{eq.AD}
-\textrm{div}(A_D(x) \nabla u) = f, \text{ where } A_D (x) := \left\{Ê
\begin{array}{cl}
1 & \text{if }Êx \notin D,\\
a & \text{if } x \in D.
\end{array}
\right.   
\end{equation}
Here, $f$ stands for a source acting in the medium, and the conductivity $a$ inside the inclusions has a complex value; 
notice that adequate boundary conditions have to be added to (\ref{eq.AD}) in order to describe the behavior ot the field $u$ at the
boundary of the external device. 
When the imaginary part of the conductivity in the inclusions is small, but positive (or negative),
the equation (\ref{eq.AD}) is elliptic and 
the associated partial differential operator has a bounded inverse. 
However, as this imaginary part vanishes, 
the inverse operator may blow up for particular (negative) values of the real part of $a$ - our \textit{plasmonic eigenvalues}.
In such a situation, the voltage potential $u$ concentrates around
the boundaries of the inclusions and presents large gradients in
these regions.\\

The invertibility of the partial differential operator
associated to (\ref{eq.AD}) in the case where $a$ assumes a real, negative value, has been investigated from 
different viewpoints.
M. Costabel and E. Stephan~\cite{CostabelStephan} 
used the framework of integral equations
(and of the Neumann-Poincar\'e operator) in the case where $D$ is a piecewise smooth
inclusion with corners. 
Potential theory has also been used by H. Ammari and his collaborators in \cite{AmmariDengMillien, AmmariMillienRuizZhang,AmmariRuizYuZhang}.
A.S. Bonnet-Ben Dhia, P. Ciarlet and their collaborators \cite{Tcoerc} 
introduced the notion of T-coercivity as a generalization of the
inf-sup theory of Ladyzenskaja-Babuska-Brezzi.
H.-M. Nguyen formalized the notion of complementary media in \cite{nguyen} (see also the ideas exposed in \cite{MiltonNicorovici})
to recast (\ref{eq.AD}) as a Cauchy
problem posed on the interfaces between the inclusions and
the background medium. In a subsequent series of articles (see the overview \cite{hmrev}), 
he used this construction to analyze the cloaking and superlensing properties of devices described by a system of the form (\ref{eq.AD}).

Notice that the work on plasmonic resonances mentionned above is closely related
(via the Neumann-Poincar\'e operator) to work on cloaking by anomalous
localized resonance~\cite{AmmariCiraolo,AmmariCiraoloII,bouchitteSchweizer} and to the problem of estimating the strength
of the gradient of the voltage potential between two close-to-touching smooth
inclusions embedded in a dielectric medium \cite{AmmariKangLim,bontrikigrad}.
In this spirit, in \cite{bontriki,bontrikitsou}, the first and third authors have derived the asymptotic expansion of the 
plasmonic resonances associated to two close to touching inclusions of size one.  \\

In this article, we aim to study how resonant metallic particles may interact
and act collectively: for a given (small) size $\varepsilon >0$, the set $D$ of inclusions is an $\varepsilon$-periodic
collection $\omega_\varepsilon$ 
of smooth metallic inclusions distributed in a bounded domain 
$\Omega \subset \mathbb{R}^d$ (see Section \ref{sec.setting} for precise definitions). 
The voltage potential $u_\varepsilon$ is the solution to:
\begin{equation}\label{eq.Aepsintro}
\left\{ \begin{array}{cl}
-\textrm{div}(A_\varepsilon \nabla u_\varepsilon) = f &\text{in } \Omega,\\
u_\varepsilon = 0 & \text{on } \partial \Omega
\end{array} \right. , \text{ where } A_\varepsilon (x) := \left\{Ê
\begin{array}{cl}
1 & \text{if }Êx \notin \omega_\varepsilon,\\
a & \text{if } x \in \omega_\varepsilon.
\end{array}
\right.  
\end{equation}

Of particular interest is the limiting behavior of this system and of the corresponding plasmonic eigenvalues as the size $\varepsilon$ 
of the particles goes to $0$ (and their number tends to infinity). 
As we shall see, this question is closely connected to the construction of so-called \textit{hyperbolic metamaterials}, 
that is, anisotropic materials whose permittivity tensor has sign-changing eigenvalues. Such materials have raised a great
enthousiasm among the physics community, since they potentially allow e.g. for subwavelength
imaging (since the 
wavelength of the propagating waves allowed by their dispersion relation may be arbitrarily short), as was theoretically investigated in \cite{bonhm}, or for calculating density of states in solid-state and condensed matter physics; see \cite{poddubny,shekhar} for 
an overview and references. Interestingly, some authors have started to address the design
of hyperbolic metamaterials by topology optimization techniques \cite{otomori}. \\

Let us now outline how we intend to study the well-posedness of (\ref{eq.Aepsintro}).  
As we have mentioned, one way relies on potential theory; see for instance \cite{AmmariKang,ando,kangln}. 
Roughly speaking, the values of $a$ for which (\ref{eq.potgen}) is ill-posed are related to the eigenvalues of the Neumann-Poincar\'e operator 
${\mathcal K}^*_{\omega_\varepsilon} : L^2(\partial \omega_\varepsilon) \rightarrow L^2(\partial \omega_\varepsilon)$ (see (\ref{eq.defK}) hereafter). 
The nice feature of this point of view 
is that it decouples the geometry of the set $\omega_\varepsilon$ of inclusions from
the value of the conductivity inside. 
The drawback is that it makes it difficult to account for complex variations of the geometry of the set of inclusions, 
which is typically what we are interested in. 

To circumvent this difficulty, we take advantage of the enlighting work of D. Khavinson, M. Putinar and H.S. Shapiro~\cite{khavinson},
who explained H.~Poincar\'e's and M.~Krein's results about the Neumann-Poincar\'e operator in the context of modern functional analysis.
As in \cite{bontriki}, we tackle the system (\ref{eq.Aepsintro}) at the level of the so-called Poincar\'e variational problem, 
which brings into play the operator $T_\varepsilon: H^1_0(\Omega) \rightarrow H^1_0(\Omega)$ defined as follows: 
for $u \in H^1_0(\Omega)$, $T_\varepsilon u$ is the unique element in $H^1_0(\Omega)$ such that:
\begin{equation}\label{eq.Tepsintro}
\forall\; v \in H^1_0(\Omega), \quad
\int_\Omega{ \nabla (T_\varepsilon u) \cdot \nabla v \:dx} =  \int_{\omega_\varepsilon} {\nabla u \cdot \nabla v \:dx}.
\end{equation}
As we recall in Section \ref{sec.TD},
it turns out that the spectrum of the Neumann-Poincar\'e operator and 
the spectrum $\sigma(T_\varepsilon)$ of the Poincar\'e variational operator are 
explicitely related to one another.

Our aim is to study the behavior of the sets $\sigma(T_\varepsilon)$ 
as the period $\varepsilon$ of the distribution of inclusions tends to $0$, 
and more precisely the structure of the \textit{limit spectrum} $\lim_{\varepsilon \rightarrow 0}{\sigma(T_\varepsilon)}$, 
defined as the set of accumulation points of the $\sigma(T_\varepsilon)$: 
\begin{equation}\label{def.limspec}
\lim\limits_{\varepsilon \rightarrow 0}{\sigma(T_\varepsilon)} := \left\{ \lambda \in [0,1], \text{ s.t. } \exists \: \varepsilon_j \downarrow 0,\:\: \lambda_{\varepsilon_j} \in \sigma(T_{\varepsilon_j}), \:\: \lambda_{\varepsilon_j}\xrightarrow{ j\rightarrow \infty} \lambda  \right\}. 
\end{equation}
Regarding plasmonic resonances, 
this question is related to understanding whether collective effects of 
resonant particles may enhance or annihilate the concentration of large 
gradients in $\Omega$.

Another question we address in this article concerns the relation between 
the asymptotic behavior of the spectrum $\sigma(T_\varepsilon)$ and the well-posedness 
of the formally homogenized limit system associated to (\ref{eq.Aepsintro}).\\

Our work is organized as follows: In Section \ref{sec.setting},
we set notations and describe the geometry of the periodic distribution of inclusions. 

Section \ref{sec.TD} is concerned with a discussion about the main features of the Poincar\'e variational problem, 
in the case of a general inclusion (or collection
of inclusions) $D$ embedded in a macroscopic domain $\Omega$. 
We recall some known facts about the Poincar\'e variational operator $T_D$, in particular 
its connection with the Poincar\'e-Neumann operator of the inclusion $D$; 
we also present a construction which we shall use repeatedly in the sequel, 
connecting $T_D$ with an operator acting on functions defined only inside the inclusion $D$. 

Starting from Section \ref{sec.limspec}, we assume that $D = \omega_\varepsilon$ is a periodic collection
of inclusions, and in addition, that the considered inhomogeneities are smooth and
strictly included in the periodicity cells.
This section is devoted to the study of the spectral properties of the Poincar\'e variational problem for $\omega_\varepsilon$, 
and of how they behave as $\varepsilon$ tends to $0$. 
A summary of the main results obtained in this direction is presented in Section \ref{sec.mainresspec}; 
in a nutshell, we first prove in Section \ref{sec.boundev}, by 
using the min-max principle associated to the
Poincar\'e variational problem, that the non-degenerate spectra of the 
Poincar\'e variational operators $T_\varepsilon$ are uniformly contained in 
an interval $[m,M] \subset (0,1)$.
Then, in order to cope with the `bad' convergence properties of $T_\varepsilon$, 
we carry out a two-scale reformulation of the problem, and construct a two-scale limit operator
(in the sense of pointwise convergence) the spectrum of which
contains that of $T_0$, the Poincar\'e operator defined
on the periodicity cell.
Carrying out this rescaling over larger and larger blocks
of inhomogeneities, in the spirit of 
 the work of G. Allaire and C. Conca~\cite{allaireconcafs,allaireconca}, 
we then show that the limit spectrum $\lim_{\varepsilon \to 0} \sigma(T_\varepsilon)$ can be
decomposed as $\sigma_{\partial \Omega} \cup \sigma_{\textrm{Bloch}}$,
where the `Bloch spectrum' $\sigma_{\textrm{Bloch}}$ accounts for the collective resonances of 
groups of inclusions as $\varepsilon \to 0$, and the `boundary layer spectrum'
$\sigma_{\partial \Omega}$,
consists of limits of sequences of eigenvalues associated to eigenfunctions
that spend a not too small part of their energy near the boundary.

In Section \ref{secdirect}, we investigate the properties of the
boundary value problem (\ref{eq.Aepsintro}) associated to a conductivity $a \in \mathbb{C}$ filling the periodic distribution of inclusions. 
We first study in Section \ref{sec.cellpb} the properties (well-posedness, positive definiteness,...) of the tensor $A^*$, 
formally obtained by applying the usual periodic homogenization formulae.
In Section \ref{sec.gendirect}, we generalize F. Murat and L. Tartar's 
compactness theorem for homogenization to the case of elliptic equations with possibly non positive coefficients 
inside the inhomogeneities. In particular, we explain that the homogenized tensor $A^*$ partially encodes the behavior of the voltage potential $u_\varepsilon$ as $\varepsilon \to 0$. 
In Section \ref{sec.hommatlimspec}, we use these results to partially characterize the limit spectrum $\lim_{\varepsilon \to 0} \sigma(T_\varepsilon)$
in terms of $A^*$. Sections \ref{sec.hc} and \ref{sec.unifhomog} focus on the `high-contrast' source problem, 
i.e. when $a \to \pm \infty$.
Using results from Section \ref{sec.limspec}, we show that this problem is always well-posed, and 
converges to its homogenized limit as $\varepsilon \to 0$ uniformly with respect to the value of $a$. 
Finally, in the Appendix, we perform the explicit calculation of the limit spectrum in the particular setting
of rank $1$ laminates.
This situation is an interesting source of examples and counter examples 
to the general features exposed in the main parts of this work. 

%%%%%%%%%%%%%%%%%%%%%%%%%%%%%%%%%%%%%%%%%%%%%%%%%%%%%%%
\section{Description of the setting and notations}\label{sec.setting}
%%%%%%%%%%%%%%%%%%%%%%%%%%%%%%%%%%%%%%%%%%%%%%%%%%%%%%%

Let $\Omega \subset \mathbb{R}^d$ and $D \Subset \Omega$ be bounded, Lipschitz regular domains.
We denote by $u^+$ (resp. $u^-$) the restriction to $\Omega \setminus \overline{D}$ (resp. $D$) of a function $u:\Omega \to \mathbb{R}$. 
If $u^+$ and $u^-$ have traces $u^+\lvert_{\partial D}$ and $u^-\lvert_{\partial D}$ on $\partial D$, we denote by $[u] := u^+\lvert_{\partial D} - u^-\lvert_{\partial D}$ the \textit{jump} of $u$ across $\partial D$.
Introducing the unit normal vector $n$ to $\partial D$ pointing outward $D$, we denote by
$$ \frac{\partial u^\pm}{\partial n}(x) = \lim\limits_{t \rightarrow 0^+}{\nabla u (x \pm tn(x)) \cdot n(x)}$$
the \textit{exterior} and \textit{interior normal derivatives} of $u$ at $x \in \partial D$, when $u$ is regular enough. 
The associated \textit{normal jump} across $\partial D$ is denoted by $\left[\frac{\partial u}{\partial n}\right] = \frac{\partial u^+}{\partial n} - \frac{\partial u^-}{\partial n}$. 

As far as the geometry of the set $D$ is concerned, the main part of this article is devoted to the particular case
where it is a periodic distribution of many `small' identical inclusions of size $\varepsilon >0$. 

To be more precise, let us denote by 
$Y = (0,1)^d$ the unit periodicity cell of $\mathbb{R}^d$, and by $\omega \Subset Y$ the 
region occupied by the rescaled inclusion. 
For the sake of simplicity, we assume that both $\omega$ and $Y \setminus \overline{\omega}$ are connected. 
In due time, we shall put additional assumptions as far as the regularity of $\omega$ is concerned, 
but for the moment, we only require that it is Lipschitz regular. 
For $\varepsilon >0$, for any point $x \in \mathbb{R}^d$, 
there exists a unique pair $(\xi,y) \in \mathbb{Z}^d \times [0,1)^d$ such that $x= \varepsilon \xi + \varepsilon y$, which we denote
$$ \xi = \left[ \frac{x}{\varepsilon}\right]_Y , \text{ and } y = \left\{Ê\frac{x}{\varepsilon}\right\} _Y.$$ 
For given $\xi \in \mathbb{Z}^d$, we denote by $Y_\varepsilon^\xi := \varepsilon(\xi + Y)$ (resp $\omega_\varepsilon^\xi := \varepsilon(\xi + \omega)$) 
the copy of $Y$ (resp. $\omega$) with actual size $\varepsilon$ and position $\xi$. 
Let ${\mathcal O}_\varepsilon$ be the open reunion of all $\varepsilon$-cells which are completely included in $\Omega$, that is, ${\mathcal O}_\varepsilon$ is the interior of
\begin{equation}\label{eq.defOeps}
\overline{{\mathcal O}_\varepsilon}= \bigcup_{\xi \in \Xi_\varepsilon}{\overline{\varepsilon(\xi + Y)}}, \text{ where } \Xi_\varepsilon = \left\{ \xi \in \mathbb{Z}^d, \: Y_\varepsilon^\xi \Subset \Omega \right\},
\end{equation}
and let ${\mathcal B}_\varepsilon = \Omega \setminus \overline{{\mathcal O}_\varepsilon}$ be the complementary region in $\Omega$.
The subset $\omega_\varepsilon \Subset \Omega$ occupied by the inclusions is therefore defined by: 
\begin{equation}\label{eq.omeps}
 \omega_\varepsilon = \bigcup_{ \xi \in \Xi_\varepsilon}{\varepsilon(\xi + \omega)};
 \end{equation}
see Figure \ref{fighomog} for an illustration. 

\begin{figure}[!ht]
\centering
\includegraphics[width=0.75 \textwidth]{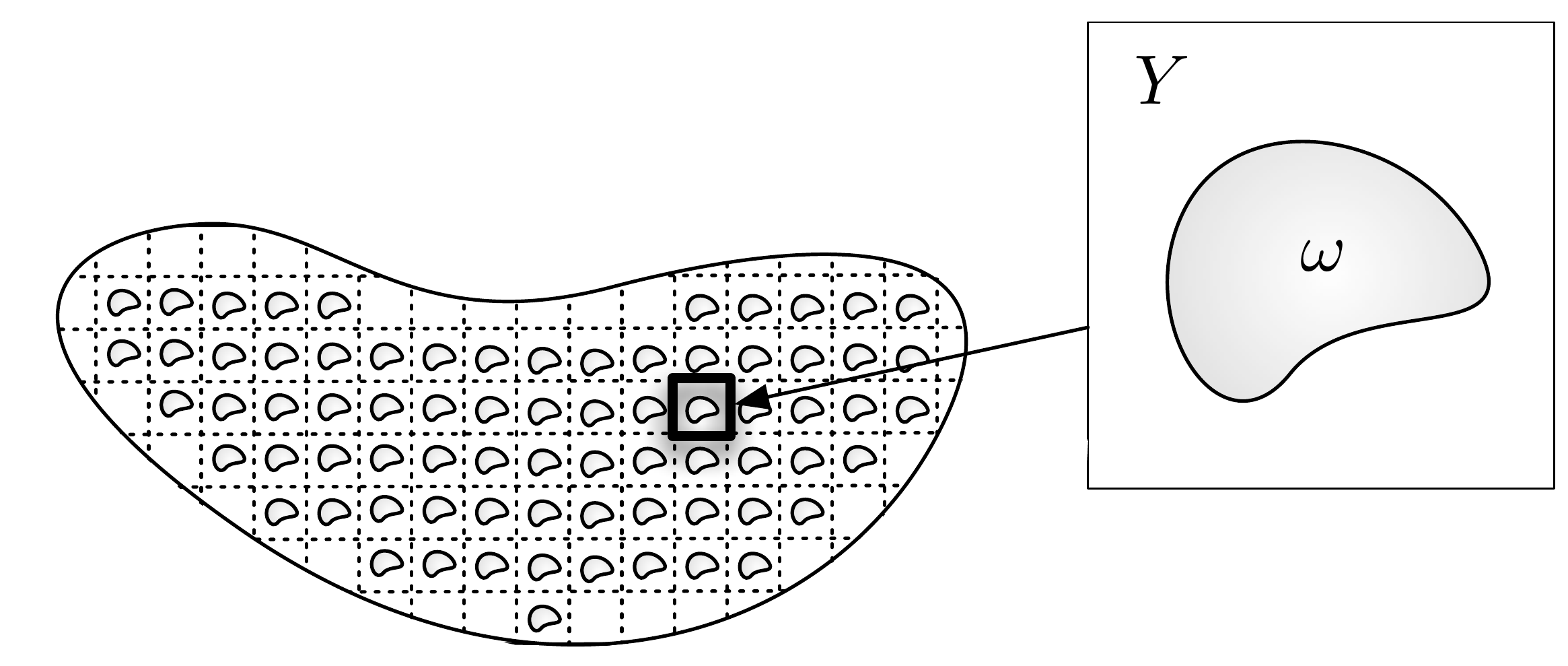}
\caption{\textit{Setting of the homogenization problem.}}
\label{fighomog}
\end{figure}

We consider a situation where the external medium to $\Omega$ is kept at constant, null potential; 
the matrix $\Omega \setminus \overline{\omega_\varepsilon}$ and the set of inclusions $\omega_\varepsilon$ 
are respectively filled with materials with constant conductivities $1$ and $a\in \mathbb{C}\setminus \left\{ 0 \right\}$. 
The main purpose of this article is to investigate the properties of the voltage potential $u \in H^1_0(\Omega)$ generated by a source $f \in H^{-1}(\Omega)$ 
acting in the medium, which is solution to the system: 
\begin{equation}\label{eq.potgen}
\left\{ 
\begin{array}{cl}
-\text{\rm div}(A_\varepsilon(x) \nabla u) = f &\text{in } \Omega, \\ 
u = 0 & \text{on } \partial \Omega
\end{array}
\right., \text{ where } A_\varepsilon(x) = \left\{
\begin{array}{cl}
1 & \text{if }Êx \in \Omega \setminus \overline{\omega_\varepsilon}, \\
a & \text{if }Êx \in \omega_\varepsilon.
\end{array}
\right. 
\end{equation}

As we recall in Section \ref{sec.defTD}, the system (\ref{eq.potgen}) does not fall into the classical setting of the Lax-Milgram theory 
when the conductivity $a$ inside the inclusions is negative, and there actually 
exists 
such values of $a$, called \textit{plasmonic eigenvalues} -
depending on the geometries of $\Omega$ and of $\omega_\varepsilon$
(thus on $\varepsilon$) -, 
for which (\ref{eq.potgen}) has non trivial solutions when $f \equiv 0$. 

Our goal is to understand the behavior of these plasmonic eigenvalues in the limit $\varepsilon \rightarrow 0$, 
when the set of inclusions is `homogenized' inside $\Omega$. 

%%%%%%%%%%%%%%%%%%%%%%%%%%%%%%%%%%%%%%%%%%%%%%%%%%%%%%%
\section{Generalities about the Poincar\'e variational problem}\label{sec.TD}
%%%%%%%%%%%%%%%%%%%%%%%%%%%%%%%%%%%%%%%%%%%%%%%%%%%%%%%

In this section, we discuss some general properties of the \textit{Poincar\'e variational problem}, in connection 
with the Neumann-Poincar\'e operator of a general Lipschitz regular subdomain $D \Subset \Omega$.
Since we chiefly intend to use these concepts in the setting of Section \ref{sec.setting}, we may think of $D$ as consisting of 
several connected components, 
that we denote $D_1,...,D_N$. We also assume throughout this section that $\Omega \setminus \overline{D}$ is connected.

%%%%%%%%%%%%%%%%%%%%%%%%%%%%%%%%%%%%%%%%%%
\subsection{Definition of the operator $T_D$, and connection with the conductivity equation}\label{sec.defTD}~\\
%%%%%%%%%%%%%%%%%%%%%%%%%%%%%%%%%%%%%%%%%%

Following \cite{khavinson}, we first define the operator $T_D : H^1_0(\Omega) \rightarrow H^1_0(\Omega)$ featured in the Poincar\'e variational problem. 
Throughout this article, the space $H^1_0(\Omega)$ is equipped with the inner product and associated norm: 
\begin{equation}\label{eq.H10prod}
 \langle u , v \rangle_{H^1_0(\Omega)} = \int_{\Omega}{\nabla u \cdot \nabla v \:dx},\text{ and } \lvert \lvert u \lvert\lvert_{H^1_0(\Omega)}^2 = \langle u, u \rangle_{H^1_0(\Omega)}.
 \end{equation}
For an arbitrary function $u \in H^1_0(\Omega)$, $T_D u$ is the solution in $H^1_0(\Omega)$ to the following variational problem: 
\begin{equation}\label{eq.defTD}
 \forall v \in H^1_0(\Omega), \:\: \int_\Omega{\nabla (T_D u ) \cdot \nabla v \:dx} = \int_{D}{\nabla u \cdot \nabla v \:dx}.
 \end{equation}
The existence and unicity of $T_D u$ are direct consequences of the Riesz representation theorem. 

Roughly speaking, $T_D u $ accounts for the fraction of the energy of $u$ which is contained 
in $D$. The following properties of $T_D$ immediately stem from the definition (\ref{eq.defTD}): 

\begin{proposition}
$T_D$ is self-adjoint and bounded, with operator norm $\lvert\lvert T_D \lvert\lvert = 1$.
 It is also non negative: 
 $$ \forall u \in H^1_0(\Omega), \:\: \langle T_D u , u \rangle_{H^1_0(\Omega)} = \int_{\Omega}{\nabla (T_D u ) \cdot \nabla u \:dx} = \int_D{\lvert \nabla u \lvert^2\:dx}.$$
\end{proposition}\par
\smallskip

This operator naturally arises in the study of the conductivity equation: 
assume that the inclusion $D$ is filled with a material with conductivity $a$, 
and the complementary part $\Omega \setminus \overline{D}$ is occupied by a material with unit conductivity, so that
the potential $u \in H^1_0(\Omega)$ generated by a source $f \in H^{-1}(\Omega)$ is solution to the equation: 
\begin{equation}\label{eq.conduc}
\left\{ 
\begin{array}{cl}
-\text{\rm div}(A(x)\nabla u) = f & \text{in } \Omega, \\
u = 0 & \text{\rm on } \partial \Omega
\end{array}
\right., \text{ where } A(x) = \left\{ \begin{array}{cl}
1 & \text{if } x \in \Omega \setminus \overline{D},\\
a & \text{if } x \in D
\end{array}
\right. ,
\end{equation}
or under variational form: 
\begin{equation}\label{eq.varpot}
 \forall v \in H^1_0(\Omega), \:\: \int_\Omega{A(x)\nabla u \cdot \nabla v \:dx} = \langle f , v \rangle_{H^{-1}(\Omega), H^1_0(\Omega)},
 \end{equation}
 where $\langle \cdot, \cdot \rangle_{H^{-1}(\Omega),H^1_0(\Omega)}$ stands for the usual duality pairing between $H^{-1}(\Omega)$ and $H^1_0(\Omega)$. 
 
Obviously, when either $a\in \mathbb{R}$, $a>0$, or $a \in \mathbb{C}$ with nonzero imaginary part, 
the well-posedness of (\ref{eq.conduc}) is a straightforward consequence of the Lax-Milgram lemma. 

In the general case, let $g \in H^1_0(\Omega)$ be the representative of $f$ supplied by the Riesz representation theorem: 
$$ \forall v \in H^1_0(\Omega), \:\: \int_{\Omega}{\nabla g \cdot \nabla v \:dx} =  \langle f , v \rangle_{H^{-1}(\Omega), H^1_0(\Omega)}.$$
Simple manipulations show that (\ref{eq.varpot}) is equivalent to: 
$$\left( \lambda I - T_D \right) u = \lambda g, \text{ with } \lambda = \frac{1}{1-a}.$$
Hence, much information on the solvability of (\ref{eq.conduc}) can be gleaned from the study of the eigenvalues of $T_D$.

%%%%%%%%%%%%%%%%%%%%%%%%%%%%%%%%%%%%%%%%%%
\subsection{Basic spectral properties of $T_D$}\label{sec.evTD}~\\
%%%%%%%%%%%%%%%%%%%%%%%%%%%%%%%%%%%%%%%%%%

One of the main stakes of this article is the study of the spectrum $\sigma(T_D)$; 
as we shall see soon, this problem essentially amounts to considering the eigenvalue problem for $T_D$: 
seek $\lambda \in \mathbb{R}$ and $u \in H^1_0(\Omega)$, $u \neq 0$, 
such that $T_D u = \lambda u$, i.e. in variational form: 
\begin{equation}\label{eq.vpnp}
 \forall v \in H^1_0(\Omega), \:\: \lambda \int_\Omega{\nabla u \cdot \nabla v \:dx} = \int_D{\nabla u \cdot \nabla v\:dx}. 
 \end{equation}
A first property of interest is the following:
\begin{proposition}\label{prop.evTD} \noindent
The spectrum $\sigma(T_D)$ of $T_D$ is included in the interval $\left[Ê0, 1 \right]$. Moreover,
\begin{enumerate}[(i)]
\item The eigenspace associated to the eigenvalue $\lambda = 0$ is 
$$ \text{\rm Ker}(T_D) = \left\{ u \in H^1_0(\Omega), \:  u  = \text{a constant } c_j \text{ in each connected component } D_j \text{ of } D, \: j=1,...,n \right\}. $$
\item The eigenspace associated to $\lambda = 1$ is 
$$ \text{\rm Ker}(I-T_D) = \left\{ u \in H^1_0(\Omega), \:  u  = 0 \text{ in } \Omega \setminus \overline{D} \right\},$$
which can naturally be
identified with $H^1_0(D)$.
\item The following orthogonal decomposition of $H^1_0(\Omega)$ holds: 
\begin{equation}\label{TDortho}
H^1_0(\Omega) = \text{\rm Ker}(T_D) \oplus \text{\rm Ker}(I-T_D) \oplus \mathfrak{h},
\end{equation}
where the closed subspace $\mathfrak{h}$ is defined as:
\begin{equation}\label{eq.mathfrakh}
 \mathfrak{h} = \left\{ u \in H^1_0(\Omega), \:\: \Delta u = 0 \text{ in } D \cup (\Omega \setminus \overline{D}), \:\: \int_{\partial D_j}{\frac{\partial u^+}{\partial n} \:ds} = 0 , \: j =1,...,N \right\}.
 \end{equation}
\end{enumerate}
\end{proposition}
\begin{proof}
The points $(i)$ and $(ii)$ are pretty straightforward from (\ref{eq.vpnp}), and we focus on proving $(iii)$. 
First, $\text{\rm Ker}(T_D)$ and $\text{\rm Ker}(I-T_D)$ are clearly closed, orthogonal subspaces of $H^1_0(\Omega)$. 
Then, a function $u \in H^1_0(\Omega)$ is orthogonal to $\text{\rm Ker}(T_D)$ if and only if
$$ \forall v \in \text{\rm Ker}(T_D), \:\: \int_{\Omega \setminus \overline{D}}{\nabla u \cdot \nabla v \:dx} = 0; $$
using Green's formula, taking at first test functions $v \in {\mathcal C}^\infty_c(\Omega \setminus \overline{D})$, then 
arbitrary $v \in \text{\rm Ker}(T_D)$, this condition turns out to be equivalent to: 
$$ \Delta u = 0 \text{ on } \Omega \setminus \overline{D}, \text{ and } \int_{\partial D_j}{\frac{\partial u^+}{\partial n} \:ds } = 0 \text{ on each connected component } \partial D_j \text{ of } \partial D.$$
Likewise, one proves that a function $u \in H^1_0(\Omega)$ is orthogonal to $\text{\rm Ker}(I-T_D)$ if and only if $\Delta u =0$ on $D$. 
The desired result follows. 
\end{proof} \par
\smallskip 

\subsection{Connection between $T_D$ and the Neumann-Poincar\'e operator}\label{sec.NPTD}~\\

The material collected in this section is not new in the literature; it appears e.g. in \cite{bontriki,khavinson} (see also the references quoted throughout this section) 
in the context where the inclusion $D$ is embedded in the whole space $\mathbb{R}^d$, and not in a bounded domain $\Omega$. 
Since we intend to use several slightly different versions of this framework, 
we sketch the main points here in the context of Section \ref{sec.defTD}; these may be easily adapted to the other situations we shall consider. 

In all this subsection, we assume $D$ to be of class ${\mathcal C}^2$ (but $\Omega$ may still only assumed to be Lipschitz regular). 

Let $P(x,y)$ be the Poisson Kernel associated to the domain $\Omega$, i.e. for any $x \in \Omega$, $P(x,\cdot)$ is the unique solution of: 
\begin{equation}\label{eq.poisk}
 \left\{ 
\begin{array}{cl}
\Delta_y P(x,y) = \delta_x & \text{for } y \in \Omega, \\
P(x,y) = 0 & \text{for } y \in \partial \Omega,
\end{array}
\right.
\end{equation}
where $\delta_x$ stands for the Dirac distribution centered at $x$. 
Recall that $P(x,y)$ is symmetric, that is $P(x,y) = P(y,x)$ for $x\neq y \in \Omega$, and has the structure: 
$$ P(x,y) = G(x,y) + R_x(y),$$
where $G(x,y)$ is the free space Green function: 
$$ G(x,y) = \left\{ 
\begin{array}{cl}
\frac{1}{2\pi}\text{\rm log} \lvert x - y \lvert  & \text{if } d = 2, \\
\frac{\lvert x -y \lvert^{d-2} }{(2-d)\omega_d}& \text{if }Êd \geq 3. 
\end{array}
\right., \:\: \omega_d = \text{ area of the unit sphere in } \mathbb{R}^d, $$
and where, for given $x\notin \partial \Omega$, $R_x(y)$ is the (smooth) solution of: 
$$ \left\{ 
\begin{array}{cl}
\Delta_yR_x(y) = 0 & \text{for } y \in \Omega, \\
R_x(y) = -G(x,y) & \text{for } y \in \partial \Omega;
\end{array} 
\right. $$
 see for instance \cite{AmmariKang,folland}. 

In this context, the single and double layer potentials ${\mathcal S}_D \phi, \:{\mathcal D}_D \phi \in L^2(\partial D)$ 
associated to a density $\phi \in L^2(\partial D)$ are respectively defined by: 
$$\forall x \in D \cup (\Omega \setminus \overline{D}), \:\:  {\mathcal S}_D \phi(x) = \int_{\partial D}{P(x,y)\phi(y) \:ds(y)}, \text{ and } {\mathcal D}_D \phi(x) = \int_{\partial D}{\frac{\partial P}{\partial n_y}(x,y)\phi(y) \:ds(y)}.$$
Both ${\mathcal S}_D\phi$ and ${\mathcal D}_D \phi$ are harmonic functions in $D$ and in $\Omega \setminus \overline{D}$,
and vanish on $\partial \Omega$. Moreover, their behavior in the neighborhood of the boundary $\partial D$ is described by the Plemelj jump relations: 
\begin{equation}\label{eq.jumpSD}
 {\mathcal S}_D \phi ^+(x) = {\mathcal S}_D \phi^-(x), \text{ and } \frac{\partial {\mathcal S}_D\phi ^\pm}{\partial n}(x)= \pm \frac{1}{2}\phi(x) + {\mathcal K}_D^* \phi(x), 
 \end{equation}
\begin{equation}\label{eq.jumpDD}
\frac{\partial {\mathcal D}_D \phi ^+}{\partial n}(x) = \frac{\partial {\mathcal D}_D \phi^-}{\partial n}(x), \text{ and } {\mathcal D}_D\phi ^\pm (x)= \mp \frac{1}{2}\phi(x) + {\mathcal K}_D \phi(x), 
\end{equation}
where the Neumann-Poincar\'e operators ${\mathcal K}_D, \: {\mathcal K}_D^*: L^2(\partial D) \rightarrow L^2(\partial D)$ 
are the weakly singular integral operators, adjoint from one another, defined by: 
\begin{equation}\label{eq.defK}
 {\mathcal K}_D \phi (x) = \int_{\partial D}{\frac{\partial P}{\partial n_y}(x,y)\phi(y)\:ds(y)},  \text{ and } {\mathcal K}_D^* \phi (x) = \int_{\partial D}{\frac{\partial P}{\partial n_x}(x,y)\phi(y)\:ds(y)}.
 \end{equation}
Let us recall facts from
potential theory (see e.g. \cite{AmmariKang,mclean,sauterschwab}): 
\begin{itemize}
\item The single layer potential extends into an operator $H^{-1/2}(\partial D) \to H^1_0(\Omega)$, which we still denote by ${\mathcal S}_D$. 
For any potential $\phi \in H^{-1/2}(\partial D)$, ${\mathcal S}_D\phi$ is the unique $u \in H^1_0(\Omega)$ such that:
\begin{equation}\label{eq.varfSD}
 \forall v \in H^1_0(\Omega), \:\: \int_{\Omega}{\nabla u \cdot \nabla v \:dx}  = \langle \phi , v \rangle_{H^{-1/2}(\partial D), H^{1/2}(\partial D)}.
 \end{equation}
\item Owing to (\ref{eq.jumpSD}), the single layer potential induces a bounded operator $S_D: L^2(\partial D) \rightarrow L^2(\partial D)$, 
defined by $S_D \phi = {\mathcal S}_D \phi \lvert_{\partial D}$. Moreover, this operator naturally extends as an operator $H^{-1/2}(\partial D) \rightarrow H^{1/2}(\partial D)$.
\item The operator ${\mathcal K}_D$ extends as an operator $H^{1/2}(\partial D) \to H^{1/2}(\partial D)$, 
and ${\mathcal K}_D^*$ extends as an operator $H^{-1/2}(\partial D) \to H^{-1/2}(\partial D)$.
\item Since $D$ is of class ${\mathcal C}^2$, both operators ${\mathcal K}_D$ and ${\mathcal K}_D^*$ are compact. 
\end{itemize}

Of particular interest for our purpose is the space $\mathfrak{h}_{\mathcal S}$ of \textit{single layer potentials}: 
$$ \mathfrak{h}_{\mathcal S} = \left\{{\mathcal S}_D \phi , \:\: \phi \in H^{-1/2}(\partial D)\right\} = \left\{ u \in H^1_0(\Omega), \:\: \Delta u = 0 \text{ in } D \text{ and } \Omega \setminus \overline{D} \right\}. $$
Note that $\mathfrak{h}_{\mathcal S}$ is a 
closed
subspace of $H^1_0(\Omega)$ and in
light of Proposition \ref{prop.evTD}, the following facts hold: 
\begin{itemize}
\item $T_D$ maps ${\mathfrak h}_S$ into itself. 
\item ${\mathfrak h}_S$ contains ${\mathfrak h}$ and is orthogonal to $\text{\rm Ker}(I-T_D)$; hence: 
$$ {\mathfrak h} \subset {\mathfrak h}_S \subset {\mathfrak h} \oplus \text{\rm Ker}(T_D).$$
More precisely, it holds: 
$$ {\mathfrak h}_S = {\mathfrak h}Ê\oplus \text{\rm span}\left\{ h_j \right\}_{j=1,...,N},$$
where, for $j=1,...,N$, $h_j \in H^1_0(\Omega)$ is the unique function such that: 
$$ h_j = \left\{ \begin{array}{cl} 
1 & \text{ on } D_j, \\
0 &\text{ on } D_k,\: k \neq j 
\end{array} \right. , \text{ and }\Delta h_j = 0 
\text{ in } 
\Omega \setminus \overline{D}.$$
\item The spectrum of $T_D: H^1_0(\Omega) \rightarrow H^1_0(\Omega)$ is the reunion of the two eigenvalues $0$ and $1$ (whose eigenspaces are described in Proposition \ref{prop.evTD}) with the spectrum of the restriction $T_D : \mathfrak{h}_{\mathcal S} \rightarrow \mathfrak{h}_{\mathcal S}$. 
\end{itemize}

Hence, $T_D$ naturally induces an operator (still denoted as $T_D$) $ \mathfrak{h}_{\mathcal S} \rightarrow  \mathfrak{h}_{\mathcal S}$, 
which captures `almost' all the spectrum of $T_D: H^1_0(\Omega) \rightarrow H^1_0(\Omega)$. 

On the other hand, this operator is connected with the 
Neumann-Poincar\'e operator ${\mathcal K}^*_D$ of $D$, as accounted for by the following proposition.
 
\begin{proposition}\label{prop.TDcompact}
The operator $R_D := T_D - \frac{1}{2}I$, from $\mathfrak{h}_{\mathcal S}$ into itself, is compact, self-adjoint, and $\lvert\lvert R_D \lvert\lvert \leq \frac{1}{2}$.
The spectrum of $R_D$ consists in a discrete sequence of eigenvalues with $0$ as unique accumulation point. Moreover, $(\lambda, u) \in \mathbb{R} \times {\mathfrak h}_S$, $\lambda \notin \left\{0,1\right\}$, is an eigenpair for $R_D$ if and only if $(\lambda, S_D^{-1}(u\lvert_{\partial D})) \in \mathbb{R} \times H^{-1/2}(\partial D)$ is an eigenpair for ${\mathcal K}_D^*$. 
\end{proposition}
\begin{proof} 
An underlying statement in this proposition is that $S_D : H^{-1/2}(\partial D) \rightarrow H^{1/2}(\partial D)$ is an isomorphism, which we first verify.
We already know that this operator is Fredholm with index $0$ (see e.g. \cite{mclean}, Th. $7.6$); 
hence to show that is has a bounded inverse, it is enough to prove that it is injective, which follows in turn from the identity: 
$$
\begin{array}{>{\displaystyle}cc>{\displaystyle}l}
 \forall \phi \in H^{-1/2}(\partial D), \:\: \int_{\partial D}{\phi \: S_D \phi\:ds} &=& \int_{\partial D}{\frac{\partial {\mathcal S}_D \phi^+}{\partial n} S_D\phi \:ds} -  \int_{\partial D}{\frac{\partial {\mathcal S}_D \phi^-}{\partial n} S_D\phi \:ds},\\
 &=& -\int_{\Omega \setminus \partial D}{\lvert \nabla {\mathcal S}_D \phi \lvert^2 \:dx };
 \end{array}$$
 thus, if $S_D \phi = 0$ on $\partial D$, then ${\mathcal S}_D \phi$ is constant on $\Omega$, and thus identically vanishes in $\Omega$ 
since it satisfies ${\mathcal S}_D \phi = 0$ on $\partial \Omega$. Hence, $\phi = \frac{\partial {\mathcal S}_D \phi^+}{\partial n} - \frac{\partial {\mathcal S}_D \phi^-}{\partial n} = 0$. 

As a consequence of this mapping property, any function $u \in \mathfrak{h}_{\mathcal S}$ has the integral representation: 
\begin{equation}\label{eq.repuphi}
 u = {\mathcal S}_D \phi, \text{ where } \phi = S_D^{-1} (u\lvert_{\partial D}) \: \in H^{-1/2}(\partial D).
 \end{equation}

The rest of the proof closely follows that of Theorem 2 in \cite{bontriki}, and consists in relating $R_D$ to the compact operator ${\mathcal K}_D^*$.
For a given $u \in \mathfrak{h}_{\mathcal S}$, the function $R_Du$ satisfies the following variational identity: 
\begin{equation}\label{eq.varRD} 
\begin{array}{>{\displaystyle}cc>{\displaystyle}l}
\forall v \in H^1_0(\Omega), \:\: 2 \int_\Omega{\nabla (R_D u) \cdot \nabla v \:dx} &=& \int_D { \nabla u \cdot \nabla v \:dx} - \int_{\Omega \setminus \overline{D}}{\nabla u \cdot \nabla v \:dx}\\
&=& \int_{\partial D}{\left( \frac{\partial u^+}{\partial n} + \frac{\partial u^-}{\partial n}\right) v\:ds}.
\end{array}.
\end{equation}
 
Representing $u$ as in (\ref{eq.repuphi}) and using the Plemelj jump relations (\ref{eq.jumpSD}), we obtain: 
$$ \frac{1}{2} \left( \frac{\partial u^+}{\partial n} + \frac{\partial u^-}{\partial n}\right) =  {\mathcal K}_{D}^* \phi = ({\mathcal K}_D^* \circ S_D^{-1}) (u\lvert_{\partial D}). $$
Comparing with (\ref{eq.varRD}) and using (\ref{eq.varfSD}) now yields: 
$$2  R_D  u  =  {\mathcal S}_D \circ ({\mathcal K}_D^* \circ S_D^{-1}) (u\lvert_{\partial D}).$$
Since ${\mathcal K}_D^*$ is compact, it readily follows that so is $R_D$. All the remaining statements are easily inferred from this identity.
\end{proof}

It follows from Proposition (\ref{prop.TDcompact}) that $T_D  : {\mathfrak h}_S \rightarrow {\mathfrak h}_S$ is a Fredholm operator with index $0$; its spectrum consists in a discrete sequence of eigenvalues with $\frac{1}{2}$ as unique accumulation point. 
Let us denote by $\left\{ \lambda_i^\pm\right\}_{i \geq 1}$ these eigenvalues, ordered in such a way that:
$$ 0 = \lambda_1^- \leq \lambda_2^- \leq ... \leq \frac{1}{2}, $$
and: 
$$ \frac{1}{2} \leq... \leq \lambda_2^+ \leq \lambda_1^+ <1.$$
The min-max principle for the compact, self-adjoint operator $R_D: {\mathfrak h}_S \rightarrow {\mathfrak h}_S$ immediately implies the 
following proposition (again, see \cite{bontriki,khavinson}): 

\begin{proposition}\label{propminmax}
Let $\left\{ w_i^\pm \right\}_{i \geq 1}$ be the set of eigenfunctions of $T_D : {\mathfrak h}_S \rightarrow {\mathfrak h}_S$,
associated to the eigenvalues $\left\{ \lambda_i^\pm\right\}_{i \geq 1}$. The following min-max formulae hold: 
$$\lambda_i^{-} = \min\limits_{u \in {\mathfrak h}_S \setminus \left\{ 0 \right\} \atop u \perp w_1^-,...,w_{i-1}^-}{\frac{ \displaystyle{\int_D{\lvert \nabla u \lvert^2 \:dx}}}{ \displaystyle{\int_\Omega{\lvert \nabla u \lvert^2 \:dx}}}} = \max\limits_{F_i \subset {\mathfrak h}_S \atop \text{\rm dim}(F_i) = i-1}{\min\limits_{u \in F_i^\perp \setminus \left\{ 0 \right\}}{\frac{ \displaystyle{\int_D{\lvert \nabla u \lvert^2 \:dx}}}{ \displaystyle{\int_\Omega{\lvert \nabla u \lvert^2 \:dx}}}}}, $$
and
$$\lambda_i^{+} = \max\limits_{u \in {\mathfrak h}_S \setminus \left\{ 0 \right\} \atop u \perp w_1^+,...,w_{i-1}^+}{\frac{ \displaystyle{\int_D{\lvert \nabla u \lvert^2 \:dx}}}{ \displaystyle{\int_\Omega{\lvert \nabla u \lvert^2 \:dx}}}} = \min\limits_{F_i \subset {\mathfrak h}_S \atop \text{\rm dim}(F_i) = i-1}{\max\limits_{u \in F_i^\perp \setminus \left\{ 0 \right\}}{\frac{ \displaystyle{\int_D{\lvert \nabla u \lvert^2 \:dx}}}{ \displaystyle{\int_\Omega{\lvert \nabla u \lvert^2 \:dx}}}}}.$$
\end{proposition}

\begin{remark}
It follows from this discussion that the spectrum $\sigma(T_D)$ of $T_D$ is only composed 
of eigenvalues, except possibly for the value $\frac{1}{2}$, which is the only element in the essential spectrum of $T_D$.
\end{remark}

\begin{remark}
in the study of the operator $T_D$, it will prove of interest to consider either $T_D : \mathfrak{h}_{\mathcal S} \rightarrow \mathfrak{h}_{\mathcal S}$, or alternatively 
 $T_D :H^1_0(\Omega) \rightarrow H^1_0(\Omega)$. The former point of view is interesting insofar as $T_D$ is then `close' to a compact operator, 
 whereas the latter allows to consider $T_D$ as an operator defined on a functional space which is independent of the inclusion $D$. 
 \end{remark}\par
 
%%%%%%%%%%%%%%%%%%%%%%%%%%%%%%
\subsection{Restriction of the operator $T_D$ to the set of inclusions}\label{sec.rest}~\\
%%%%%%%%%%%%%%%%%%%%%%%%%%%%%%

In this section, we discuss a general construction which will 
be useful on several occasions
in this article, in slightly different contexts. Our purpose is to define an operator $\mathring{T_D}$ which `resembles' $T_D$ - in particular, it retains 
its spectral properties - except for the fact that it acts on functions defined solely on the set 
of inclusions $D$. 

Let us first introduce some notations: 
\begin{itemize}
\item The harmonic extension operator $U_D: H^1(D) \rightarrow H^1_0(\Omega)$ maps $u \in H^1(D)$ into
the function $U_D u \in H^1_0(\Omega)$ defined by the properties: 
$$ (U_D u) \lvert_D = u \text{ in } D, \text{ and } -\Delta (U_D u) = 0 \text{ in } \Omega \setminus \overline{D}.$$
As a straighforward consequence of the usual energy estimates for the Laplace operator, there exists a constant $C$ (depending on both $D$ and $\Omega$) 
such that: 
\begin{equation}\label{eq.contUD}
 \forall u \in H^1(D), \:\: \lvert\lvert U_D u \lvert\lvert_{H^1_0(\Omega)} \leq C \lvert\lvert u \lvert\lvert_{H^1(D)}.
 \end{equation}
\item $C(D)$ is the subset of $H^1(D)$ composed of functions which are constant on each connected component of $D$: 
$$ C(D) = \left\{ u \in H^1(D), \:\: \exists\: c_j \in \mathbb{R}, \:u =  c_j \text{ in } D_j, \: j=1,...,N \right\}.$$
\item The quotient space 
$H_D := H^1(D) / C(D)$ is a Hilbert space when equipped with the inner product and norm: 
\begin{equation}\label{eq.innerprodHD}
 \langle u,v \rangle_{H_D} := \int_D{\nabla u \cdot \nabla v \:dx}, \text{ and } \lvert\lvert u \lvert\lvert_{H_D}^2 = \langle u , u \rangle_{H_D}. 
 \end{equation}
We denote by $\widetilde{u} \in H_D$ the equivalence class of a function $u \in H^1(D)$.  
\end{itemize}
We now rely on the following facts about $T_D$, which arise immediately from the definition (\ref{eq.defTD}):
\begin{equation}\label{eq.factsTD}
\begin{minipage}{0.92\textwidth}
\begin{enumerate}[(i)]
\item $T_D : H^1_0(\Omega) \to H^1_0(\Omega)$ induces a bounded operator, still denoted by $T_D$ for simplicity, from $H_D $ into $H^1_0(\Omega)$ by the relation: 
$$ \forall \widetilde{u} \in H_D, \:\: T_D \widetilde{u} = T_D u, \text{ for any } u \in H^1_0(\Omega) \text{ such that } u\lvert_D \text{ belongs to the class } \widetilde{u}.$$ 
\item For any function $u \in H^1_0(\Omega)$, one has: $U_D (T_D u)\lvert_D = T_D\left(U_D(u\lvert_D)\right) =T_D u$. 
\end{enumerate}
\end{minipage}
\end{equation}
We are now in position to define our operator $\mathring{T_D}: H_D \rightarrow H_D$ by: 
$$ \forall \widetilde{u} \in H_D, \:\: \mathring{T_D} \widetilde{u}Ê = \widetilde{(T_D \widetilde{u} )\lvert_D}.$$ 

The first interesting property of $\mathring{T_D}$ is the following: 

\begin{proposition}\label{prop.TDring}
The operator $\mathring{T_D}$ is a self-adjoint isomorphism from $H_D$ into itself. 
\end{proposition} 
\begin{proof}
That $\mathring{T_D}$ is self-adjoint follows from the chain of equalities: 
$$ \forall u,v \in H^1_0(\Omega), \:\: \int_D{\nabla (T_D u) \cdot \nabla v \:dx} = \int_\Omega{\nabla (T_D u) \cdot \nabla (T_D v) \:dx} = \int_D{\nabla u \cdot \nabla (T_D v) \:dx}. $$

It is also fairly easy to prove that $\mathring{T_D}$ is injective. Indeed, let $u \in H^1_0(\Omega)$ be such that $(T_D u) \lvert_D  \in C(D)$; 
we have to prove that $u \in C(D)$. 
Obviously, $T_Du \in \text{\rm Ran}(T_D)$, but, using Proposition \ref{prop.evTD}, $T_D u \in \text{\rm Ker}(T_D)$, which is also $\text{\rm Ran}(T_D)^\perp$, since $T_D$ is self-adjoint. Therefore, $T_D u = 0$, and $u \in \text{\rm Ker}(T_D)$, that is, $u\lvert_D \in C(D)$, which is the desired conclusion.

The surjectivity of $\mathring{T_D}$ is a little more involved; it is implied by the following statement: 
\begin{equation}\label{eq.stTDsurj}
\forall v \in H^1_0(\Omega), \: \exists v_0 \in \text{\rm Ker}(T_D)  \text{ and } u \in H^1_0(\Omega) \text{ s.t. } T_D u = v-v_0.
\end{equation} 
Since $T_D$ is self-adjoint, one has $\text{\rm Ker}(T_D)^\perp = \overline{\text{\rm Ran}(T_D)}$. 
Hence, (\ref{eq.stTDsurj}) is proved provided we show that $T_D$ has closed range, which we now do. 
Let $v \in H^1_0(\Omega)$ and a sequence $u_n \in H^1_0(\Omega)$ be such that: 
$$ T_D u_n \stackrel{n \rightarrow \infty}{\longrightarrow} v \text{ strongly in } H^1_0(\Omega).$$
Then, for any $n,m \in \mathbb{N}$, and an arbitrary function $w \in C(D)$, it holds: 
$$ 
\begin{array}{>{\displaystyle}cc>{\displaystyle}l}
\int_{D}{\lvert \nabla u_n - \nabla u_m - \nabla w \lvert^2 \:dx} &=&  \int_{\Omega}{\nabla(T_D u_n - T_D u_m) \cdot \nabla(U_D u_n - U_D u_m - U_D w) \:dx},\\
&\leq & C \lvert\lvert T_D u_n - T_D u_m \lvert\lvert_{H^1_0(\Omega)} \lvert\lvert u_n -  u_m  - w\lvert\lvert_{H^1(D)},
\end{array}
$$  
where we have used (\ref{eq.contUD}) and (\ref{eq.factsTD}).
Now choosing 
$$w = \frac{1}{\lvert D_j \lvert }\int_{D_j}{(u_n -u_m) \:dx} \text{ in } D_j , \:\: j=1,...,N,$$ 
and using the Poincar\'e-Wirtinger inequality, we obtain: 
$$ \lvert\lvert \nabla u_n - \nabla u_m \lvert\lvert_{L^2(D)^d} \leq C \lvert\lvert T_D u_n - T_D u_m \lvert\lvert_{H^1_0(\Omega)},$$
which proves that  $\widetilde{u_n\lvert_D}$ is a Cauchy sequence in $H_D$. It thus converges to some element $\widetilde{u} \in H_D$. 
Let $u \in H^1_0(\Omega)$ be such that $u\lvert_D$ belongs to the equivalence class $\widetilde{u}$. 
By the continuity of $T_D$ (and more exactly using (\ref{eq.factsTD}), $(i)$), it follows that $T_D u_n = T_D ( \widetilde{u_n\lvert_D}) \rightarrow T_D \widetilde{u} = T_Du$, 
which ends proving that $T_D$ has closed range, and thus Proposition \ref{prop.TDring}.
\end{proof}

\begin{proposition}\label{propspecrescalTD}
The spectrum of the operator $\mathring{T_D}$ is $\sigma(\mathring{T_D}) = \sigma(T_D) \setminus \left\{0 \right\}$.
Moreover, $\lambda \neq 0$ is an eigenvalue of $T_D$ if and only if it is an eigenvalue of $\mathring{T_D}$.  
\end{proposition} 
\begin{proof}
Let $\lambda \in \sigma(T_D) \setminus \left\{0 \right\}$. Then there exists a sequence $u_n \in H^1_0(\Omega)$ of quasi-eigenvectors, i.e. 
\begin{equation}\label{eq.evTDtoTDhat}
 \lvert\lvert T_D u_n - \lambda u_n \lvert\lvert_{H^1_0(\Omega)} \xrightarrow{n \to \infty} 0, \text{ and } \lvert\lvert u_n \lvert\lvert_{H^1_0(\Omega)}=1.
 \end{equation}
Let $v_n = \widetilde{u_n\lvert_D} \in H_D$. Then, 
 $$ \begin{array}{ccl}
\lvert\lvert \mathring{T_D}v_n - \lambda v_n \lvert\lvert_{H_D} & = & \lvert\lvert \nabla (T_D u_n) - \lambda \nabla u_n \lvert\lvert_{L^2(D)^d}, \\
&\leq &  \lvert\lvert T_D u_n - \lambda u_n \lvert\lvert_{H^1_0(\Omega)}, 
\end{array}$$
converges to $0$ as $n \to \infty$, which proves that $\lambda \in \sigma(\mathring{T_D})$, provided $v_n$ does not converge to $0$ in $H_D$. 
But if that were the case, by (\ref{eq.evTDtoTDhat}), we would have $T_D u_n \to 0$ in $H^1_0(\Omega)$, and therefore $\lambda u_n \to 0$ in $H^1_0(\Omega)$; the latter is impossible since $\lambda \neq 0$ and $\lvert\lvert u_n \lvert\lvert_{H^1_0(\Omega)}=1$. 
Hence $\lambda$ belongs to the spectrum of $\mathring{T_D}$. 

Conversely, let $\lambda \in \sigma(\mathring{T_D})$. By Proposition \ref{prop.TDring}, $\lambda \neq 0$, 
and there exists a sequence $v_n \in H_D$ of quasi-eigenvectors for $\lambda$: 
\begin{equation}\label{eq.evTDhat}
 \lvert\lvert \mathring{T_D} v_n - \lambda v_n \lvert\lvert_{H_D} \xrightarrow{n \to \infty} 0, \text{ and } \lvert\lvert v_n \lvert\lvert_{H_D} = 1.
 \end{equation}
Taking functions $u_n \in H^1_0(\Omega)$ such that the restriction $u_n \lvert_D$ belongs to the class $v_n$, (\ref{eq.evTDhat}) becomes: 
$$ \lvert\lvert \nabla (T_Du_n) - \lambda \nabla u_n \lvert\lvert_{L^2(D)^d}  \xrightarrow{n \to \infty} 0, \text{ and } \lvert\lvert \nabla u_n \lvert\lvert_{L^2(D)^d} = 1.$$ 
Let $w_n \in C(D)$ be defined by $w_n = \frac{1}{\lvert D_j \lvert}\int_{D_j} {(T_D u_n - \lambda u_n)\:dx}$ on each connected component $D_j$ of $D$, $j=1,...,N$. By the properties of the extension operator $U_D$ and the the Poincar\'e-Wirtinger inequality, it comes: 
$$ \begin{array}{ccl}
\lvert\lvert U_D (T_D u_n\lvert_D) - \lambda U_D (u_n \lvert_D) - U_D w_n \lvert\lvert_{H^1_0(\Omega)} &\leq & C\lvert\lvert T_D u_n - \lambda u_n - w_n \lvert\lvert_{H^1(D)}, \\
&\leq & C\lvert\lvert \nabla (T_D u_n) - \lambda \nabla u_n \lvert\lvert_{L^2(D)^d}. 
\end{array}$$
Now defining $z_n = (U_D (u_n\lvert_D) - \frac{1}{\lambda} U_D w_n)\in H^1_0(\Omega)$, and using (\ref{eq.factsTD}), we obtain 
$$ \lvert\lvert T_D z_n - \lambda z_n \lvert\lvert_{H^1_0(\Omega)} \xrightarrow{n \to \infty} 0, $$
which allows to conclude that $\lambda \in \sigma(T_D)$, provided we can show that $z_n$ does not converge to $0$ in $H^1_0(\Omega)$.
But this last fact is obvious since $\lvert\lvert \nabla z_n \lvert\lvert_{L^2(D)^d} = \lvert\lvert v_n \lvert\lvert_{H_D} = 1$.

The correspondance between eigenvalues $\lambda \neq 0$ of $T_D$ and eigenvalues of $\mathring{T_D}$ is proved in exactly the same way, 
working directly on eigenvectors rather than on quasi-eigenvector sequences.  
\end{proof}

\begin{remark}
From these results, one might get the misleading impression that the operator $T_D$ (and notably its spectral properties) 
depends only on the geometry of the set of inclusions $D$, and not on that of $\Omega$. This is wrong: the above material is merely a convenient point of view for 
appraising $T_D$ only by its action on functions defined on $D$; in this regard, see the results of Section \ref{seccompleteness}.
\end{remark}

%%%%%%%%%%%%%%%%%%%%%%%%%%%%%%%%%%%%%%%%%%%%%%%%%%%%%%%
\section{Structure of the limit spectrum of the Poincar\'e variational problem}\label{sec.limspec}
%%%%%%%%%%%%%%%%%%%%%%%%%%%%%%%%%%%%%%%%%%%%%%%%%%%%%%%

We now come
back to the periodic inclusion setting described in Section \ref{sec.setting}: for a fixed $\varepsilon >0$, 
the considered set of inclusions is $D=\omega_\varepsilon$, as defined by (\ref{eq.omeps}), and we use the shorthand $T_\varepsilon := T_{\omega_\varepsilon}$.
Our main goal in this section is to describe the structure of the \textit{limit spectrum} $\lim_{\varepsilon \rightarrow 0}{\sigma(T_\varepsilon)}$, which is the closed subset of $[0,1]$ defined by (\ref{def.limspec}); 
unless otherwise specified, we proceed under the following assumptions: 
\begin{equation}\label{eq.assumom}
\begin{minipage}{0.9\textwidth}
\begin{enumerate}[(i)]
\item The rescaled inclusion $\omega$ is strictly
contained in the unit periodicity cell $Y$: $\omega \Subset Y$,
\item $\omega$ is of class ${\mathcal C}^2$.
\end{enumerate}
\end{minipage}
\end{equation}
Interestingly, many of the conclusions of this section do not hold if these assumptions are not satisfied; see the Appendix \ref{sec.laminates} for counter-examples.

%%%%%%%%%%%%%%%%%%%%%%%%%%%%%%%%%%%%%%
\subsection{Presentation of the main results}\label{sec.mainresspec}~\\
%%%%%%%%%%%%%%%%%%%%%%%%%%%%%%%%%%%%%%

The first result of this investigation concerns the behavior of the non-degenerate part of the spectrum $\sigma(T_\varepsilon)$ as $\varepsilon \to 0$ 
(i.e. the eigenvalues of $T_\varepsilon$ which are different from $0$ and $1$).  
We show in Theorem \ref{th.evnp} that this part of $\sigma(T_\varepsilon)$ is uniformly contained in some interval of the form $(m,M)$ for some 
$0 < m < M < 1$ independent of $\varepsilon$,
which we manage to characterize in terms of the shape of $\omega$ only. 
We shall draw several important consequences from this result in Section \ref{secdirect}. 

We then examine in more details the behavior of the limit spectrum $\lim_{\varepsilon \rightarrow 0}{\sigma(T_\varepsilon)}$.
Usually, such endeavour
involves the study of the limiting behavior of the operator $T_\varepsilon$ as $\varepsilon \rightarrow 0$. 
Unfortunately, in the present context, $T_\varepsilon$ lacks `nice' convergence properties; indeed, it is fairly simple to show 
that $T_\varepsilon$ converges \textit{weakly} to the trivial operator $\lvert \omega \lvert I$, that is:
$$\forall u \in H^1_0(\Omega), \:\: T_\varepsilon u \stackrel{\varepsilon \rightarrow 0}{\longrightarrow}  \lvert \omega \lvert u, \text{ weakly in } H^1_0(\Omega),$$
but that convergence yields no information about the asymptotic behavior of $\sigma(T_\varepsilon)$. 

The reason for this `bad' behavior is well understood in the mathematical theory of homogenization: the sequence $T_\varepsilon u$ shows oscillations of small amplitude, at the $\varepsilon$-scale of the inclusions, which are somehow
averaged in the trivial limit operator $\lvert \omega\lvert I$. Usually, so-called bulk or boundary layer \textit{correctors} are introduced in order to strenghten this convergence (see for instance \cite{allairehomog,blp,jikov}), but this approach seems difficult to implement here. 

To carry out our study, we follow the approach of \cite{allaireconcafs,allaireconca}, using the \textit{Bloch-wave homogenization} method: 
namely, we rescale the operator $T_\varepsilon: H^1_0(\Omega) \to H^1_0(\Omega)$ into one $\mathbb{T}_\varepsilon: L^2(\Omega,H^1(\omega)/\mathbb{R}) \to L^2(\Omega,H^1(\omega)/\mathbb{R})$, which explicitely takes into account both the macroscopic and 
microscopic scales of the problem. The operator $\mathbb{T}_\varepsilon$ retains the spectral features of
$T_\varepsilon$, but shows better convergence properties. 
More precisely, Proposition \ref{prop.cvrescal1} below shows that it converges \textit{pointwise} to a limit operator $\mathbb{T}_0$, i.e.: 
\begin{equation}\label{eq.explTepsT0}
 \forall \phi \in  L^2(\Omega,H^1(\omega)/\mathbb{R}), \:\: \mathbb{T}_\varepsilon \phi \xrightarrow{\varepsilon\to 0} \mathbb{T}_0 \phi, \text{ strongly in }  L^2(\Omega,H^1(\omega)/\mathbb{R}).
 \end{equation}

We shall then make use of the following abstract result: 

\begin{proposition}\label{prop.speclsc}
Let $T_\varepsilon: H \rightarrow H$ be a sequence of bounded, self-adjoint operators in a Hilbert space $(H,\langle,\cdot,\cdot \rangle)$. 
We assume that $T_\varepsilon$ converges pointwise
to a limit self-adjoint operator $T_0$. Then, every point $\lambda_0$ in the spectrum $\sigma(T_0)$ is the 
accumulation point of a sequence $\lambda_\varepsilon$, where $\lambda_\varepsilon \in \sigma(T_\varepsilon)$. 
\end{proposition}
\begin{proof}
Assume that there exists $\lambda_0 \in \sigma(T_0)$ which is not the accumulation point of any sequence $\lambda_\varepsilon \in \sigma(T_\varepsilon)$.
Then, there exists $\delta > 0$ such that 
$$ \forall \varepsilon >0, \: \forall \lambda \in \sigma(T_\varepsilon), \quad \lvert \lambda - \lambda_0 \lvert > \delta.$$
Introducing a resolution of the identity $E_{T_\varepsilon}$ associated to the operator $T_\varepsilon$, it comes, for an arbitrary $u \in H$ (see e.g. \cite{rudin}, Chap. 12):
$$ \lvert\lvert T_\varepsilon u - \lambda_0 u \lvert\lvert^2 =  \left\lvert\left\lvert  \int_{\sigma(T_\varepsilon)}{(\lambda - \lambda_0) \:dE_{T_\varepsilon}(\lambda)}u \:\right\lvert\right\lvert^2 =  \int_{\sigma(T_\varepsilon)}{\lvert \lambda - \lambda_0\lvert^2 \langle dE_{T_\varepsilon}(\lambda)u,u\rangle} \geq \delta^2 \lvert\lvert u \lvert\lvert^2.$$
Since the convergence $T_\varepsilon u \rightarrow T_0 u$ is strong, passing to the limit $\varepsilon \rightarrow 0$ in the previous relation yields: 
$$ \forall u \in H, \:\: \lvert\lvert T_0 u - \lambda_0 u \lvert\lvert \geq \delta^2 \lvert\lvert u \lvert\lvert.$$
Thus, $T_0- \lambda_0 I$ is injective and has closed range. 
Since it is self-adjoint, it follows that it is also surjective, and, by the open mapping theorem, that it is an isomorphism.
This contradicts the hypothesis $\lambda_0 \in \sigma(T_0)$.
\end{proof}

Hence, combining (\ref{eq.explTepsT0}) with Proposition \ref{prop.speclsc} 
allows to capture one part of the limit spectrum, which is (partially) identified in Lemma \ref{lem.loc} below
as the spectrum of the periodic Poincar\'e variational operator associated to the  
inclusion pattern $\omega$.

\begin{theorem}\label{th.cellspec}
The limit spectrum $\lim_{\varepsilon \to 0}{\sigma(T_\varepsilon)}$ contains the \textit{cell spectrum}, that is, the spectrum $\sigma(T_0)$ of 
the operator $T_0 : H^1_\#(Y) / \mathbb{R} \to H^1_\#(Y) / \mathbb{R}$ defined as follows: for $u \in H^1_\#(Y) / \mathbb{R}$, $T_0 u$ is the unique element in $H^1_\#(Y) / \mathbb{R}$ such that
\begin{equation}\label{eq.defT0}
 \forall v \in H^1_\#(Y) / \mathbb{R}, \:\: \int_Y{\nabla_y (T_0 u) \cdot \nabla_y v \:dy} = \int_\omega{\nabla_y u \cdot \nabla_y v \:dy}.
 \end{equation}
\end{theorem}

In the above statement, 
$H^1_\#(Y)$ stands for the completion of the set of smooth $Y$-periodic functions for the usual $H^1(Y)$-norm (this definition readily extends to define the set $W^{1,p}_\#(Y))$.
The Hilbert quotient space $H^1_\#(Y) / \mathbb{R}$ is equipped with the inner product: 
$$ \langle u , v\rangle_{H^1_\#(Y) / \mathbb{R}} = \int_Y{\nabla_y u \cdot \nabla_y v \:dy},$$
and its associated norm. 

By changing the topology measuring the convergence of $T_\varepsilon$, we are able to capture \textit{one part} of the limit spectrum $\lim_{\varepsilon \to 0}{\sigma(T_\varepsilon)}$, but not all of it. There are actually several ways to consider $T_\varepsilon$ on both macroscopic and microscopic cells. 
Indeed, following an idea of J. Planchard \cite{planchard}, then carried on by G. Allaire and C. Conca \cite{allaireconcafs, allaireconca}, 
we could perform the very same procedure with a pack of $K^d$ cells as a basis for the microscopic scale, 
which are naturally associated to a discrete Bloch decomposition. We then obtain: 

\begin{theorem}\label{th.blochspec}
The limit spectrum $\lim_{\varepsilon\to 0}{\sigma(T_\varepsilon)}$ contains the \textit{Bloch spectrum} 
$\sigma_{\text{\rm Bloch}}$ defined by 
\begin{equation}\label{eq.defBlochspec}
\sigma_{\text{\rm Bloch}} = \bigcup_{j =0}^\infty{\left[ \min_{\eta \in [0,1]^d}{\lambda_j(\eta)}, \max_{\eta \in [0,1]^d}{\lambda_j(\eta)}  \right]},
\end{equation}
where, for $j=0,...$ and $\eta \in [0,1]^d$, $\lambda_j(\eta)$ is the $j^{\text{th}}$ eigenvalue of the Bloch operator $T_\eta$ defined in (\ref{eq.defT0Bloch}-\ref{eq.defTetaBloch}) below. 
\end{theorem}

Obviously, the Bloch spectrum contains the cell spectrum; it really captures `bulk' effects of the collective resonances of particles.

%as is shown by Proposition \ref{prop.blochcompactsupp}, which roughly speaking asserts that every element in the Bloch spectrum is the limit of 
%a sequence of eigenvalues whose associated eigenvectors have all their energy far away from the boundary $\partial \Omega$. 
%

The former study makes it possible to capture the so-called Bloch part of the limit spectrum, 
and one may wonder whether all the possible limits are captured. Such is not the case because of the possible
concentration of energy on the boundary. In Section \ref{seccompleteness}, Theorem \ref{th.blochbl}, we show that 
$\lim_{\varepsilon \to 0}{\sigma(T_\varepsilon)}= \sigma_{\text{\rm Bloch}} \cup \sigma_{\partial \Omega}$, 
where the so-called \textit{Boundary layer spectrum} $\sigma_{\partial \Omega}$ is composed of the accumulation points $\lambda$ of the sequences of 
eigenvalues $\lambda_\varepsilon \in \sigma(T_\varepsilon)$ whose associated eigenfunctions concentrate a large enough proportion of their energy near the boundary $\partial \Omega$ of the macroscopic domain.

%%%%%%%%%%%%%%%%%%%%%%%%%%%%%%%%%%%%%%
\subsection{Uniform bounds on the eigenvalues of the Poincar\'e variational problem}\label{sec.boundev}~\\
%%%%%%%%%%%%%%%%%%%%%%%%%%%%%%%%%%%%%%

From Proposition \ref{prop.evTD}, we already now that the spectrum $\sigma(T_\varepsilon)$ is included in the interval $\left[0,1\right]$. 
Moreover, this proposition states that $0$ and $1$ are both eigenvalues of $T_\varepsilon$, and gives a complete characterization of the associated eigenspaces. 

The following theorem actually reveals that there is a gap
between the eigenvalues $0$ and $1$ and the rest of the spectrum $\sigma(T_\varepsilon)$, 
which is uniform with respect to $\varepsilon$. 

\begin{theorem}\label{th.evnp}
Under the assumptions (\ref{eq.assumom}), there exists $\varepsilon_0$ such that, for $0< \varepsilon <\varepsilon_0$,
$$ 
\left( \lambda \in \sigma(T_\varepsilon),  \:\: 
\lambda \notin \left\{Ê0, 1 \right\} \right)
\quad \Rightarrow \quad m \leq \lambda \leq M,
$$
where $0<m<M<1$ are the constants, independent of $\varepsilon$,
defined by:
\begin{equation}\label{eq.defmM}
 m = \min \limits_{u \in \widehat{{\mathfrak h}_0} \atop u \neq 0}{\frac{\displaystyle{\int_\omega{\lvert \nabla_y u \lvert^2\:dy}}}{\displaystyle{\int_Y{\lvert \nabla_y u \lvert^2\:dy}}}}, \quad M = \max \limits_{u \in \widehat{{\mathfrak h}_0} \atop u \neq 0}{\frac{\displaystyle{\int_\omega{\lvert \nabla_y u \lvert^2\:dy}}}{\displaystyle{\int_Y{\lvert \nabla_y u \lvert^2\:dy}}}},
 \end{equation}
and $\widehat{{\mathfrak h}_0} \subset H^1(Y) / \mathbb{R}$ is the Hilbert space defined by: 
\begin{equation}\label{eq.defhhat0}
\widehat{{\mathfrak h}_0} = \left\{Êu  \in H^1(Y) / \mathbb{R}, \:\: \Delta_y u = 0 \text{ in } \omega \cup (Y \setminus \overline{\omega}), \text{ and } \int_{\partial \omega}{\frac{\partial u^+}{\partial n_y}\:ds} = 0\right\} .
\end{equation}
\end{theorem}
\begin{proof}
Let us start with a short comment about the bounds $m,M$ and the space $\widehat{{\mathfrak h}_0}$ defined by (\ref{eq.defmM}) and (\ref{eq.defhhat0}) respectively. 
Considering the version of the Poincar\'e variational problem (\ref{eq.defTD}) associated
to the inclusion $\omega$ in the periodicity cell $Y$, 
let us introduce the operator $\widehat{T_0} : H^1(Y) / \mathbb{R} \to H^1(Y) / \mathbb{R}$ defined as follows: 
for $u \in H^1(Y) / \mathbb{R}$, $\widehat{T_0}u$ is the unique element in $H^1(Y) / \mathbb{R}$ such that: 
$$ \forall v \in H^1(Y) / \mathbb{R}, \:\: \int_Y{\nabla_y (\widehat{T_0}u)\cdot \nabla_y v \:dy} = \int_\Omega{\nabla_y u \cdot \nabla_y v \:dy}.$$
Observe that the operator $\widehat{T_0}$ and space $\widehat{{\mathfrak h}_0}$ are very close in essence to those $T_0$ and ${\mathfrak h}_0$ 
defined below (see (\ref{eq.varfT0}) and Proposition \ref{prop.specTeta}), except that they \textit{do not} bring into play $Y$-periodic functions. 

Adapting the analysis of Section \ref{sec.NPTD} to the present case - in particular, replacing the Poisson kernel (\ref{eq.poisk}) 
with the so-called Neumann function of the inclusion $\omega \Subset Y$ - one shows that $m$ and $M$ are respectively
the smallest and largest eigenvalues of $\widehat{T_0}$ which are different from $0$ and $1$
(and thus $0 < m < M < 1$).\\

Let us now pass to the proof of Theorem \ref{th.evnp}, properly speaking. 
We denote by $\lambda_\varepsilon^-$ (resp. $\lambda_\varepsilon^+$) the smallest (resp. largest) eigenvalue of $T_\varepsilon$ 
which is different from $0$ (resp. different from $1$).

Exploiting the min-max principle of Proposition \ref{propminmax} in combination with the characterization of 
$\text{\rm Ker}(T_\varepsilon)$ given in
Proposition \ref{prop.evTD}, it comes:
\begin{equation}\label{eq.deflepspm}
 \lambda^-_\varepsilon =  \min \limits_{u \in {\mathfrak h}_\varepsilon \atop u \neq 0}{\frac{\displaystyle{\int_{\omega_\varepsilon}{\lvert \nabla u \lvert^2\:dx}}}{\displaystyle{\int_{\Omega}{\lvert \nabla u \lvert^2\:dx}}}}, \text{ and } \lambda^+_\varepsilon =  \max \limits_{u \in {\mathfrak h}_\varepsilon \atop u \neq 0}{\frac{\displaystyle{\int_{\omega_\varepsilon}{\lvert \nabla u \lvert^2\:dx}}}{\displaystyle{\int_{\Omega}{\lvert \nabla u \lvert^2\:dx}}}},
 \end{equation}
where the space ${\mathfrak h}_\varepsilon$ is defined by (see (\ref{eq.mathfrakh})):
$$ {\mathfrak h}_\varepsilon = \left\{ u \in H^1_0(\Omega), \:\: \Delta u = 0 \text{ on } \omega_\varepsilon \cup (\Omega\setminus \overline{\omega_\varepsilon}), \text{ and } \int_{\partial \omega_\varepsilon^\xi}{\frac{\partial u^+}{\partial n}\:ds} = 0, \:\: \xi \in \Xi_\varepsilon\right\}. $$

Our purpose is to prove that 
\begin{equation}\label{eq.boundsrap}
m \leq \lambda_\varepsilon^- , \text{ and } \lambda_\varepsilon^+ \leq M.
\end{equation}\par
\medskip

\noindent\textit{Proof of the right-hand inequality in (\ref{eq.boundsrap}):} 
Let $u \in {\mathfrak h}_\varepsilon$, $u \neq 0$ be arbitrary. 
For any $\xi \in \Xi_\varepsilon$, define the rescaled function $u_\varepsilon^\xi(y) := u(\varepsilon \xi + \varepsilon y)$ in $H^1(Y)$. 
A simple change of variables yields: 
\begin{equation}\label{eq.decintomeps1}
\int_{\omega_\varepsilon}{\lvert \nabla u \lvert^2 \:dx} = \varepsilon^{d-2} \sum\limits_{\xi \in \Xi_\varepsilon}{\int_\omega{\lvert \nabla_y u_\varepsilon^\xi \lvert^2 \:dy}},
\end{equation}
and similarly:
\begin{equation}\label{eq.decintomeps2}
\int_{\Omega}{\lvert \nabla u \lvert^2 \:dx} = \int_{{\mathcal B}_\varepsilon}{\lvert \nabla u \lvert^2 \:dx} +  \varepsilon^{d-2} \sum\limits_{\xi \in \Xi_\varepsilon}{\int_Y{\lvert \nabla_y u_\varepsilon^\xi \lvert^2 \:dy}}.
\end{equation}
We then obtain: 
\begin{equation}\label{eq.decrapmaj}
\frac{\displaystyle{\int_{\omega_\varepsilon}{\lvert \nabla u \lvert^2\:dx}}}{\displaystyle{\int_{\Omega}{\lvert \nabla u \lvert^2\:dx}}} = \frac{\displaystyle{\varepsilon^{d-2} \sum\limits_{\xi \in \Xi_\varepsilon}{\int_\omega{\lvert \nabla_y u_\varepsilon^\xi \lvert^2 \:dy}}}}{\displaystyle{ \int_{{\mathcal B}_\varepsilon}{\lvert \nabla u \lvert^2 \:dx} +  \varepsilon^{d-2} \sum\limits_{\xi \in \Xi_\varepsilon}{\int_Y{\lvert \nabla_y u_\varepsilon^\xi \lvert^2 \:dy}}}} \leq  \max\limits_{\xi \in \Xi_\varepsilon}{\frac{\displaystyle{\int_{\omega}{\lvert \nabla_y u_\varepsilon^\xi \lvert^2\:dy}}}{\displaystyle{\int_{Y}{\lvert \nabla_y u_\varepsilon^\xi \lvert^2\:dy}}} },
\end{equation}
where we have used the easy algebraic identity: 
\begin{equation}\label{eq.easyalg}
\min\left( \frac{p_1}{q_1} , \frac{p_2}{q_2}\right) \leq \frac{p_1+p_2}{q_1+q_2} \leq \max\left( \frac{p_1}{q_1}, \frac{p_2}{q_2}\right), \:\: p_1,p_2,q_1,q_2 \geq 0, \:\: q_1q_2 \neq 0.
\end{equation}
Now, since $u \in {\mathfrak h}_\varepsilon$, it follows that for every $\xi \in \Xi_\varepsilon$, $u_\varepsilon^\xi$ (or more exactly its class up to constants) 
belongs to $\widehat{{\mathfrak h}_0}$. Hence, in view of (\ref{eq.defmM}), the right-hand side in (\ref{eq.decrapmaj}) is bounded from above by $M$. 
The desired inequality now follows from the combination of (\ref{eq.deflepspm}) and (\ref{eq.decrapmaj}).\\
\\
\textit{Proof of the left-hand inequality in (\ref{eq.boundsrap}):} The proof of this inequality is a little more involved than that of the former one, and we proceed in
two steps.\\ 
\\
\textit{First step:} We prove that, for an arbitrary $u \in {\mathfrak h}_\varepsilon$, $u\neq 0$, 
there exist $\lvert \Xi_\varepsilon \lvert$ real constants $c_\xi \in \mathbb{R}$, indexed by $\xi \in \Xi_\varepsilon$, 
so that the function $v \in H^1_0({\mathcal O}_\varepsilon)$ defined by the system:
\begin{equation}\label{eq.defveigenv}
\left\{ 
\begin{array}{cl}
-\Delta v = 0 & \text{in }Ê{\mathcal O}_\varepsilon, \\
v = u + c_\xi & \text{in } \omega_\varepsilon^\xi, \text{ for each } \xi \in \Xi_\varepsilon, \\
v = 0 & \text{on }Ê\partial {\mathcal O}_\varepsilon
\end{array}
\right.
\end{equation}
additionally satisfies:
\begin{equation}\label{eq.intdvdn+=0}
 \int_{\partial \omega_\varepsilon^\xi}{\frac{\partial v^+}{\partial n}\:ds} = 0 \text{ for all }Ê\xi \in \Xi_\varepsilon.
\end{equation}
To see this, for a given $\xi \in \Xi_\varepsilon$, let $s_\xi \in H^1_0({\mathcal O}_\varepsilon)$ be defined by: 
$$ 
\left\{ 
\begin{array}{cl}
-\Delta s_\xi = 0 & \text{in }Ê{\mathcal O}_\varepsilon, \\
s_\xi = 1 & \text{in } \omega_\varepsilon^\xi, \\
s_\xi = 0 & \text{in }Ê\omega_\varepsilon^\eta, \text{ for each } \eta \in \Xi_\varepsilon \setminus \left\{ \xi \right\}, \\
s_\xi = 0 & \text{on } \partial {\mathcal O}_\varepsilon.
\end{array}
\right.
$$
We then search for $v$ under the form: 
$$ v = u + \sum\limits_{\xi \in \Xi_\varepsilon}{c_\xi s_\xi},$$
where the constants $c_\xi$ have to be adjusted so that (\ref{eq.intdvdn+=0}) holds; this last condition in turn rewrites: 
\begin{equation}\label{eq.proofbndevsysl}
\forall \xi \in \Xi_\varepsilon, \:\:  \sum\limits_{\eta \in \Xi_\varepsilon}{A_{\xi\eta} c_\eta} = b_\xi,
\end{equation}
a linear system of size $\lvert \Xi_\varepsilon\lvert$, 
whose matrix $A = (A_{\xi\eta})_{\xi,\eta \in \Xi_\varepsilon}$ and right-hand side $\left\{Êb_\xi \right\}_{\xi \in \Xi_\varepsilon}$ are defined by: 
$$ A_{\xi\eta} = \int_{\partial \omega_\varepsilon^\xi}{\frac{\partial s_{\eta}^+}{\partial n} \:ds}, \text{ and } b_\xi = -\int_{\partial \omega_\varepsilon^\xi}{\frac{\partial u^+}{\partial n}\:ds}.$$
Hence, the existence of $v$ satisfying (\ref{eq.defveigenv}) and (\ref{eq.intdvdn+=0}) follows from the invertibility of $A = (A_{\xi\eta})_{\xi,\eta \in \Xi_\varepsilon}$, which we now prove. 
Assume then that $\sum_{\eta \in \Xi_\varepsilon}{A_{\xi\eta}c_\eta} = 0$ for some collection $\left\{Êc_{\xi} \right\}_{\xi \in \Xi_\varepsilon}$ of $\lvert \Xi_\varepsilon\lvert$ real constants,  and consider the function 
$$w = \sum\limits_{\xi \in \Xi_\varepsilon}{c_{\xi} s_{\xi}};$$ 
$w$ satisfies by construction:
$$ 
\left\{ 
\begin{array}{cl}
-\Delta w = 0 & \text{in }Ê{\mathcal O}_\varepsilon, \\
w = c_\xi & \text{in }Ê\omega_\varepsilon^\xi, \text{ for each }Ê\xi \in \Xi_\varepsilon, \\
w = 0 & \text{on } \partial {\mathcal O}_\varepsilon,
\end{array}
\right.
$$
together with $\int_{\partial \omega_\varepsilon^\xi}{\frac{\partial w^+}{\partial n}\:ds} = 0$ for each $\xi \in \Xi_\varepsilon$. 
Hence, an integration by parts produces:
$$ \begin{array}{>{\displaystyle}cc>{\displaystyle}l}
\int_{{\mathcal O}_\varepsilon}{\lvert \nabla w \lvert^2 \:dx} &=& -\int_{{\mathcal O}_\varepsilon \setminus \overline{\omega_\varepsilon}} {\Delta w \: w \:dx} - \sum\limits_{\xi \in \Xi_\varepsilon}{\int_{\partial \omega_\varepsilon^\xi}{\frac{\partial w^+}{\partial n} w \:ds}}, \\
&=& - \sum\limits_{\xi \in \Xi_\varepsilon}{c_\xi \: \int_{\partial \omega_\varepsilon^\xi}{\frac{\partial w^+}{\partial n} \:ds}}, \\
&=& 0,
\end{array}
$$
and so $w \equiv 0$ on ${\mathcal O}_\varepsilon$; in particular, this implies $c_\xi = 0$, for $\xi \in \Xi_\varepsilon$, which eventually proves the invertibility of the matrix $A$ in (\ref{eq.proofbndevsysl}), thus the existence of $v$ satisfying (\ref{eq.defveigenv}) and (\ref{eq.intdvdn+=0}). \\
\\
\textit{Second step:} We use the construction of the first step to obtain a lower bound for $\lambda_\varepsilon^-$ in (\ref{eq.deflepspm}). 
First, let $u \in {\mathfrak h}_\varepsilon$, $u\neq 0$ be arbitrary, and let $\left\{Êc_\xi \right\}_{\xi \in \Xi_\varepsilon}$ and $v \in H^1_0({\mathcal O}_\varepsilon)$ be 
as in the first step. We define $r \in \text{\rm Ker}(T_\varepsilon)$ by (see Proposition \ref{prop.evTD}): 
$$ \left\{ 
\begin{array}{cl}
-\Delta r = 0 & \text{in } \Omega \setminus \overline{\omega_\varepsilon}, \\ 
r = c_\xi & \text{in } \omega_\varepsilon^\xi, \text{ for each } \xi \in \Xi_\varepsilon, \\
r = 0 &\text{on } \partial \Omega.
\end{array}
\right. $$
Then,
\begin{equation}\label{eq.inturleqintv}
\int_\Omega{\lvert \nabla u +\nabla r \lvert^2 \:dx} = \int_\Omega{\lvert \nabla u \lvert^2 \:dx} + \int_\Omega{\lvert \nabla r \lvert^2 \:dx} \leq \int_{{\mathcal O}_\varepsilon}{\lvert \nabla v \lvert^2\:dx},
\end{equation}
where the left-hand identity follows from the orthogonality between ${\mathfrak h}_\varepsilon$ and $\text{\rm Ker}(T_\varepsilon)$, 
and the right-hand inequality follows from the fact that $u+r$ is the minimizer of the 
energy $\int_\Omega{\lvert \nabla w \lvert^2 \:dx}$ over the closed, convex subset of $H^1_0(\Omega)$ defined by: 
$$\left\{ w \in H^1_0(\Omega), \:\: w = u + c_\xi \text{ in } \omega_\varepsilon^\xi, \:\: \xi \in \Xi_\varepsilon \right\}.$$  

Eventually, introducing again the rescaled function $v_\varepsilon^\xi(y) = v(\varepsilon \xi + \varepsilon y)$, a change of variable produces (viz. (\ref{eq.decintomeps1}) and (\ref{eq.decintomeps2})): 
$$
\int_{\omega_\varepsilon}{\lvert \nabla v \lvert^2 \:dx} = \varepsilon^{d-2} \sum\limits_{\xi \in \Xi_\varepsilon}{\int_\omega{\lvert \nabla_y v_\varepsilon^\xi \lvert^2 \:dy}},
\text{ and }
\int_{{\mathcal O}_\varepsilon}{\lvert \nabla v \lvert^2 \:dx} =  \varepsilon^{d-2} \sum\limits_{\xi \in \Xi_\varepsilon}{\int_Y{\lvert \nabla_y v_\varepsilon^\xi \lvert^2 \:dy}},
$$
whence, using (\ref{eq.easyalg}) and (\ref{eq.inturleqintv}): 
\begin{equation*}
\frac{\displaystyle{\int_{\omega_\varepsilon}{\lvert \nabla u \lvert^2\:dx}}}{\displaystyle{\int_{\Omega}{\lvert \nabla u \lvert^2\:dx}}} \geq  \frac{\displaystyle{\int_{\omega_\varepsilon}{\lvert \nabla v \lvert^2\:dx}}}{\displaystyle{\int_{{\mathcal O}_\varepsilon}{\lvert \nabla v \lvert^2\:dx}}}= \frac{\displaystyle{\varepsilon^{d-2} \sum\limits_{\xi \in \Xi_\varepsilon}{\int_\omega{\lvert \nabla_y v_\varepsilon^\xi \lvert^2 \:dy}}}}{\displaystyle{\varepsilon^{d-2} \sum\limits_{\xi \in \Xi_\varepsilon}{\int_Y{\lvert \nabla_y v_\varepsilon^\xi \lvert^2 \:dy}}}} \geq  \min\limits_{\xi \in \Xi_\varepsilon}{\frac{\displaystyle{\int_{\omega}{\lvert \nabla_y v_\varepsilon^\xi \lvert^2\:dy}}}{\displaystyle{\int_{Y}{\lvert \nabla_y v_\varepsilon^\xi \lvert^2\:dy}}} },
\end{equation*}
which allows to conclude from (\ref{eq.defmM}) and (\ref{eq.deflepspm}), since by construction $v_\varepsilon^\xi$ belongs to $\widehat{{\mathfrak h}_0}$ for each $\xi \in \Xi_\varepsilon$.
\end{proof}

\begin{remark}
\noindent \begin{itemize}
\item Interestingly, the bounds (\ref{eq.defmM}) over the non degenerate part of the spectrum of $T_\varepsilon$ 
only depend on the geometry of the inclusion $\omega \Subset Y$, and not on that of the macroscopic domain $\Omega$.
\item As a consequence of Theorem \ref{th.evnp}, the operator $(\frac{1}{1-a}I - T_\varepsilon)$, from ${\mathfrak h}_\varepsilon$ into itself, 
is invertible as soon as $a$ belongs to $(-\infty,\frac{m-1}{m}) \cup (\frac{M-1}{M},0)$. We believe that this result can be extended to the case where the inclusion $\omega$ is only assumed to be Lipschitz regular. However, for the sake of simplicity, we do not discuss such generality in the present article.
\item In the case where $\omega$ is additionnally assumed to be star-shaped, it is possible to derive more explicit expressions of the bounds (\ref{eq.defmM}) 
in terms of the Lipschitz character $\inf_{y \in \partial \omega}(y \cdot n(y))$ of the inclusion $\omega$, be relying on the Rellich identities as in \cite{AmmariSeo}.
\end{itemize}
\end{remark}
\subsection{Individual resonances of the inclusions: proof of Theorem \ref{th.cellspec}}
%%%%%%%%%%%%%%%%%%%%%%%%%%%%%%%%%%%%%

\subsubsection{Extension and projection operators: the periodic unfolding method}\label{secunfold}~\\ 

The cornerstone of our rescaling procedure of $T_\varepsilon$ is the definition of extension and projection operators for transforming 
a function of the macroscopic variable $x \in \Omega$ into one depending on both the macroscopic and microscopic variables $x \in \Omega$, $y \in Y$, 
and the other way around. 
These operators have been introduced in \cite{allaireconcafs,allaireconca}, then studied in a more systematic way in \cite{unfoldcras,unfold}, 
under the name of `periodic unfolding method'. We follow the presentation of the latter references. 

\begin{definition}\label{def.unfold}
Let $p \in [1,\infty]$.
 
\begin{enumerate}[(i)]
\item The extension (or unfolding) operator $E_\varepsilon: L^p(\Omega) \rightarrow L^p(\Omega \times Y)$ is defined for $u \in L^p(\Omega)$ by:
$$E_\varepsilon u(x,y) = \left\{ 
\begin{array}{cl}
u(\varepsilon \left[ \frac{x}{\varepsilon}\right]_Y + \varepsilon y) & \text{if } x \in {\mathcal O}_\varepsilon,\\
0 & \text{otherwise}.
\end{array}
\right.$$
\item The projection (or local averaging) operator
$P_\varepsilon : L^p(\Omega \times Y) \rightarrow L^{p}(\Omega)$ is defined for $\phi \in L^p(\Omega \times Y)$ by: 
$$P_\varepsilon \phi(x) = \left\{ \begin{array}{>{\displaystyle}cl}
\int_Y{\phi\left(\varepsilon\left[ \frac{x}{\varepsilon}\right]_Y + \varepsilon z ,\left\{ \frac{x}{\varepsilon}\right\}_Y \right)\:dz} & \text{if } x \in {\mathcal O}_\varepsilon, \\
0 & \text{otherwise}.\\
\end{array}
\right.$$
\end{enumerate}
\end{definition}
\begin{remark}
A more intuitive definition of the projection $P_\varepsilon \phi$ would read: $\Omega \ni x \mapsto \phi(x,\left\{\frac{x}{\varepsilon}\right\}_Y)$, 
but this function may not be measurable for general $\phi \in L^p(\Omega \times Y)$.
\end{remark}

These operators satisfy the following properties, whose elementary proofs are reproduced for the sake of convenience: 
\begin{proposition}\label{prop.unfold}
Let $\Omega \subset \mathbb{R}^d$ be a bounded, Lipschitz domain. 
\noindent\begin{enumerate}[(i)]
\item For any function $u
\in L^1(\Omega)$, one has: 
\begin{equation}\label{eq.unfoldbdy}
\int_{\Omega}{u \:dx} = \int_{\Omega \times Y}{E_\varepsilon u \:dxdy} + \int_{{\mathcal B}_\varepsilon}{u \:dx}.
 \end{equation}
In particular, for $p\in [1,\infty]$, if $u_\varepsilon$ is a bounded sequence in $L^p(\Omega)$ and $v \in L^q(\Omega)$ is a fixed function, 
with $\frac{1}{p} + \frac{1}{q} = 1$, one has: 
\begin{equation}\label{cvTeps}
\left\lvert \int_{\Omega}{u_\varepsilon v \:dx} - \int_{\Omega \times Y}{E_\varepsilon u_\varepsilon \: E_\varepsilon v  \:dxdy} \right\lvert \stackrel{\varepsilon\rightarrow 0}{\longrightarrow} 0.
\end{equation}

\item For $p \in [1,\infty]$, both operators $E_\varepsilon: L^p(\Omega) \rightarrow L^p(\Omega \times Y)$ and $P_\varepsilon: L^p(\Omega \times Y) \rightarrow L^p(\Omega)$ are bounded with norm $1$.
\item If $p \in [1,\infty]$ and $\frac{1}{p} + \frac{1}{q} = 1$, $P_\varepsilon: L^{q}(\Omega \times Y) \rightarrow L^{q}(\Omega)$ is the adjoint of $E_\varepsilon: L^p(\Omega) \rightarrow L^p(\Omega \times Y)$. 
\item For $p \in [1,\infty]$, the operator $P_\varepsilon: L^p(\Omega \times Y) \rightarrow L^p(\Omega)$ is `almost' a left inverse for $E_\varepsilon: L^p(\Omega) \rightarrow L^p(\Omega \times Y)$: for any function $u \in L^p(\Omega)$, 
$$ P_\varepsilon E_\varepsilon u (x) = \left\{Ê
\begin{array}{cl}
u(x) & \text{if } x \in {\mathcal O}_\varepsilon, \\
0 & \text{otherwise}. 
\end{array}
\right.$$
\item For $p \in [1,\infty)$ and a given function $u \in L^p(\Omega)$, $E_\varepsilon u \to u$ strongly in $L^p(\Omega \times Y)$. 
\item For $p \in [1,\infty)$, and a given function $\psi \in {\mathcal D}(\Omega, L^p(Y))$, define 
$u_\varepsilon(x) 
:= \psi(x,\frac{x}{\varepsilon}) \in L^p(\Omega)$; then 
$E_\varepsilon u_\varepsilon 
\to \psi$ strongly in $L^p(\Omega \times Y)$.
\item For $p \in [1,\infty)$, and for any function $\phi \in L^p(\Omega \times Y)$, one has: 
$$ E_\varepsilon \: P_\varepsilon \phi \rightarrow \phi \text{ strongly in } L^p(\Omega \times Y). $$ 
\end{enumerate}
\end{proposition}
\begin{proof}  \textit{(i)} Let $u \in L^1(\Omega)$ be arbitrary; it follows from the definitions that: 
$$ 
\begin{array}{>{\displaystyle}c c>{\displaystyle}l}
\int_\Omega{u \:dx} &=& \int_{{\mathcal O}_\varepsilon}{u \:dx} + \int_{{\mathcal B}_\varepsilon}{u \:dx},\\
&=& \sum\limits_{\xi \in \Xi_\varepsilon}{\int_{\varepsilon(\xi + Y)}{u \:dx}}  + \int_{{\mathcal B}_\varepsilon}{u \:dx},\\ 
&=& \sum\limits_{\xi \in \Xi_\varepsilon}{\varepsilon^d \int_{Y}{u\left(\varepsilon \xi + \varepsilon y \right) \:dy}} +\int_{{\mathcal B}_\varepsilon}{u \:dx},\\
&=& \sum\limits_{\xi \in \Xi_\varepsilon}{\varepsilon^d \int_Y E_\varepsilon u (\xi, y) {\, dy}} + \int_{{\mathcal B}_\varepsilon}{u \:dx},
\end{array}
$$
which leads to the desired conclusion. 

The convergence (\ref{cvTeps}) follows from H\"older's inequality and Lebesgue's dominated convergence theorem applied to the last term in the right-hand side of (\ref{eq.unfoldbdy}).\\

\noindent \textit{(ii)} Follows directly from $(i)$.\\

\noindent \textit{(iii)} Let $u \in L^p(\Omega)$, and $\phi \in L^{q}(\Omega \times Y)$ be arbitrary functions. Then, 
$$ 
\begin{array}{>{\displaystyle}c c>{\displaystyle}l}
\int_{\Omega}{u(x)\: P_\varepsilon \phi(x) \:dx} &=& \sum\limits_{\xi \in \Xi_\varepsilon}{\varepsilon^d \int_{Y}{u \left( \varepsilon \xi + \varepsilon y\right) P_\varepsilon \phi \left( \varepsilon \xi + \varepsilon y\right) \:dy }}, \\
&=& \sum\limits_{\xi \in \Xi_\varepsilon}{\varepsilon^d \int_{Y}{\int_Y{E_\varepsilon u \left(\varepsilon\xi, y\right) \phi \left( \varepsilon \xi + \varepsilon z, y\right) \:dzdy }}},\\
&=& \sum\limits_{\xi \in \Xi_\varepsilon}{\int_{Y}{\int_{\varepsilon(\xi + Y)}{E_\varepsilon u \left(x, y\right) \phi \left( x, y\right) \:dx}dy }}, \\
&=& \int_{\Omega \times Y}{E_\varepsilon u (x,y) \phi(x,y)\:dxdy}.
\end{array}
$$

\noindent \textit{(iv)} Let $u \in L^p(\Omega)$ be given. Then, a simple computation reveals that, for a.e. $x\in {\mathcal O}_\varepsilon$, 
$$ 
P_\varepsilon E_\varepsilon u(x) = \int_Y{E_\varepsilon u \left(\varepsilon \left[\frac{x}{\varepsilon} \right] +\varepsilon z , \left\{ \frac{x}{\varepsilon} \right\} \right) \:dz}
= \int_Y{u \left(\varepsilon \left[\frac{x}{\varepsilon} \right] + \varepsilon \left\{ \frac{x}{\varepsilon} \right\} \right) \:dz},
$$
which is by definition $u(x)$. \\

\noindent \textit{(v)} Assume first that $u \in {\mathcal D}(\Omega)$. Then, for a.e. $(x,y) \in  {\mathcal O}_\varepsilon \times Y$, 
$$ \lvert E_\varepsilon u(x,y) - u(x)\lvert  =  \lvert u\left(\varepsilon\left[\frac{x}{\varepsilon}\right]_Y + \varepsilon y\right) - u(x) \lvert \leq C \varepsilon \lvert\nabla u(x) \lvert.$$
Integrating over ${\mathcal O}_\varepsilon \times Y$, then using $(i)$, it follows that $\lvert\lvert E_\varepsilon u - u \lvert\lvert_{L^p(\Omega \times Y)} \to 0$ as $\varepsilon \to 0$. 
The result for general $u \in L^p(\Omega)$ follows from a standard density argument, using the fact that $E_\varepsilon$ has operator norm $1$. \\

\noindent \textit{(vi)} The proof is almost identical to that of $(v)$, noticing that, for a.e. $(x,y) \in {\mathcal O}_\varepsilon \times Y$, 
$$ E_\varepsilon 
u_\varepsilon
(x,y ) = \psi\left(\varepsilon\left[\frac{x}{\varepsilon}\right]_Y + \varepsilon y,y\right).$$

\noindent \textit{(vii)} For a.e. $(x,y) \in \Omega \times Y$, one has:
$$ 
E_\varepsilon P_\varepsilon \phi(x,y) = P_\varepsilon \phi \left(\varepsilon \left[ \frac{x}{\varepsilon} \right] + \varepsilon y \right) 
= \int_Y{\phi \left(\varepsilon \left[ \frac{x}{\varepsilon} \right] + \varepsilon z, y\right) \:dz}.
$$
From this formula, it is easily seen that, if $\phi \in {\mathcal D}(\Omega \times Y)$, one has, for $\varepsilon$ small enough: 
$$ \lvert\lvert E_\varepsilon  P_\varepsilon \phi - \phi \lvert\lvert_{L^p(\Omega \times Y)} \leq C \varepsilon,$$
where $C$ only depends on a modulus of continuity for $\phi$. 
The general result follows by a density argument, again using the fact that $||E_{\varepsilon}|| \leq 1$ and
$||P_{\varepsilon}|| \leq 1$.
\end{proof}

The operators $E_\varepsilon$ and $P_\varepsilon$ enjoy several additional important 
properties, which are the transcription of the compactness properties of the two-scale convergence topology (see e.g. \cite{allaire2s,nguetseng}). 
We only quote the results, referring to \cite{unfold} for proofs. 

\begin{theorem}\label{thcompactunfold}
Let $p \in (1,\infty)$, and let $u_\varepsilon$ be a bounded sequence in $W^{1,p}(\Omega)$. Then, 
up to a subsequence, there exist $u_0 \in W^{1,p}(\Omega)$ and $\widehat{u} \in L^p(\Omega, W^{1,p}_\#(Y))$ such that: 
$$ u_\varepsilon \rightarrow u_0 \text{ weakly in } W^{1,p}(\Omega), \text{ and } E_\varepsilon(\nabla u_\varepsilon) \rightarrow \nabla u_0 + \nabla_y \widehat{u} \text{ weakly in }  L^p(\Omega \times Y).$$
\end{theorem}
\par
\smallskip

\subsubsection{A two-scale formulation of the Poincar\'e-Neumann variational problem}\label{2scaleNP}~\\

As mentionned previously, the operator $T_\varepsilon: H^1_0(\Omega) \rightarrow H^1_0(\Omega)$ fails to 
have `nice' convergence properties, amenable to the study of its asymptotic spectral properties. To improve this feature, 
we `unfold' $T_\varepsilon$ into an operator $\mathbb{T}_\varepsilon$, from $ L^2(\Omega, H^1(\omega) / \mathbb{R}) $ into itself, 
where $L^2(\Omega,H^1(\omega) / \mathbb{R})$ is equipped with the inner product
$$ \langle \phi, \psi \rangle_{L^2(\Omega,H^1(\omega)/\mathbb{R})} = \int_{\Omega \times \omega}{\nabla_y \phi (x,y) \cdot \nabla_y \psi(x,y) \:dxdy}, $$
and the associated norm.

This task relies on the extension and projection operators $E_\varepsilon$ and $P_\varepsilon$ introduced in the previous section
and demands a little caution, since the definition of $T_\varepsilon$ involves first-order derivatives of functions, 
operations that are not necessarily well-behaved with respect to 
$E_\varepsilon$ and $P_\varepsilon$ (for instance, if $u \in H^1(\Omega)$, $E_\varepsilon u$
is in general not in $H^1(\Omega \times Y)$). 

More precisely, to cope with this technical difficulty, we use the ideas of Section \ref{sec.rest}:
let us introduce the Hilbert space 
\begin{equation}\label{eq.defHeps}
H_\varepsilon = H^1(\omega_\varepsilon) / C(\omega_\varepsilon),\text{ where } C(\omega_\varepsilon) = \left\{ u \in H^1(\omega_\varepsilon), \:\exists \: c_\xi \in \mathbb{R}, \: u = c_\xi \text{ on } \omega_\varepsilon^\xi, \:\: \xi \in \Xi_\varepsilon \right\},
\end{equation}
equipped with the inner product $\langle \cdot, \cdot \rangle_{H_\varepsilon}$ and norm $\lvert\lvert \cdot \lvert\lvert_{H_\varepsilon}$ defined by (see (\ref{eq.innerprodHD})): 
$$ \langle u, v \rangle_{H_\varepsilon} = \int_{\omega_\varepsilon}{\nabla u \cdot \nabla v \:dx}, \text{ and } \lvert\lvert u \lvert\lvert_{H_\varepsilon}^2 = \langle u , u \rangle_{H_\varepsilon},\:\: u,v\in H_\varepsilon.$$
To simplify notations, we shall write in the same way a function in $H^1(\omega_\varepsilon)$ and its class in $H_\varepsilon$ in the forthcoming developments. 
We have seen that $T_\varepsilon$ induces an isomorphism $\mathring{T_\varepsilon} : H_\varepsilon \to H_\varepsilon$
which retains the spectral properties of $T_\varepsilon$, in the sense of Proposition \ref{propspecrescalTD}.

Let us now observe that: 
\begin{itemize}
\item The operator $E_\varepsilon: L^2(\Omega) \rightarrow L^2(\Omega \times Y)$ induces an operator (still denoted by $E_\varepsilon$) from 
$H_\varepsilon$ into $L^2(\Omega, H^1(\omega) / \mathbb{R})$, and: 
\begin{equation}\label{eq.gradteps}
 \forall u \in H_\varepsilon, \:\: \nabla_y(E_\varepsilon u) = \varepsilon E_\varepsilon(\nabla u),
 \end{equation}
\item The operator $P_\varepsilon : L^2(\Omega \times Y) \rightarrow L^2(\Omega)$ induces an operator (still denoted by $P_\varepsilon$) from $L^2(\Omega, H^1(\omega) / \mathbb{R})$ into $H_\varepsilon$ and: 
\begin{equation}\label{eq.gradueps}
 \forall \phi \in L^2(\Omega,H^1(\omega) / \mathbb{R}), \:\: \nabla (P_\varepsilon \phi) = \frac{1}{\varepsilon}P_\varepsilon(\nabla_y \phi).
 \end{equation}
\end{itemize}
Taking advantage of these facts, we now define the rescaled operator $ \mathbb{T}_\varepsilon = E_\varepsilon \mathring{T_\varepsilon} P_\varepsilon$
from $L^2(\Omega,H^1(\omega)/\mathbb{R})$ into itself.

\begin{lemma}
The operator $\mathbb{T}_\varepsilon$ is self-adjoint. 
\end{lemma}
\begin{proof}
We simply calculate, for arbitrary functions $\phi,\psi \in L^2(\Omega,H^1(\omega)/\mathbb{R})$, 
$$ \begin{array}{>{\displaystyle}cc>{\displaystyle}l}
\int_{\Omega \times \omega}{\nabla_y (\mathbb{T}_\varepsilon\phi)  \cdot \nabla_y \psi \:dxdy} &=& \int_{\Omega \times \omega}{\nabla_y (E_\varepsilon \mathring{T_\varepsilon} P_\varepsilon \phi)  \cdot \nabla_y \psi  \:dxdy},\\
&=& \varepsilon \int_{\Omega \times \omega}{E_\varepsilon \left( \nabla ( \mathring{T_\varepsilon} P_\varepsilon \phi) \right)  \cdot \nabla_y \psi  \:dxdy},\\
&=& \varepsilon \int_{\Omega \times Y}{E_\varepsilon \left( \nabla ( \mathring{T_\varepsilon} P_\varepsilon \phi) \right)  \cdot (\mathds{1}_\omega(y)\nabla_y \psi)  \:dxdy},
\end{array}$$
where $\mathds{1}_\omega$ is the characteristic function of the inclusion $\omega$. Then,
$$ \begin{array}{>{\displaystyle}cc>{\displaystyle}l}
\int_{\Omega \times \omega}{\nabla_y (\mathbb{T}_\varepsilon\phi)  \cdot \nabla_y \psi \:dxdy} 
&=& \varepsilon \int_{\Omega}
{P_\varepsilon (\mathds{1}_\omega) \:  \left( \nabla ( \mathring{T_\varepsilon} P_\varepsilon \phi) \right) 
\cdot P_\varepsilon \left( \nabla_y \psi\right)  \:dx},\\
&=& \varepsilon \int_{\omega_\varepsilon}{ \nabla ( \mathring{T_\varepsilon} P_\varepsilon \phi)  \cdot P_\varepsilon \left( \nabla_y \psi \right) \:dx},\\
&=& \varepsilon^2 \int_{\omega_\varepsilon}{ \nabla ( \mathring{T_\varepsilon} P_\varepsilon \phi)  \cdot  \nabla \left(P_\varepsilon \psi \right) \:dx}.
\end{array}
$$
In the above equalities, we have used (\ref{eq.gradteps}), (\ref{eq.gradueps}) and Proposition \ref{prop.unfold} $(iii)$. 
Now, using the self-adjointness of $\mathring{T}_\varepsilon$ (see Proposition \ref{prop.TDring}), 
and the exact same calculations as above, we obtain:
$$ \begin{array}{>{\displaystyle}cc>{\displaystyle}l}
\int_{\Omega \times \omega}{\nabla_y (\mathbb{T}_\varepsilon \phi)  \cdot \nabla_y \psi \:dxdy} &=& \varepsilon^2 \int_{\omega_\varepsilon}{ \nabla ( P_\varepsilon \phi)  \cdot  \nabla \left(  \mathring{T_\varepsilon} P_\varepsilon \psi \right) \:dx},\\
&=&  \varepsilon \int_{\omega_\varepsilon}{P_\varepsilon( \nabla_y  \phi)  \cdot  \nabla \left(  \mathring{T_\varepsilon} P_\varepsilon \psi \right) \:dx},\\
&=&  \varepsilon \int_{\Omega \times \omega}{ \nabla_y  \phi  \cdot E_\varepsilon\left( \nabla \left(  \mathring{T_\varepsilon}P_\varepsilon \psi \right)\right) \:dxdy},\\
&=&  \int_{\Omega \times \omega}{ \nabla_y  \phi \cdot \nabla_y ( \mathbb{T}_\varepsilon  \psi ) \:dxdy},
\end{array}
$$
whence the desired result.
\end{proof}
We now verify that $\mathbb{T}_\varepsilon$ retains the same spectrum as $T_\varepsilon$.

\begin{proposition}\label{propsamespec1}
For given $\varepsilon >0$, the spectrum of the operator $\mathbb{T}_\varepsilon: L^2(\Omega,H^1(\omega)/\mathbb{R})\to  L^2(\Omega,H^1(\omega)/\mathbb{R})$ equals: 
$$ \sigma(\mathbb{T}_\varepsilon) = \sigma(T_\varepsilon) \setminus \left\{Ê0 \right\}.$$ 
Moreover, $\lambda \neq 0$ is an eigenvalue of $T_\varepsilon$ if and only if it is an eigenvalue of $\mathbb{T}_\varepsilon$.  
\end{proposition}
\begin{proof}
Let $\lambda \in \sigma(T_\varepsilon)$, $\lambda \neq 0$. 
By Proposition \ref{propspecrescalTD}, $\lambda$ belongs to $\sigma(\mathring{T_\varepsilon})$, 
and there exists a sequence $u_n \in H^1_0(\Omega)$ such that: 
$$ \lvert\lvert \nabla (T_\varepsilon u_n) - \lambda \nabla u_n \lvert\lvert_{L^2(\omega_\varepsilon)^d} \xrightarrow{n \rightarrow 0} 0, \text{ and } \lvert\lvert \nabla u_n \lvert\lvert_{L^2(\omega_\varepsilon)^d} = 1.$$
Set $\phi_n = E_\varepsilon u_n \in L^2(\Omega,H^1(\omega) /\mathbb{R})$;
it comes: 
$$ \int_{\Omega \times \omega}{\lvert \nabla_y (\mathbb{T}_\varepsilon \phi_n) - \lambda \nabla_y\phi_n \lvert^2 \:dxdy} = \varepsilon^2 \int_{\omega_\varepsilon}{\lvert 
\nabla(T_\varepsilon u_n) - \lambda \nabla u_n \lvert^2 \:dx} \xrightarrow{n \to \infty} 0,$$
where we have used Proposition \ref{prop.unfold} and the identity, for $u \in H_\varepsilon$,
\begin{eqnarray} \label{est_norms}
\int_{\Omega \times \omega} \lvert E_\varepsilon \nabla u \lvert^2 \,dxdy
&=&
\int_{{\mathcal O}_\varepsilon \times \omega}
\lvert E_\varepsilon \nabla u \lvert^2 \,dxdy
\;=\;
\sum\limits_{\xi \in \Xi_\varepsilon}
\int_{\varepsilon(\xi + Y) \times \omega}
\lvert E_\varepsilon \nabla u \lvert^2 \,dxdy
\nonumber \\
&=&
\sum\limits_{\xi \in \Xi_\varepsilon}
\varepsilon^d \int_{\omega} \lvert \nabla u(\varepsilon \xi + \varepsilon y) \lvert^2 \,dy
\;=\;
\int_{\omega_\varepsilon} |\nabla u|^2 \,dx.
\end{eqnarray}
Moreover, 
$$\lvert\lvert \phi_n \lvert\lvert^2_{L^2(\Omega,H^1(\omega)/\mathbb{R})} = \int_{\Omega \times \omega}{\lvert \nabla_y \phi_n\lvert^2 \:dxdy} = \varepsilon^2 \int_{\Omega \times \omega}{\lvert E_\varepsilon(\nabla u_n )\lvert^2 \:dxdy} = \varepsilon^2; $$ 
since $\phi_n$ is bounded away from $0$ in $L^2(\Omega,H^1(\omega) / \mathbb{R})$ as $n \to \infty$, we conclude that $\lambda \in \sigma(\mathbb{T}_\varepsilon)$.\\

Let now $\lambda \in \sigma(\mathbb{T}_\varepsilon) \setminus \left\{ 0 \right\}$;
there exists a sequence $\phi_n \in L^2(\Omega, H^1(\omega) / \mathbb{R})$, such that: 
\begin{equation}\label{propcvTun}
\lvert\lvert \nabla_y \phi_n \lvert\lvert_{L^2(\Omega \times \omega)} = 1, \text{ and } \int_{\Omega \times \omega}{\lvert \nabla_y(\mathbb{T}_\varepsilon\phi_n - \lambda \phi_n) \lvert^2 \:dxdy} \xrightarrow{n \rightarrow \infty} 0. 
 \end{equation}
Let us define the sequence $u_n \in H^1_0(\Omega)$ by: 
$$ u_n = P_\varepsilon \phi_n - w_n \text{ on } \omega_\varepsilon, \text{ and } -\Delta u_n = 0 \text{ on } \Omega \setminus \overline{\omega_\varepsilon},$$
in which $P_\varepsilon \phi_n$ stands for any element of the equivalence class $P_\varepsilon \phi_n \in H_\varepsilon$, 
and $w_n \in C(\omega_\varepsilon)$ is a function which is constant in each connected component of $\omega_\varepsilon$, yet to be chosen. 
Using Proposition \ref{prop.unfold}, $(ii)$ yields: 
$$ \begin{array}{>{\displaystyle}cc>{\displaystyle}l}
\int_{\Omega \times \omega}{\lvert \nabla_y ( \mathbb{T}_\varepsilon \phi_n - \lambda \phi_n ) \lvert^2 \:dxdy} & \geq & \int_{\omega_\varepsilon}{\lvert P_\varepsilon(\nabla_y (\mathbb{T}_\varepsilon \phi_n) - \lambda \nabla_y \phi_n) \lvert^2\:dx}, \\
&=& \varepsilon^2  \int_{\omega_\varepsilon}{\lvert  \nabla (T_\varepsilon u_n) - \lambda  \nabla u_n \lvert^2\:dx}.
\end{array}$$  
Now, using the boundedness of the harmonic extension operator, we obtain: 
$$
\begin{array}{ccl}
 \lvert\lvert \nabla (T_\varepsilon u_n - \lambda u_n) \lvert\lvert_{L^2(\Omega)^d} &\leq& C \lvert\lvert T_\varepsilon u_n - \lambda u_n \lvert\lvert_{H^1(\omega_\varepsilon)},\\
 &\leq & C\lvert\lvert \nabla (T_\varepsilon u_n) - \lambda \nabla u_n \lvert\lvert_{L^2(\omega_\varepsilon)^d},
 \end{array}$$
  provided $w_n$ is chosen in such a way that $\int_{\omega_\varepsilon^\xi}{(T_\varepsilon u_n - \lambda u_n) \:dx}  = 0$ for all $\xi \in \Xi_\varepsilon$ (owing to 
  the Poincar\'e Wirtinger inequality). 
 Note that, in the above inequality, the constant $C$ may depend on $\varepsilon$, which has a fixed value in the present proof.  
  Therefore,  
\begin{equation}\label{eq.cvweylTeps}
 \lvert\lvert \nabla (T_\varepsilon u_n - \lambda u_n) \lvert\lvert_{L^2(\Omega)}^2 \xrightarrow{n \rightarrow \infty} 0.
 \end{equation}
To complete the proof of the claim, we only need to
show that the sequence $u_n$ does not converge to $0$ in $H^1_0(\Omega)$ as $n \to \infty$. 
But this follows easily from (\ref{eq.cvweylTeps}): since $\lambda \neq 0$, if $u_n$ were to converge to $0$ in $H^1_0(\Omega)$, then so would $T_\varepsilon u_n$; 
in turn $E_\varepsilon T_\varepsilon u_n = \mathbb{T}_\varepsilon \phi_n$ would converge to $0$ in $L^2(\Omega,H^1(\omega)/\mathbb{R})$, which 
is impossible in view of (\ref{propcvTun}). The claim follows.
 
The last statement of Proposition \ref{propsamespec1}, concerning eigenvalues, is proved in an analogous way.  
\end{proof}

\subsubsection{Pointwise convergence of the rescaled operator $\mathbb{T}_\varepsilon$.}\label{sec.step3rescal}~\\

The main result in this section states that $\mathbb{T}_\varepsilon$ converges pointwise, albeit possibly not in operator norm,
to a limit operator $\mathbb{T}_0 : L^2(\Omega,H^1(\omega) / \mathbb{R}) \rightarrow L^2(\Omega,H^1(\omega) / \mathbb{R})$. 

\begin{proposition}\label{prop.cvrescal1}
The operator $\mathbb{T}_\varepsilon$ converges pointwise to a limit operator $\mathbb{T}_0$, mapping $L^2(\Omega,H^1(\omega) / \mathbb{R})$ into itself: 
for any function $\phi \in L^2(\Omega, H^1(\omega) /\mathbb{R})$, the following convergence holds strongly in $L^2(\Omega,H^1(\omega) / \mathbb{R})$: 
$$ \mathbb{T}_\varepsilon \phi  \:\: \xrightarrow{\varepsilon \rightarrow 0} \:\:Ê\mathbb{T}_0\phi. $$
The operator $\mathbb{T}_0$ is defined by 
$$\mathbb{T}_0\phi(x,y) = Q \left(\nabla v_0(x) \cdot y + \widehat{v}(x,y) \right),$$
where 
\begin{itemize}
\item $Q : L^2(\Omega, H^1(Y)) \rightarrow L^2(\Omega,H^1(\omega) /\mathbb{R})$ is the natural restriction operator, 
\item $\widehat{v}$ is the unique solution in $L^2(\Omega,H^1_\#(Y) /\mathbb{R})$ to the following equation: 
\begin{equation}\label{eq.defvhat}
-\Delta_y \widehat{v} (x,y) = -\text{\rm div}_y(\mathds{1}_\omega(y) \nabla_y \phi)(x,y)  \text{ in } H^1_\#(Y), \text{ a.e. in } x \in \Omega,
\end{equation}
\item $v_0$ is the unique solution in $H^1_0(\Omega)$ to the equation:
\begin{equation}\label{eq.defv0}
 -\Delta v_0 = -\text{\rm div} \left(\int_\omega{\nabla_y \phi(x,y)\:dy} \right).
 \end{equation}
\end{itemize}
\end{proposition}
\begin{proof}
Let $\phi \in L^2(\Omega,H^1(\omega) /\mathbb{R})$ be given, and define $u_\varepsilon $ as (any element in the class) $P_\varepsilon \phi \in H_\varepsilon$. 
The function $v_\varepsilon := T_\varepsilon u_\varepsilon \in H^1_0(\Omega)$ fulfills the following variational problem: 
\begin{equation}\label{varfveps} 
\forall v \in H^1_0(\Omega), \:\: \int_\Omega{\nabla v_\varepsilon \cdot \nabla v \:dx} = \int_{\omega_\varepsilon}{\nabla u_\varepsilon \cdot \nabla v\:dx},
\end{equation}
and it follows from the definitions that $\mathbb{T}_\varepsilon\phi = Q E_\varepsilon v_\varepsilon$. 
Standard energy estimates imply: 
$$\begin{array}{>{\displaystyle}c c>{\displaystyle}l}
 \lvert\lvert \nabla v_\varepsilon  \lvert\lvert_{L^2(\Omega)} &\leq &C \lvert \lvert \nabla u_\varepsilon \lvert\lvert_{L^2(\omega_\varepsilon)},\\
&\leq & \frac{C}{\varepsilon} \lvert\lvert \nabla_y \phi \lvert\lvert_{L^2(\Omega \times \omega)}.
\end{array}$$

Then, using Theorem \ref{thcompactunfold}, we infer that, up to a subsequence still indexed by $\varepsilon$, there exist $v_0 \in H^1_0(\Omega)$ and $\widehat{v} \in L^2(\Omega, H^1_\#(Y) / \mathbb{R})$ such that the following convergences hold: 
\begin{equation}\label{wkcvveps}
\varepsilon  v_\varepsilon \rightarrow v_0 \: \text{ weakly in }\: H^1_0(\Omega), \:\: \text{ and } \:\: \nabla_y(E_\varepsilon v_\varepsilon) = \varepsilon E_\varepsilon (\nabla v_\varepsilon) \rightarrow \nabla v_0 + \nabla_y \widehat{v} \:\text{ weakly in }\: L^2(\Omega\times Y).
\end{equation}
The identity (\ref{varfveps}) may be rewritten as: 
\begin{equation*}
\forall v \in H^1_0(\Omega), \:\: \int_\Omega{\varepsilon \nabla v_\varepsilon \cdot \nabla v \:dx} = \int_{\omega_\varepsilon}{P_\varepsilon(\nabla_y \phi) \cdot \nabla v\:dx},
\end{equation*}
and, using Proposition \ref{prop.unfold}, it can be rescaled into: 
\begin{equation}\label{varfvepsrescal}
\forall v \in H^1_0(\Omega), \:\: \int_{\Omega \times Y}{\varepsilon E_\varepsilon(\nabla v_\varepsilon) \cdot E_\varepsilon(\nabla v) \:dxdy} + r_\varepsilon = \int_{\Omega \times \omega}{E_\varepsilon P_\varepsilon(\nabla_y \phi) \cdot E_\varepsilon(\nabla v)\:dxdy},
\end{equation}
where the remainder $r_\varepsilon$ is defined by $r_\varepsilon := \int_{{\mathcal B}_\varepsilon}{\varepsilon \nabla v_\varepsilon \cdot \nabla v\:dx}$.
Let us now consider a test function of the form $v(x) = \varphi(x) + \varepsilon \psi\left(x,\frac{x}{\varepsilon}\right)$, 
where $\varphi \in H^1_0(\Omega)$, $\psi \in {\mathcal D}(\Omega, H^1_\#(Y))$. Using Proposition \ref{prop.unfold}, $(v)$, $(vi)$, it comes: 
$$ E_\varepsilon (\nabla v) \longrightarrow \nabla \varphi +  \nabla_y \psi \text{ strongly in }  L^2(\Omega \times Y), \text{ and } r_\varepsilon \to 0.$$

Thus, taking limits in (\ref{varfvepsrescal}), using the weak convergence (\ref{wkcvveps}), and Proposition \ref{prop.unfold}, $(vii)$, it follows that: 
\begin{equation}\label{eq.v0vhat}
 \int_{\Omega \times Y}{(\nabla v_0 + \nabla_y \widehat{v}) \cdot (\nabla \varphi + \nabla_y \psi ) \:dxdy} = \int_{\Omega \times \omega}{\nabla_y \phi \cdot (\nabla \varphi + \nabla_y \psi ) \:dxdy},
 \end{equation}
and a standard density argument reveals that (\ref{eq.v0vhat}) actually holds for arbitrary $\varphi \in H^1_0(\Omega)$, $\psi \in L^2(\Omega, H^1_\#(Y))$.
This two-scale variational formulation allows to identify the functions $v_0$ and $\widehat{v}$: taking $\varphi = 0$ and arbitrary $\psi$ in (\ref{eq.v0vhat}), 
it follows that $\widehat{v}$ satisfies the equation: 
$$ -\Delta_y \widehat{v}(x,y) = -\text{\rm div}_y(\mathds{1}_\omega(y) \nabla_y \phi)(x,y) \text{ a.e. in } \Omega \times Y,$$
i.e. $\widehat{v}$ is the function defined in (\ref{eq.defvhat}). 
Likewise, taking arbitrary $\varphi \in H^1_0(\Omega)$ and $\psi = 0$ in (\ref{eq.v0vhat}) yields: 
$$ -\Delta v_0 = -\text{\rm div}\left( \int_\omega{ \nabla_y \phi(x,y)\:dy}\right),$$
i.e. $v_0$ is the function defined in (\ref{eq.defv0}).

To complete the proof, we now 
show that the convergence $\nabla_y(E_\varepsilon v_\varepsilon) \to \nabla v_0 + \nabla_y \widehat{v}$ 
holds strongly in $L^2(\Omega \times Y)$. We 
start by proving the convergence of energies: 
\begin{equation}\label{rescal.cvnrj}
\lim\limits_{\varepsilon \rightarrow 0}{\int_{\Omega}{\lvert \varepsilon \nabla v_\varepsilon \lvert^2\:dx}} = \lim\limits_{\varepsilon \rightarrow 0}{\int_{\Omega \times Y}{\lvert\nabla_y (E_\varepsilon v_\varepsilon) \lvert^2\:dxdy}}  = \int_{\Omega \times Y}{\lvert \nabla v_0 + \nabla_y \widehat{v} \lvert^2 \:dxdy}.
\end{equation}

To achieve (\ref{rescal.cvnrj}), remark that, by lower semi-continuity of the weak convergence, 
$$ \int_{\Omega \times Y}{\lvert \nabla v_0 + \nabla_y \widehat{v} \lvert^2 \:dxdy} \leq \liminf\limits_{\varepsilon\rightarrow 0}{\int_{\Omega\times Y}{\lvert\nabla_y (E_\varepsilon v_\varepsilon) \lvert^2\:dxdy}}.$$
Now, we have successively: 
$$\begin{array}{>{\displaystyle}cc>{\displaystyle}l}
\liminf\limits_{\varepsilon\rightarrow 0}{\int_{\Omega\times Y}{\lvert\nabla_y (E_\varepsilon v_\varepsilon) \lvert^2\:dxdy}} &=& \liminf\limits_{\varepsilon\rightarrow 0}{ \varepsilon^2 \int_{{\mathcal O}_\varepsilon}{\lvert \nabla v_\varepsilon \lvert^2 \:dx}}, \\
&\leq & \liminf\limits_{\varepsilon\rightarrow 0}{ \varepsilon^2 \int_{\Omega}{\lvert \nabla v_\varepsilon \lvert^2 \:dx}},\\
&= & \liminf\limits_{\varepsilon\rightarrow 0}{ \varepsilon^2 \int_{\omega_\varepsilon}{ \nabla u_\varepsilon \cdot \nabla v_\varepsilon \:dx}},\\
& = & \liminf\limits_{\varepsilon\rightarrow 0}{ \varepsilon \int_{\omega_\varepsilon}{ P_\varepsilon \left( \nabla_y \phi \right) \cdot \nabla v_\varepsilon \:dx}},\\
& = & \lim\limits_{\varepsilon\rightarrow 0}{ \int_{\Omega \times \omega}{ E_\varepsilon P_\varepsilon \left( \nabla_y \phi \right) \cdot \nabla_y (E_\varepsilon v_\varepsilon ) \:dxdy}},\\
&=& \int_{\Omega \times \omega}{\nabla_y \phi \cdot (\nabla v_0+ \nabla_y \widehat{v}) \:dxdy},
\end{array} $$
where we have used Proposition \ref{prop.unfold}, $(vii)$, and the weak convergence (\ref{wkcvveps}). Considering the variational equation (\ref{eq.v0vhat}), this proves (\ref{rescal.cvnrj}).
We now calculate the discrepancy: 
$$\begin{array}{>{\displaystyle}c c>{\displaystyle}l}
 \int_{\Omega \times Y}{\lvert \nabla_y(E_\varepsilon v_\varepsilon)- \nabla v_0 -\nabla_y \widehat{v} \lvert^2\:dxdy} &= & \int_{\Omega\times Y}{ \lvert \nabla_y(E_\varepsilon v_\varepsilon) \lvert^2  \:dxdy} - 2 \int_{\Omega\times Y}{\nabla_y(E_\varepsilon v_\varepsilon) \cdot (\nabla v_0 + \nabla_y \widehat{v})  \:dxdy}\\
 && +  \int_{\Omega\times Y}{ \lvert \nabla v_0 + \nabla_y \widehat{v} \lvert^2 \:dxdy}.
 \end{array}$$
Using the convergence of energies (\ref{rescal.cvnrj}) for the first term in the right-hand side of the above identity and the weak convergence (\ref{wkcvveps}) for the second one, we obtain that the difference $ \int_{\Omega \times Y}{\lvert \nabla_y(E_\varepsilon v_\varepsilon)- \nabla v_0 -\nabla_y \widehat{v} \lvert^2\:dxdy}$ vanishes. 
Putting definitions together, we have eventually proved that: 
$$ \mathbb{T}_\varepsilon \phi = Q E_\varepsilon v_\varepsilon \xrightarrow{\varepsilon \to 0} Q\left(\nabla v_0 \cdot y + \widehat{v} \right), \text{ strongly in } L^2(\Omega,H^1(\omega)/\mathbb{R}),$$
which terminates the proof.
\end{proof}~\par
\smallskip

\subsubsection{End of the proof of Theorem \ref{th.cellspec}}~\\
 
 Theorem \ref{th.cellspec} now arises as a simple consequence of Propositions \ref{prop.speclsc} and \ref{prop.cvrescal1}, 
 and of the following partial identification of the spectrum of $\mathbb{T}_0$ in terms of that of a simpler operator. 

\begin{lemma}\label{lem.loc}
Let $T_0 : H^1_\#(Y) / \mathbb{R} \to H^1_\#(Y) / \mathbb{R}$ be the operator that maps
$u \in H^1_\#(Y) / \mathbb{R}$ into the unique solution $T_0 u \in H^1_\#(Y)/\mathbb{R}$ of the equation:
$$ -\Delta_y (T_0 u)  = -\text{\rm div}_y(\mathds{1}_\omega \nabla_y u),$$
or under variational form: 
\begin{equation}\label{eq.varfT0}
 \forall v \in H^1_\#(Y) / \mathbb{R}, \:\: \int_Y{\nabla_y (T_0 u) \cdot \nabla_y v \:dy} = \int_\omega{\nabla_y u \cdot \nabla_y v \:dy}. 
 \end{equation}
Then the spectrum $\sigma(\mathbb{T}_0)$ contains the set $\sigma(T_0)$. 
\end{lemma}
\begin{proof}
First observe that, adaptating the general setting of Section \ref{sec.NPTD} to the case of the operator $T_0$, we see that 
$\sigma(T_0)\subset [0,1]$ simply consists of a sequence of eigenvalues with $\frac{1}{2}$ as unique accumulation point. 

Let now $\lambda \in \sigma(T_0)$; since it is easily seen that $0 \in \sigma(\mathbb{T}_0)$, 
we may assume that $\lambda \neq 0$ is an eigenvalue of $T_0$. Then, there exists a function $u \in H^1_\#(Y)/\mathbb{R}$, $u \neq 0$, such that
\begin{equation}\label{eq.varfeigenvT0}
 \forall v \in H^1_\#(Y) / \mathbb{R}, \:\: \lambda \int_Y{\nabla_y u \cdot \nabla_y v \:dy} = \int_\omega{\nabla_y u \cdot \nabla_y v \:dy},
 \end{equation}
and we define $\phi \in L^2(\Omega, H^1_\#(Y)/\mathbb{R})$ by $\phi(x,y) = u(y)$.
Using (\ref{eq.varfeigenvT0}), one sees that 
\begin{equation}\label{eq.specT0T0}
 \lambda \int_{\Omega \times Y}{\nabla_y \phi \cdot \nabla_y \psi \:dxdy} = \int_{\Omega \times \omega}{\nabla_y \phi \cdot \nabla_y \psi \:dxdy}
 \end{equation}
holds for any test function $\psi \in L^2(\Omega, H^1_\#(Y)/\mathbb{R})$ of the form $\psi(x,y) = \sum_{j=1}^N{\varphi_j(x) v_j(y)}$, where the
$\varphi_j \in L^2(\Omega)$ are the characteristic functions of some measurable subsets of $\Omega$ and the $v_j$ are arbitrary elements in $H^1_\#(Y)/\mathbb{R}$.
Hence, it follows from a usual density argument that (\ref{eq.specT0T0}) holds for all $\psi \in L^2(\Omega, H^1_\#(Y)/\mathbb{R})$.

Let us now recall that $\mathbb{T}_0 (Q\phi) \in L^2(\Omega,H^1(\omega)/\mathbb{R})$ is defined by $\mathbb{T}_0 (Q\phi) = Q(\nabla v_0 \cdot y + \widehat{v})$, 
where $Q$, $v_0$ and $\widehat{v}$ are as in in Proposition \ref{prop.cvrescal1}. 
In particular, $\widehat{v}$ is the unique element in $L^2(\Omega,H^1_\#(Y)/\mathbb{R})$ which satisfies (see (\ref{eq.defvhat})): 
$$ \forall \psi \in L^2(\Omega,H^1_\#(Y)/\mathbb{R}), \:\: \int_{\Omega \times Y}{\nabla_y \widehat{v} \cdot 
\nabla_y \psi \:dxdy} = \int_{\Omega \times \omega}{\nabla_y \phi \cdot \nabla_y \psi \:dxdy}.$$
Hence, $\widehat{v} = \lambda \phi$, and, using the definition (\ref{eq.defv0}), it follows that $v_0 = 0$. 
Thus, we have proved that $\mathbb{T}_0 (Q\phi) = \lambda Q\phi$. In addition, it is easily verified from (\ref{eq.specT0T0}) and the fact that $\lambda \neq 0$ that $Q\phi \neq 0$ in $L^2(\Omega,H^1(\omega)/\mathbb{R})$, i.e. $\lambda$ is an eigenvalue of $\mathbb{T}_0$. 
\end{proof}

\begin{remark}
Roughly speaking, $\sigma(T_0)$ captures the modes of $\mathbb{T}_0$ which are `self-resonating':
in the above proof, we have indeed seen that these have at least one associated eigenfunction which is independent of the macroscopic variable $x \in \Omega$. 
\end{remark}\par
\smallskip 

%%%%%%%%%%%%%%%%%%%%%%%%%%%%%
\subsection{Collective resonances of the inclusions}~\\
%%%%%%%%%%%%%%%%%%%%%%%%%%%%%
\subsubsection{Extension of the previous work to a rescaling procedure over packs of cells}~\\

The work carried out in Section \ref{th.cellspec} allows to capture that part of the limit spectrum $\lim_{\varepsilon\rightarrow 0}{\sigma(T_\varepsilon)}$,
which measures the self-resonance frequencies of the inclusion $\omega$. 
However, as $\varepsilon \to 0$, it may happen that several inclusions come into resonance together, 
thus giving rise to many different limit modes $\lambda \in \lim_{\varepsilon\rightarrow 0}{\sigma(T_\varepsilon)}$. To capture these, 
we rely on the idea of \textit{homogenization by packets} put forth by J. Planchard in \cite{planchard}, then used by G. Allaire and C. Conca \cite{allaireconcafs,allaireconca}.
We rescale the operator $T_\varepsilon$ over a collection of $K^d$ copies of the unit periodicity cell $Y$, where $K$ is an arbitrary integer.
This new rescaling procedure is quite similar to that described in Section \ref{2scaleNP}, and we only outline the main features, emphasizing the differences between both settings.

The ambient space $\mathbb{R}^d$ is now divided into copies of $KY$. We accordingly set the notations: 
$$\Xi_\varepsilon^K := \left\{Ê\xi \in \mathbb{Z}^d, \: \varepsilon K(\xi + Y) \Subset \Omega \right\}, \text{ and } {\mathcal O}_\varepsilon^K \text{ is the interior of the set } \bigcup\limits_{\xi \in \Xi_\varepsilon^K }{\overline{\varepsilon K(\xi + Y)}}.$$  
We also denote by:
$$ \omega^K := \bigcup\limits_{j \in \mathbb{N}^d \atop 0\leq j \leq K-1}{\omega_j^K} \subset KY, \text{ where } \omega_j^K := j+ \omega \subset KY,$$
the reunion of $K^d$ copies of the pattern inclusion $\omega$ in $KY$. 

In the above definition, and throughout this article, the indexation $0 \leq j \leq K-1$ is a shorthand for: $0 \leq j_i \leq K-1$, $i=1,...,d$, when 
$j =(j_1,...,j_d) \in \mathbb{N}^d$ is a multi-index with size $d$.

For a given point $x \in \mathbb{R}^d$, 
there exists a unique pair $(\xi,y) \in \mathbb{Z}^d \times [0,K)^d$ such that $x= \varepsilon K \xi + \varepsilon y$, and we denote: 
$$ \xi = \left[ \frac{x}{\varepsilon}\right]_{KY} , \text{ and } y = \left\{Ê\frac{x}{\varepsilon}\right\} _{KY}.$$ 
Observe that, with these definitions, ${\mathcal O}_\varepsilon^K$ is a strict subset of ${\mathcal O}_\varepsilon$, as defined in (\ref{th.cellspec}); 
the difference between both sets is the reunion of the $\varepsilon$-cells which do not intersect $\partial \Omega$, 
but whose associated $\varepsilon K$-cell does intersect $\partial \Omega$.

\begin{definition}
Let $p \in [1,\infty]$. 
\begin{itemize}
\item The extension operator over $K^d$ cells $ E_\varepsilon^K: L^p(\Omega) \rightarrow L^p(\Omega \times KY)$ is defined by: 
$$ E_\varepsilon^K u(x,y) = \left\{Ê
\begin{array}{cl}
u^e(\varepsilon K \left[ \frac{x}{\varepsilon} \right]_{KY} + \varepsilon y) & \text{if } x \in {\mathcal O}_\varepsilon,\\
0 & \text{otherwise}
\end{array}
\right. , \:\: x \in \Omega, \: y \in KY,$$
where $u^e(x) = u(x)$ if $x \in {\mathcal O}_\varepsilon$, and $0$ 
if $x \in \mathbb{R}^d \setminus \overline{{\mathcal O}_\varepsilon}$.

\item The projection operator over $K^d$ cells $ P_\varepsilon^K: L^p(\Omega \times KY) \rightarrow L^p(\Omega)$ is defined by:
$$P_\varepsilon^K \phi(x) = \left\{ \begin{array}{>{\displaystyle}cl}
\frac{1}{K^d}\int_{KY}{\phi^e\left(\varepsilon K  \left[Ê\frac{x}{\varepsilon}\right] _{KY} + \varepsilon z , \left\{Ê\frac{x}{\varepsilon}\right\} _{KY}\right)\:dz}  & \text{if } x \in {\mathcal O}_\varepsilon,\\
0 & \text{otherwise}
\end{array}
\right., \:\: x \in \Omega,$$
where $\phi^e(x,y) = \phi(x,y)$ if $x \in {\mathcal O}_\varepsilon$ and $0$ otherwise.
\end{itemize}
\end{definition}

\begin{remark}
Let us stress that $E_\varepsilon^K$ and $P_\varepsilon^K$ are not exactly the straightforward generalizations 
of $E_\varepsilon$ and $P_\varepsilon$ to the case of a rescaling over a set of $K^d$ cells (instead of one single cell).
Indeed, there are some packs of $K^d$ cells which intersect $\partial \Omega$ (and are thus excluded from 
${\mathcal O}_\varepsilon^K$) and nevertheless contain inclusions $\omega_\varepsilon^\xi$, for some $\xi \in \Xi_\varepsilon$.
The straightforward generalizations  of $E_\varepsilon$ and $P_\varepsilon$ to the present context would vanish on $\Omega \setminus 
\overline{{\mathcal O}_\varepsilon^K}$; this would entail a loss of the information contained in the inclusions of ${\mathcal O}_\varepsilon \setminus \overline{{\mathcal O}_\varepsilon^K}$, and in particular, Proposition \ref{prop.specmatchK} would not hold.
\end{remark} 

The operators $E_\varepsilon^K$ and $P_\varepsilon^K$ enjoy the counterpart properties of those of Proposition \ref{prop.unfold}. 
For the sake of convenience, we state the results without proof.

\begin{proposition}\label{prop.unfold2}
Let $\Omega \subset \mathbb{R}^d$ be a bounded, Lipschitz domain. 
\noindent\begin{enumerate}[(i)]
\item For an arbitrary function $u \in L^1(\Omega)$, one has: 
\begin{equation*}
\int_{\Omega}{u \:dx} = \frac{1}{K^d} \int_{\Omega \times KY}{E_\varepsilon^K u \:dxdy} + \int_{{\mathcal B}_\varepsilon}{u \:dx}.
 \end{equation*}
\item If $p \in [1,\infty]$ and $\frac{1}{p} + \frac{1}{q} = 1$, $P_\varepsilon^K: L^{q}(\Omega \times KY) \rightarrow L^{q}(\Omega)$ is the adjoint of $E_\varepsilon^K: L^p(\Omega) \rightarrow L^p(\Omega \times KY)$. 
\item For $p \in [1,\infty]$, and for any function $u \in L^p(\Omega)$: 
$$ P_\varepsilon^K E_\varepsilon^K u (x) = \left\{Ê
\begin{array}{cl}
u(x) & \text{if } x \in {\mathcal O}_\varepsilon, \\
0 & \text{otherwise}. 
\end{array}
\right.$$
\item For $p \in [1,\infty)$, and for any function $u \in L^p(\Omega \times KY)$, 
$$ E^K_\varepsilon \: P_\varepsilon^K u \rightarrow u \text{ strongly in } L^p(\Omega \times KY). $$ 
\end{enumerate}
\end{proposition}

We are now in position to rescale the operator $\mathring{T_\varepsilon}$, from $H_\varepsilon$ into itself (see (\ref{eq.defHeps})), 
in a way that allows to take into account the interactions between several inclusions. 
Let us define the subspace $C(\omega^K) \subset H^1(\omega^K)$: 
$$ C(\omega^K) = \left\{ u \in H^1(\omega^K), \:\: \exists\: c_j  \in \mathbb{R}, u = c_j \text{ on } \omega^K_j, \:\: j \in \mathbb{N}^d, \: 0\leq j \leq K-1\right\},$$
and consider the quotient space $H^K := H^1(\omega^K) / C(\omega^K)$. 
The desired rescaled operator is: 
\begin{equation}\label{eq.defTepsK}
\mathbb{T}_\varepsilon^K: L^2(\Omega,H^K) \rightarrow L^2(\Omega,H^K), \:\: \mathbb{T}_\varepsilon^K = E_\varepsilon^K \mathring{T_\varepsilon} P_\varepsilon^K.
\end{equation}

The next result is proved in the same way as Proposition \ref{propsamespec1}, using the ingredients of Proposition \ref{prop.unfold2}. 
\begin{proposition}\label{prop.specmatchK}
$\mathbb{T}_\varepsilon^K$ is a self-adjoint isomorphism from $H^K$ into itself. 
Its spectrum $\sigma(\mathbb{T}_\varepsilon^K)$ is:
$$\sigma(\mathbb{T}_\varepsilon^K) = \sigma(T_\varepsilon) \setminus \left\{ 0 \right\},$$
and a real value $\lambda \neq 0$ is an eigenvalue of $\mathbb{T}_\varepsilon^K$ if and only if it is an eigenvalue of $T_\varepsilon$. 
\end{proposition}

The following result is the counterpart of Proposition \ref{prop.cvrescal1} in the present context. 

\begin{proposition}\label{prop.cvT0K}
The operator $\mathbb{T}_\varepsilon^K$ converges pointwise to a limit operator $\mathbb{T}_0^K : L^2(\Omega,H^K) \rightarrow L^2(\Omega,H^K)$, 
i.e. for any function $\phi \in L^2(\Omega, H^K)$, the following convergence holds strongly in $L^2(\Omega,H^K)$: 
$$ \mathbb{T}_\varepsilon^K \phi  \:\: \xrightarrow{\varepsilon \rightarrow \phi} \:\:Ê\mathbb{T}_0^K\phi. $$
The operator $\mathbb{T}_0^K$ is defined by 
$$\mathbb{T}_0^K \phi(x,y) = Q(\nabla v_0(x) \cdot y + \widehat{v}(x,y)),$$
where 
\begin{itemize}
\item $Q : L^2(\Omega, H^1(KY)) \rightarrow L^2(\Omega,H^K)$ is the restriction operator, 
\item $\widehat{v}$ is the unique solution in $L^2(\Omega, H^1_\#(KY)/\mathbb{R})$ to the following equation:
\begin{equation}\label{defT0K}
-\Delta_y \widehat{v} (x,y) = -\text{\rm div}_y(\mathds{1}_{\omega^K}(y) \nabla \phi)(x,y)  \text{ in } H^1_\#(KY), \text{ a.e. in } x \in \Omega,
\end{equation}
\item $v_0$ is the unique solution in $H^1_0(\Omega)$ to the following equation: 
$$ -\Delta v_0 = - \frac{1}{K^d} \: \text{\rm div}\left( \int_{\omega^K}{\nabla_y \phi(x,y)\:dy}\right).$$
\end{itemize}
\end{proposition}
In the above statement, $H^1_\#(KY)$ is the closed subspace of $H^1(KY)$ 
consisting of functions which are $KY$-periodic. We equip the space $H^1_\#(KY) / \mathbb{R}$ with the inner product: 
 $$ \langle u,v \rangle_{H^1_\#(KY) / \mathbb{R}} = \int_{KY}{\nabla_y u \cdot \nabla_y v \:dy}.$$ ~\par
\smallskip

\subsubsection{Study of the limit operator $\mathbb{T}_0^K$ with a discrete Bloch decomposition}~\\

The considerations of the previous section lead us to the study of the spectrum of the operator $\mathbb{T}_0^K$.
One first remark in this direction is the following counterpart of Lemma \ref{lem.loc}, whose proof is omitted.

\begin{lemma}\label{lem.locK}
Let $T_0^K : H^1_\#(KY) / \mathbb{R} \to H^1_\#(KY) / \mathbb{R}$ be the operator that maps
$u \in H^1_\#(KY) / \mathbb{R}$ into the unique solution $T_0^K u \in H^1_\#(KY)/\mathbb{R}$ of the equation:
\begin{equation}\label{eq.defT0K}
 -\Delta_y (T_0^K u)  = -\text{\rm div}_y(\mathds{1}_{\omega^K} \nabla_y u),
 \end{equation}
or under variational form: 
$$ \forall v \in H^1_\#(KY) / \mathbb{R}, \:\: \int_{KY}{\nabla_y (T_0^K u) \cdot \nabla_y v \:dy} = \int_{\omega^K}{\nabla_y u \cdot \nabla_y v \:dy}. $$
Then the spectrum of $\mathbb{T}_0^K$ contains the spectrum $\sigma(T_0^K)$ of $T_0^K$.
\end{lemma}

Hence, our problem boils down to the study of the spectrum of the operator $T_0^K:H^1_\#(KY) / \mathbb{R} \to H^1_\#(KY) / \mathbb{R}$.
Since $T_0^K$ has $Y$-periodic coefficients, a natural idea is to consider its effect 
on the \textit{Bloch coefficients} of a function $u \in H^1_\#(KY) / \mathbb{R}$.\\

Following the lead of \cite{aguirreconca,allaireconcafs, allaireconca}, the main tool of this section is the following discrete Bloch decomposition theorem, 
which roughly speaking will allow us to diagonalize operators with $Y$-periodic coefficients
acting on $KY$-periodic functions; 
see \cite{blp} Chap. 4 \S 3.2, \cite{reedsimon4}, \S XIII.16, \cite{kuchment,wilcox} for the continuous Bloch decomposition, 
and \cite{allairebrianevanni,concaplanchard,concavan} for its use in the context of periodic homogenization. 
In this article, we shall use several versions of the discrete decomposition, associated to various scalings. 
All of them are proved in the same way, which is exemplified in Appendix (see Theorem \ref{thbloch}).

\begin{theorem}\label{thblochK}
Let $u$ in $L^2_\#(KY)$. Then, there exist $K^d$ complex-valued functions $u_j(y) \in L^2_\#(Y)$, $j=(j_1,...,j_d)$, $j_1,...,j_d = 0,...,K-1$, 
such that the following identity holds: 
\begin{equation}\label{prop.blochK}
 u(z) = \sum\limits_{0\leq j \leq K-1}{u_j(z)\: e^{\frac{2i\pi j}{K} \cdot z}},  \text{ a.e. } z \in KY;
 \end{equation}
The $u_j \in L^2_\#(Y)$ satisfying (\ref{prop.blochK}) are unique and are defined by: 
$$u_j(y) = \sum\limits_{0\leq j^\prime \leq K-1}{u(y+j^\prime)e^{-2i\pi \frac{j}{K} \cdot (y + j^\prime)}},  \text{ a.e. } y \in Y .$$ 
Furthermore, the Parseval identity holds: for $u,v \in L^2_\#(KY)$, with coefficients $u_j$, $v_j$, $0\leq j \leq K-1$, one has: 
\begin{equation}\label{eq.parsbloch}
 \frac{1}{K^d} \int_{KY}{u(z)\overline{v(z)}\:dx}= \sum\limits_{0\leq j\leq K-1}{\int_{Y}{u_j(y)\overline{v_j(y)}\:dy}}.
 \end{equation}
\end{theorem}
Note that, in the above statement and throughout this article, 
we use the same notation for a real-valued Hilbert space (e.g. $L^2(Y)$, $H^1(Y)$,...) and its complexified version.\\

If $u$ additionally belongs to $H^1_\#(KY)$, each coefficient $u_j$, $0\leq j \leq K-1$ belongs to $H^1_\#(Y)$, 
and the following relation holds: 
$$ \nabla u(z) = \sum\limits_{0\leq j \leq K-1}{\left( \nabla u_j(z) + \frac{2i\pi j}{K} \right) \: e^{\frac{2i\pi j}{K} \cdot z}},  \text{ a.e. } z \in KY.$$

Using this tool, the spectrum of $T_0^K$ may be given a quite convenient expression. 
The following result is an immediate consequence of
the definition (\ref{eq.defT0K}) in combination 
with the decomposition (\ref{prop.blochK}) and 
the Parseval identity (\ref{eq.parsbloch}).

\begin{proposition}\label{prop.secpTetaBloch}
For a multi-index $j=(j_1,...,j_d)$, $0 \leq j_i \leq K-1$, define the quasi-momentum $\eta_j = \frac{j}{K}$. 
Then, the spectrum $\sigma(T_0^K)$ of $T_0^K$ is exactly the reunion 
$$\sigma(T_0^K) = \bigcup\limits_{0\leq j \leq K-1}{\sigma(T_{\eta_j})},$$
where the operators $T_\eta$ are introduced and analyzed in the next Section \ref{sec.Teta}.
\end{proposition}

\subsubsection{Definition and properties of the operators $T_\eta$}\label{sec.Teta}~\\

For any quasi-momentum
$\eta \in \mathbb{R}^d$, 
we introduce an operator $T_\eta$ as follows: 
\begin{itemize} 
\item For $\eta =0$, $T_{0}$ is defined from $H^1_\#(Y) / \mathbb{C}$ into itself; it
maps an arbitrary function $u \in  H^1_\#(Y) / \mathbb{C}$ into the unique element $T_0 u \in H^1_\#(Y) / \mathbb{C}$ such that: 
\begin{equation}\label{eq.defT0Bloch}
 -\Delta_y (T_0 u)= -\text{\rm div}(\mathds{1}_\omega(y) \nabla_y u).
 \end{equation}
\item For $\eta \neq 0$, $T_{\eta} : H^1_\#(Y) \rightarrow H^1_\#(Y)$ 
maps $u \in H^1_\#(Y)$ into the unique element $T_{\eta} u \in H^1_\#(Y)$ such that: 
\begin{equation}\label{eq.defTetaBloch}
 -(\text{\rm div}_y + 2i\pi \eta )(\nabla_y + 2i\pi \eta)(T_{\eta} u) =  -(\text{\rm div}_y + 2i\pi \eta)(\mathds{1}_\omega (\nabla_y u  + 2i\pi \eta u)). 
 \end{equation}
\end{itemize}

These definitions make
sense owing to the Riesz representation theorem, since the mapping
$$ ( u , v ) \mapsto \int_{Y}{(\nabla_y u + 2i\pi \eta  u) \cdot \overline{(\nabla_y v + 2i\pi \eta v)} \:dy}$$
is an inner product on $H^1_\#(Y)$ when $\eta \neq 0$.

In the following, it is convenient to see alternatively $T_{\eta}$, $\eta \neq 0$ as an operator acting on the space $H^1_{\eta}(Y)$ of $e^{2i\pi \eta}$-periodic functions: 
$$ H^1_{\eta}(Y) := \left\{ u(y) e^{2i\pi \eta \cdot y}, \:\: u \in H^1_\#(Y) \right\} \subset H^1(Y),$$
which is a Hilbert space when equipped with the inner product 
$$ (u,v) \mapsto \int_Y{\nabla_y u \cdot \overline{\nabla_y v}\:dy}.$$
With a small abuse in notations, when $\eta \neq 0$, we thenceforth denote by $T_\eta$ the operator which maps any $u \in H^1_\eta(Y)$ to the unique element $T_\eta u \in H^1_\eta(y)$ such that: 
$$ \forall v \in H^1_\eta(Y), \:\: \int_Y{\nabla_y (T_\eta u) \cdot \overline{\nabla_y v}\:dy} = \int_\omega{\nabla_y u \cdot \overline{\nabla_y v}\:dy}.$$
Notice that, from this perspective, even though $T_\eta$ makes perfect sense for $\eta \in \mathbb{R}^d$, 
it is only worth studying it when $\eta \in \overline{Y}$, the closed unit cell. 
Also, observe that $T_\eta$ is self-adjoint.

For a fixed value $\eta \in \overline{Y}$, the analysis of the spectral properties of the operators $T_{\eta}$ arises in pretty much the same way as that detailed in 
Section \ref{sec.evTD} (see in particular Proposition \ref{prop.evTD}):
 
\begin{proposition}\label{prop.specTeta}
In the context outlined above, 
\begin{itemize}
\item The following orthogonal decomposition holds: 
$$ H^1_\#(Y) / \mathbb{C} = \text{\rm Ker}(T_0) \oplus \text{\rm Ker}(I-T_0) \oplus \mathfrak{h}_0,$$
where the spaces $\text{\rm Ker}(T_0)$, $\text{\rm Ker}(I-T_0)$, and $ \mathfrak{h}_0$ are respectively equal to:
$$ \text{\rm Ker}(T_0) = \left\{ u \in H^1_\#(Y) / \mathbb{C}, \: u = 0 
\text{ in }
\omega \right\}, 
\quad \text{\rm Ker}(I - T_0) = \left\{ u \in H^1_\#(Y) / \mathbb{C}, \: u = 0 
\text{ in }
Y \setminus \overline{\omega} \right\}, \text{ and }$$
$$ \mathfrak{h}_0 = \left\{ u \in H^1_\#(Y) / \mathbb{R}, \:\: -\Delta_y u = 0 \text{ in } \omega \text{ and in } 
Y \setminus \overline{\omega}\right\}.$$
\item For $\eta \in \overline{Y}$, $\eta \neq 0$, the following orthogonal decomposition holds: 
$$ H^1_{\eta}(Y) = \text{\rm Ker}(T_\eta) \oplus \text{\rm Ker}(I-T_\eta) \oplus \mathfrak{h}_\eta,$$
where the spaces $\text{\rm Ker}(T_\eta)$, $\text{\rm Ker}(I-T_\eta)$, and $ \mathfrak{h}_\eta$ are respectively equal to:
$$ \text{\rm Ker}(T_\eta) = \left\{ u \in H^1_\eta(Y), \: \exists c \in \mathbb{C}, \: u = c \text{ in }
\omega \right\}, \quad \text{\rm Ker}(I - T_\eta) = \left\{ u \in H^1_\eta(Y), \: u = 0 \text{ in } 
Y \setminus \overline{\omega} \right\}, \text{ and }$$
$$ \mathfrak{h}_\eta = \left\{ u \in H^1_\eta(Y), \:\: -\Delta_y u = 0 \text{ in } \omega \text{ and in }
Y \setminus \overline{\omega}, \text{ and } \int_{\partial \omega}{\frac{\partial u^+}{\partial n} \:ds} = 0\right\}.$$
\end{itemize}
\end{proposition}

The next result about the spectrum of the operators $T_\eta$ is proved in exactly the same way as Proposition \ref{prop.TDcompact}, following the approach developed in \cite{bontriki},
replacing the Green function of the macroscopic domain $\Omega$ with either the periodic Green function (in the case $\eta =0$), or with the quasi-periodic 
Green function of the unit cell $Y$ with quasi-period $\eta$ (in the case $\eta \neq 0$). 
We refer to \cite{AmmariKangLee} for properties of these Green functions (see also \cite{lipton}).

\begin{proposition}
For every $\eta \in \overline{Y}$, the operator $R_\eta := T_\eta - \frac{1}{2}I : \mathfrak{h}_\eta \to \mathfrak{h}_\eta$ 
is compact with operator norm $\lvert\lvert R_\eta \lvert\lvert \leq \frac{1}{2}$. Consequently, the spectrum of $T_\eta$ 
is a discrete sequence of eigenvalues with finite multiplicity; its unique accumulation point is $\frac{1}{2}$. 
\end{proposition}

Hence, for a fixed value $\eta$, the operator $T_{\eta}: \mathfrak{h}_\eta \to \mathfrak{h}_\eta$ has a discrete spectrum, 
which is only composed of eigenvalues $\lambda_k^\pm(\eta)$, whose only accumulation point is $\frac{1}{2}$. We order these eigenvalues as follows:
$$0 < \lambda_1^-(\eta_j) \leq \lambda_2^-(\eta_j) \leq ... \leq \frac{1}{2} \leq ... \leq... \lambda_2^+(\eta_j) \leq  \lambda_1^+(\eta_j) <1.$$
We also denote by $w_k^\pm(\eta)$ the $k^{\text{th}}$ eigenfunction of $T_\eta$, associated to the eigenvalue $\lambda^{\pm}_k(\eta)$. 
Our next purpose is to prove that the mappings $\eta \mapsto \lambda_k^{\pm}(\eta)$ are continuous with respect to $\eta$, 
a result, which to the best of our knowledge is new.

\begin{theorem}\label{th.contev}
For any $k \geq 1$, the mappings $\eta \mapsto \lambda_k^-(\eta)$, $\eta \mapsto \lambda_k^+(\eta)$ are Lipschitz continuous, with a Lipschitz constant which is independent of 
the index $k$. 
\end{theorem}
\begin{proof}
The proof is inspired from an idea in \cite{gerard} and consists in expressing $\eta \mapsto \lambda_k^\pm(\eta)$ 
as min-max of Lipschitz functions with respect to $\eta$. However, before doing so, we have to deal with the difficulty that the operators $T_\eta$ are defined on different Hilbert spaces.  
We use again the idea introduced in Section \ref{sec.rest}, which relates $T_\eta$ with an operator $\mathring{T_\eta}$, from $H^1(\omega) / \mathbb{C}$ into itself: 
\begin{itemize}
\item For $\eta \in \overline{Y} \setminus \left\{ 0 \right\}$, $\mathring{T_\eta}$ is defined as follows:
$$\forall u \in  H^1(\omega) / \mathbb{C},\:\: \mathring{T_\eta} uÊ= (T_\eta U_\eta u) \lvert_\omega,$$ 
where $U_\eta : H^1(\omega) \to H^1_\eta(Y)$ is the harmonic extension operator: 
$$ (U_\eta u) \lvert_\omega = u, \text{ and } -\Delta_y (U_\eta u) = 0 \text{ on } Y \setminus \overline{w},$$
which induces a bounded operator (still denoted by $U_\eta$) from $H^1(\omega) / \mathbb{C}$ into $H^1_\eta(Y) / \text{\rm Ker}(T_\eta)$ (see Proposition \ref{prop.specTeta}).
\item The construction of $\mathring{T_0} : H^1(\omega) / \mathbb{C} \to H^1(\omega) / \mathbb{C}$ is carried out in a completely analogous way.  
 \end{itemize}
 
One proves along the line of Section \ref{sec.rest} that, for any $\eta \in \overline{Y}$: 
\begin{itemize}
\item The spectrum of $\mathring{T_\eta}$ equals $\sigma(\mathring{T_\eta}) = \sigma(T_\eta) \setminus \left\{Ê0 \right\}$.
\item $\lambda \in \mathbb{R}$ is an eigenvalue of $\mathring{T_\eta}$ if and only if it is an eigenvalue of $T_\eta$.
\item $\mathring{T_\eta}$ is injective, and the following decomposition holds: 
$$ H^1(\omega) / \mathbb{C} = \text{\rm Ker}(I-\mathring{T_\eta}) \oplus \mathfrak{h}_\omega, \text{ where }$$
$$
\text{\rm Ker}(I-\mathring{T_\eta}) = \left\{ u \in H^1(\omega)/\mathbb{C}, \: \exists \:c \in \mathbb{C}, \: u = c \text{ on } \partial \omega \right\},
\text{ and } \mathfrak{h}_\omega = \left\{ u \in H^1(\omega) / \mathbb{C}, \: -\Delta_y u = 0 \text{ in } \omega \right\}.$$
\item The mapping $\mathring{R_\eta} := \mathring{T_\eta} - \frac{1}{2}I$, from $\mathfrak{h}_\omega$ into itself, is compact. 
\end{itemize} 

We now rely on the min-max formulae for the eigenvalues of compact operators:
$$\lambda_k^{-}(\eta) = \min\limits_{u \in {\mathfrak h}_\omega \setminus \left\{0 \right\} \atop u \perp w_1^-(\eta),...,w_{k-1}^-(\eta)}{\frac{ \displaystyle{\int_\omega { \nabla_y (\mathring{T_\eta} u) \cdot \overline{\nabla_y u}  \:dx}}}{ \displaystyle{\int_\omega{\lvert \nabla_y u \lvert^2 \:dx}}}} = \max\limits_{F_k \subset {\mathfrak h}_\omega \atop \text{\rm dim}(F_k) = k-1}{\min\limits_{u \in F_k^\perp \setminus \left\{ 0 \right\}}{\frac{ \displaystyle{\int_\omega {\nabla_y (\mathring{T_\eta} u) \cdot \overline{\nabla_y u}  \:dx}}}{ \displaystyle{\int_\omega{\lvert \nabla_y u \lvert^2 \:dx}}}}}, $$
and similar fomulae hold for the eigenvalues $\lambda_k^{+}(\eta)$. 
Hence, to prove Theorem \ref{th.contev}, it is enough to show that, for any $u \in \mathfrak{h}_\omega$, $u\neq 0$, 
the mapping 
\begin{equation}\label{eq.maptoLip}
\eta \longmapsto \frac{\displaystyle{\int_{\omega}{\nabla_y (\mathring{T_\eta} u) \cdot \overline{\nabla_y u} \:dy}}}{\displaystyle{
\int_{\omega}
{\lvert \nabla_y u \lvert^2 \:dy}}}
\end{equation}
 is Lipschitz continuous with Lipschitz constant independent of $u$, which we do now. 
%%--------------------------------------------------------------
%%--------------------------------------------------------------

Let $u \in \mathfrak{h}_\omega$, $u\neq 0$ and $\eta_1, \eta_2 \in \overline{Y}$ be given. 
Let $z \in H^1(\omega)$ denote an element in the class $u \in \mathfrak{h}_\omega$.
By definition, $\mathring{T}_{\eta_1} u$ (resp. $\mathring{T}_{\eta_2} u$) is the class modulo constants of $T_{\eta_1}U_{\eta_1}(z + c_1)\lvert _\omega \in H^1(\omega)$
(resp. of $T_{\eta_2}U_{\eta_2}(z + c_2)\lvert _\omega$),
where $c_1$ and $c_2$ are arbitrary constants.
We set $v_i = T_{\eta_i}U_{\eta_i}(z + c_i) \in H^1_{\eta_i}(Y)$, $i=1,2$.
By definition of $T_{\eta_i}$, $v_i$ solves
\begin{equation}\label{eq.varfvi}
\forall w \in H^1_{\eta_i}(Y), \:\: \int_Y{\nabla_y v_i \cdot \overline{\nabla_y w} \:dy} = \int_\omega{\nabla_y u \cdot \overline{\nabla_y w} \:dy}. 
\end{equation}
The quantity of interest is now:
$$\int_{\omega}{\nabla_y (\mathring{T}_{\eta_1} u) \cdot \overline{\nabla_y u} \:dy}  -  \int_{\omega}{\nabla_y (\mathring{T}_{\eta_2} u) \cdot \overline{\nabla_y u} \:dy}  =  \int_{\omega}{\nabla_y (v_1 -v_2) \cdot \overline{\nabla_y u} \:dy},$$
and we are led to estimate the \textit{real-valued} integral 
$$I :=  \int_{\omega}{\nabla_y (v_1 -v_2) \cdot \overline{\nabla_y u} \:dy}.$$
To achieve this, we insert $w(y) = p(y)v_1(y)$ (resp. $w = p(y)^{-1}v_2(y)$), where $p(y) \equiv e^{2i\pi(\eta_2-\eta_1)\cdot y}$, 
as a test function in the variational formulation (\ref{eq.varfvi}) for $v_2$ (resp. for $v_1$), and subtract the resulting identities to obtain:
\begin{multline*}
\int_\omega{\overline{\nabla_y u} \cdot \left( p(y)^{-1} \nabla_y v_2 - p(y) \nabla_y v_1\right)\:dy} = 2i\pi(\eta_2-\eta_1) \cdot \int_{\omega}{\left(p(y)^{-1}v_2 - p(y)v_1 \right) \overline{\nabla_y u} \:dy} \\
+ \int_Y{\left(p(y)^{-1} \overline{\nabla_y v_1} \cdot \nabla_y v_2 - p(y) \overline{\nabla_y v_2} \cdot \nabla_y v_1 \right) \:dy} + 2i\pi(\eta_1-\eta_2) \cdot \int_{Y}{\left(p(y)^{-1}v_2 \overline{\nabla_y v_1} + p(y)v_1 \overline{\nabla_y v_2}\right) \:dy} .
\end{multline*} 
By rearranging this expression, and notably bringing into play the quantities $(p(y)^{-1}-1)$ and $(p(y)-1)$, we infer the following expression for the integral $I$:
\begin{multline*}
I = \int_\omega{(1-p(y)^{-1}) \overline{\nabla_y u} \cdot \nabla_y v_2 \:dy} -  \int_\omega{(1-p(y)) \overline{\nabla_y u} \cdot \nabla_y v_1 \:dy} \\+ 2i\pi(\eta_2-\eta_1) \cdot \int_{\omega}{\left(p(y)^{-1}v_2 - p(y)v_1 \right) \overline{\nabla_y u} \:dy} 
+ \int_Y{\left( \overline{\nabla_y v_1} \cdot \nabla_y v_2 - \overline{\nabla_y v_2} \cdot \nabla_y v_1 \right) \:dy} \\ 
+  \int_Y{(p(y)^{-1}-1) \overline{\nabla_y v_1} \cdot \nabla_y v_2\:dy} + \int_Y{(1-p(y)) \overline{\nabla_y v_2} \cdot \nabla_y v_1\:dy} \\
+ 2i\pi(\eta_1-\eta_2) \cdot \int_{Y}{\left(p(y)^{-1}v_2 \overline{\nabla_y v_1} + p(y)v_1 \overline{\nabla_y v_2}\right) \:dy}.
\end{multline*} 
Since $I$ is real-valued, taking real parts in the above identity, the fourth term in the right-hand side vanishes, and using the fact that there exists a constant $C>0$ such that:
$$\lvert p(y)-1\lvert + \lvert p(y)^{-1} - 1 \lvert \leq C \lvert\eta_1 - \eta_2 \lvert , \text{ a.e. } y \in Y,$$
we obtain: 
\begin{equation}\label{eq.estI}
\begin{array}{ccl}
 \lvert I \lvert &\leq& C \lvert \eta_1 - \eta_2\lvert \left(\lvert\lvert \nabla_y u \lvert\lvert_{L^2(\omega)^d} \lvert\lvert \nabla_y v_2 \lvert\lvert_{L^2(\omega)^d} + \lvert\lvert \nabla_y u \lvert\lvert_{L^2(\omega)^d} \lvert\lvert \nabla_y v_1 \lvert\lvert_{L^2(\omega)^d}\right) \\ 
  &&+ C \lvert \eta_1 - \eta_2\lvert \left(\lvert\lvert \nabla_y u \lvert\lvert_{L^2(\omega)^d} \lvert\lvert  v_2 \lvert\lvert_{L^2(\omega)} + \lvert\lvert \nabla_y u \lvert\lvert_{L^2(\omega)^d} \lvert\lvert v_1 \lvert\lvert_{L^2(\omega)}\right)\\
  &&+ C \lvert \eta_1 - \eta_2\lvert \left(\lvert\lvert \nabla_y v_1 \lvert\lvert_{L^2(Y)^d} \lvert\lvert  \nabla_y v_2 \lvert\lvert_{L^2(Y)^d}  + \lvert\lvert \nabla_y v_1 \lvert\lvert_{L^2(Y)^d} \lvert\lvert  v_2 \lvert\lvert_{L^2(Y)} + \lvert\lvert v_1 \lvert\lvert_{L^2(Y)} \lvert\lvert  \nabla_y v_2 \lvert\lvert_{L^2(Y)^d} \right),
  \end{array}
 \end{equation}
 where the constant $C$ only depends on the geometry of $\omega$.

In order to estimate the right-hand side of the above inequality, we use the a priori estimate $\lvert\lvert \nabla_y v_i \lvert\lvert_{L^2(Y)^d} \leq \lvert\lvert \nabla_y u \lvert\lvert_{L^2(Y)^d}$, which is an immediate consequence of (\ref{eq.varfvi}).
Moreover, we rely on the following Poincar\'e-Wirtinger inequality, which follows 
from a classical contradiction argument: 
there exists a constant $C >0$ (depending only on $\omega$) such that, for any function $v \in H^1(Y)$ with $\int_\omega{v \:dy} = 0$, 
\begin{equation}\label{eq.pwvi}
\int_Y{\lvert v \lvert^2 \:dx} \leq C \int_{Y}{\lvert \nabla_y v \lvert^2 \:dx}.
\end{equation} 
Hence, adjusting the constant $c_i$ so that $\int_{\omega}{v_i \:dy} = 0$, it follows:
\begin{eqnarray*}
\lvert\lvert v_i \lvert\lvert_{L^2(Y)}^2 &\leq&C\,\int_Y{\lvert \nabla_y (T_{\eta_i}(U_{\eta_i}(z+c_i))\lvert^2 \:dy},\\
&\leq& C \, \int_\omega{\lvert \nabla_y z \lvert^2 \:dy}.
\end{eqnarray*}
Eventually, (\ref{eq.estI}) yields: 
$$ \lvert I \lvert \leq C \lvert \eta_1 - \eta_2 \lvert \:\lvert\lvert \nabla_y u \lvert \lvert^2_{L^2(\omega)^d}.$$
This proves that the mapping (\ref{eq.maptoLip}) is Lipschitz continuous with respect to the parameter $\eta$, 
and terminates the proof.
\end{proof}\par
\smallskip
\subsubsection{End of the proof of Theorem \ref{th.blochspec}}~\\

Theorem \ref{th.blochspec} arises as a consequence of the material proved in the previous subsections. 
We know from Proposition \ref{prop.specmatchK} that, for an arbitrary integer $K$, the rescaled operator $\mathbb{T}_\varepsilon^K: L^2(\Omega,H^K) \rightarrow L^2(\Omega,H^K)$
defined by (\ref{eq.defTepsK}) has the same spectrum as $T_\varepsilon$ (except possibly for the eigenvalue $0$). 

We also know that $\mathbb{T}_\varepsilon^K$ converges pointwise to a limit operator $\mathbb{T}_0^K$ (see Proposition \ref{prop.cvT0K}). 
As a consequence of Proposition \ref{prop.speclsc}, the limit spectrum $\lim_{\varepsilon \to 0}{\sigma(T_\varepsilon)}$ contains the spectrum $\sigma(\mathbb{T}_0^K)$. 

Using Lemma \ref{lem.locK} and Proposition \ref{prop.secpTetaBloch}, this implies the inclusion:
$$\bigcup\limits_{0\leq j \leq K-1}{\left\{\lambda_1^-(\eta_j), \lambda_2^-(\eta_j), ..., \lambda_2^+(\eta_j), \lambda_1^+(\eta_j) \right\}} \subset \lim\limits_{\varepsilon \to 0}{\sigma(T_\varepsilon)},\:\: (\eta_j \equiv \frac{j}{K}).$$
Eventually, since this procedure is available for any value of $K$, and using the continuity result of Theorem \ref{th.contev}, 
we end up with the fact that the Bloch spectrum $\sigma_{\text{\rm Bloch}}$ defined by (\ref{eq.defBlochspec}) is contained in the limit $\lim_{\varepsilon \to 0}{\sigma(T_\varepsilon)}$.

\subsection{Interaction of the inclusions with the boundary of $\Omega$: a completeness result}\label{seccompleteness}~\\

In this previous sections, we have identified one part $\sigma_{\text{\rm Bloch}}$ of the limit spectrum $\lim _{\varepsilon \rightarrow 0}{\sigma(T_\varepsilon)}$, 
which measures the interactions between several inclusions. 
In the spirit of the work of G. Allaire and C. Conca in \cite{allaireconcafs,allaireconca}, 
we now prove that the remaining part is a consequence of interactions between these inclusions and the boundary $\partial \Omega$ of the macroscopic domain $\Omega$.

\begin{theorem}\label{th.blochbl}
The limit spectrum is decomposed as: 
$$ \lim \limits_{\varepsilon \rightarrow 0}{\sigma(T_\varepsilon)} = \left\{ 0,1 \right\} \cup \sigma_{\partial \Omega} \cup \sigma_{\text{\rm Bloch}},$$
where the Bloch spectrum $\sigma_{\text{\rm Bloch}}$ is defined in (\ref{eq.defBlochspec}) and the `boundary layer spectrum' $\sigma_{\partial \Omega}$ is: 
$$\sigma_{\partial \Omega} = \left\{ \lambda \in (0,1), \text{ s.t. }\:\:
\begin{minipage}{0.65\textwidth}
For any sequence $\lambda_\varepsilon \in \sigma(T_\varepsilon)$ satisfying $\lambda_\varepsilon \rightarrow \lambda$,  
and any sequence $u_\varepsilon \in H^1_0(\Omega)$ of associated eigenvectors, one has: 
$$ \varepsilon^{-(1-1/d+s)} \lvert\lvert \nabla u_\varepsilon \lvert\lvert_{L^2({\mathcal U}_\varepsilon)} \xrightarrow{\varepsilon \to 0} +\infty \text{ for all } s>0.$$
\end{minipage}\right\},$$
where ${\mathcal U}_\varepsilon := \left\{ x \in \Omega, \:\: d(x,\partial \Omega) < \varepsilon \right\}$ is the tubular neighborhood of $\partial \Omega$ with width $\varepsilon$. 
\end{theorem}

\begin{proof}
The proof is quite similar to that of the completeness results shown in \cite{allaireconcafs, allaireconca}. Let $\lambda \in \lim_{\varepsilon \to 0}{\sigma(T_\varepsilon)}$ with $\lambda \notin \sigma_{\partial \Omega}$. We also assume $\lambda \neq \frac{1}{2}$ (a value which obviously belongs to $\sigma_{\text{\rm Bloch}}$). 
Then there exists a sequence $\lambda_\varepsilon \in \sigma(T_\varepsilon)$ and associated eigenvectors $u_\varepsilon \in H^1_0(\Omega)$, 
such that:
\begin{equation}\label{eq.evcplproof}
 \forall \varphi \in H^1_0(\Omega), \:\: \lambda_\varepsilon \int_\Omega{\nabla u_\varepsilon \cdot \nabla \varphi \:dx} = \int_{\omega_\varepsilon}{\nabla u_\varepsilon \cdot \nabla \varphi \:dx}, \text{ and } \lvert\lvert \nabla u_\varepsilon \lvert\lvert_{L^2(\Omega)^d} = 1,
 \end{equation}
and there exists constants $s>0$ and $C >0$ such that the following estimate of
the decay of $u_\varepsilon$ near the boundary $\partial \Omega$ holds:
$$ \lvert\lvert \nabla u_\varepsilon \lvert\lvert_{L^2({\mathcal U}_\varepsilon)} \leq  C\varepsilon^{1-1/d +s}.$$

Our aim is to prove that $\lambda$ belongs to $\sigma_{\text{\rm Bloch}}$. To this end, we proceed in three steps: firstly, the sequence $u_\varepsilon$ is modified into a sequence $v_\varepsilon$ defined on a larger, cubic domain such that the eigenvalue equation (\ref{eq.evcplproof}) is approximately satisfied. Secondly, a Bloch decomposition of $v_\varepsilon$ is performed to project this approximate equation onto the Bloch modes of (\ref{eq.evcplproof}); thirdly, we prove that at least one of the coefficients in this Bloch decomposition stays bounded away from $0$, which allows us
to identify one Bloch mode corresponding to $\lambda$. \\

\noindent \textit{Step 1: Construction of a sequence of approximate eigenvectors on a larger domain}. Let $K_\varepsilon$ be the smallest integer 
such that $\Omega \Subset {\mathcal Q}_\varepsilon$, where ${\mathcal Q}_\varepsilon = \left[ -\frac{\varepsilon K_\varepsilon}{2}, \frac{\varepsilon K_\varepsilon}{2} \right]$ 
is the hypercube centered at $0$, with size $\varepsilon K_\varepsilon$, i.e. 
it is composed of $K_\varepsilon^d$ cells of size $\varepsilon$.
Let also $v_\varepsilon \in H^1_0({\mathcal Q}_\varepsilon)$ be the extension of $u_\varepsilon$ 
by $0$ outside $\Omega$. We show that $v_\varepsilon$ approximately satisfies the eigenvalue problem (\ref{eq.evcplproof}) on ${\mathcal Q}_\varepsilon$ in the sense that:
\begin{equation}\label{quasiev} 
\forall \varphi \in H^1_\#({\mathcal Q}_\varepsilon), \:\: \lambda_\varepsilon \int_{{\mathcal Q}_\varepsilon} {\nabla v_\varepsilon \cdot \nabla \varphi \:dx} = \int_{\omega_\varepsilon^e}{\nabla v_\varepsilon \cdot \nabla \varphi\:dx} + r_\varepsilon(\varphi), \text{ where }\frac{\lvert r_\varepsilon(\varphi) \lvert}{\lvert\lvert \varphi \lvert\lvert_{H^1({\mathcal Q}_\varepsilon)}} \stackrel{\varepsilon \rightarrow 0}{\longrightarrow} 0,
\end{equation}
and where 
$$\omega_\varepsilon^e = \bigcup\limits_{j = (j_1,...,j_d) \atop j_i \in -\frac{K_\varepsilon}{2} + \left\{0,...,K_\varepsilon-1 \right\}}{\varepsilon(j+\omega)}$$
is the reunion of all the $\varepsilon$-copies of $\omega$ included in ${\mathcal Q}_\varepsilon$ (and not only in $\Omega$).

Let $\zeta_\varepsilon$ be a smooth cutoff function with the properties: 
$$ \zeta_\varepsilon = 1 \text{ on } \Omega \setminus \overline{{\mathcal U}_\varepsilon}, \:\: \zeta_\varepsilon = 0 \text{ on } {\mathcal Q}_\varepsilon \setminus \overline{\Omega}, \text{ and } \lvert\lvert \nabla \zeta_\varepsilon \lvert\lvert_{L^\infty({\mathcal U}_\varepsilon)} \leq \frac{C}{\varepsilon}.$$
An easy calculation yields, for an arbitrary test function $\varphi \in H^1_\#({\mathcal Q}_\varepsilon)$: 
\begin{equation}\label{eq.decintue}
\begin{array}{>{\displaystyle}cc>{\displaystyle}l}
\int_{\Omega}{\nabla u_\varepsilon \cdot \nabla \varphi \:dx} & = & \int_{\Omega}{\nabla u_\varepsilon \cdot \nabla (\zeta_\varepsilon \varphi)\:dx} +  \int_{\Omega}{\nabla u_\varepsilon \cdot \nabla ((1-\zeta_\varepsilon) \varphi)\:dx} \\
& = & \int_{\Omega}{\nabla u_\varepsilon \cdot \nabla (\zeta_\varepsilon \varphi)\:dx} +  \int_{\Omega}{(1-\zeta_\varepsilon)  \nabla u_\varepsilon \cdot \nabla \varphi\:dx} -  \int_{\Omega}{\varphi  \nabla u_\varepsilon \cdot \nabla \zeta_\varepsilon\:dx}.\\ 
\end{array}
\end{equation}
The last term in the right-hand side of the above identity may be controlled by: 
$$ \left\lvert   \int_{\Omega}{\varphi  \nabla u_\varepsilon \cdot \nabla \zeta_\varepsilon\:dx} \right\lvert \leq \lvert\lvert \nabla u_\varepsilon \lvert\lvert_{L^2({\mathcal U}_\varepsilon)^d} \lvert\lvert \varphi \lvert\lvert_{L^q(\Omega)} \lvert\lvert \nabla \zeta_\varepsilon\lvert\lvert_{L^{q^\prime}({\mathcal U}_\varepsilon)},  $$
for any $q < \frac{2d}{d-2}$ if $d\geq 3$, any $q <\infty$ if $d=2$, as an application of the Sobolev embedding theorem (see e.g. \cite{adams}), where $q^\prime$ satisfies $\frac{1}{2} + \frac{1}{q} + \frac{1}{q^\prime} = 1$. 
Hence, we obtain: 
$$\sup\limits_{\varphi \in H^1_\#({\mathcal Q}_\varepsilon) \atop \varphi \neq 0} {\left( \frac{1}{\lvert\lvert \varphi\lvert\lvert_{H^1_\#({\mathcal Q}_\varepsilon)}}  \left\lvert   \int_{\Omega}{\varphi  \nabla u_\varepsilon \cdot \nabla \zeta_\varepsilon\:dx} \right\lvert \right) } \leq C \varepsilon^{1-1/d +s} \varepsilon^{1/q^\prime -1}  \xrightarrow{\varepsilon \to 0} 0,$$
provided $q < \frac{2d}{d-2}$ is large enough. Arguing in the same way for the second term in the right-hand side of (\ref{eq.decintue}), we obtain: 
\begin{equation}\label{eq.qev1}
\sup\limits_{\varphi \in H^1_\#({\mathcal Q}_\varepsilon)\atop \varphi \neq 0} {\left( \frac{1}{\lvert\lvert \varphi \lvert\lvert_{H^1_\#({\mathcal Q}_\varepsilon)}} \left\lvert \int_{\Omega}{\nabla u_\varepsilon \cdot \nabla \varphi \:dx} - \int_{\Omega}{\nabla u_\varepsilon \cdot \nabla (\zeta_\varepsilon \varphi) \:dx} \right\lvert \right)} \xrightarrow{\varepsilon \to 0} 0.
\end{equation}
On the other hand, one proves in the same way that:
\begin{equation}\label{eq.qev2}
 \sup\limits_{\varphi \in H^1_\#({\mathcal Q}_\varepsilon)\atop \varphi \neq 0} {\left( \frac{1}{\lvert\lvert \varphi\lvert\lvert_{H^1_\#({\mathcal Q}_\varepsilon)}} \left\lvert   \int_{\omega_\varepsilon^e}{\nabla u_\varepsilon \cdot \nabla \varphi \:dx} - \int_{\omega_\varepsilon^e}{\nabla u_\varepsilon \cdot \nabla (\zeta_\varepsilon \varphi) \:dx} \right\lvert \right)}\xrightarrow{\varepsilon \to 0} 0.
 \end{equation}
Combining (\ref{eq.qev1}) and (\ref{eq.qev2}) together with (\ref{eq.evcplproof}), the desired relation (\ref{quasiev}) is obtained.\\

\noindent \textit{Step $2$: Bloch wave decomposition of the approximate eigenvalue problem (\ref{quasiev}).} We need yet another version of the Bloch decomposition theorem, adapted to the particular scaling under consideration. 

\begin{theorem}\label{th.Blochcompleteness}
Let $u \in L^2_\#({\mathcal Q}_\varepsilon)$; there exists a unique collection of (complex-valued) functions $\left\{ u_j \right\}_{j \in \left\{ 0,...,K_\varepsilon -1 \right\}^d}$, $u_j \in L^2_\#(Y)$, 
such that the following identity holds: 
$$ u(x) = \sum\limits_{0 \leq j \leq K_\varepsilon -1}{u_j(\frac{x}{\varepsilon}) e^{2i\pi \frac{j}{K_\varepsilon} \cdot \frac{x}{\varepsilon}}}, \text{ a.e. } x \in {\mathcal Q}_\varepsilon.$$
Moreover, if $u,v \in L^2_\#({\mathcal Q}_\varepsilon)$ have coefficients $\left\{ u_j \right\}$ and $\left\{ v_j\right\}$ in the above decomposition, the Parseval identity holds: 
 \begin{equation}\label{eq.Parsevalcomplete}
  \int_{{\mathcal Q}_\varepsilon}{u\overline{v} \:dx} = (\varepsilon K_\varepsilon)^d \sum\limits_{0\leq j \leq K_\varepsilon-1}{\int_{Y}{u_j(y) \overline{v_j(y)}\:dy}}.
 \end{equation}
\end{theorem}

Decomposing the function $v_\varepsilon \in H^1_\#({\mathcal Q}_\varepsilon)$ introduced in the first step according to Theorem \ref{th.Blochcompleteness} yields:
$$ v_\varepsilon (x) = \sum\limits_{0 \leq j \leq K_\varepsilon -1}{v_j^\varepsilon(\frac{x}{\varepsilon}) e^{2i\pi \frac{j}{K_\varepsilon} \cdot \frac{x}{\varepsilon}}}, \text{ where } v^\varepsilon_j \in H^1_\#(Y).$$ 
We further decompose each coefficient $v^\varepsilon_j$ ($0\leq j \leq K_\varepsilon-1$) on the Bloch eigenvectors associated to the quasi-moment $j/K_\varepsilon$.  
More precisely, let us recall the Bloch eigenvalue problem associated to (\ref{eq.evcplproof}), as studied in Section \ref{sec.Teta}. 
For $\eta \in \overline{Y}$, $\eta \neq 0$, one seeks $\lambda \in \mathbb{C}$ and $v \in H^1_\#(Y)$, $v \neq 0$, such that: 
$$\forall w \in H^1_\#(Y),\:\: \lambda \int_{Y}{(\nabla_y v +  2i\pi \eta v) \cdot \overline{(\nabla_y w + 2i\pi \eta w)}\:dy} = \int_{\omega}{(\nabla_y v + 2i\pi \eta v) \cdot \overline{(\nabla_y w + 2i\pi \eta w)}\:dy}.$$
In the case $\eta =0$, one searches for $\lambda \in \mathbb{C}$ and $v \in H^1_\#(Y) / \mathbb{C}$, $v \neq 0$, such that: 
$$\forall w \in H^1_\#(Y) / \mathbb{C},\:\: \lambda \int_{Y}{\nabla_y v \cdot \overline{\nabla_y w}\:dy} = \int_{\omega}{\nabla_y v \cdot \overline{\nabla_y w}\:dy}.$$

Each of these problems gives rise to a sequence of (real) eigenvalues $\lambda^k(\eta)$, $k \geq 1$ which converges to $\frac{1}{2}$, 
and to associated eigenvectors $v^k(\eta,y)$ in $H^1_\#(Y)$ (in $H^1_\#(Y)/\mathbb{C}$ in the case $\eta =0$). These eigenvectors are normalized by:
\begin{equation}\label{eq.orthovk}
 \int_{Y}{(\nabla_y v^k(\eta,y)  +  2i\pi \eta v^k(\eta,y) ) \cdot \overline{(\nabla_y v^{k^\prime}(\eta,y) + 2i\pi \eta v^{k^\prime}(\eta,y) )}\:dy}  = \delta_{kk^\prime}.
 \end{equation}
Now, each function $v_j^\varepsilon$ is decomposed on this basis as: 
$$ v_j^\varepsilon(y) = \sum\limits_{k\geq 1}{\alpha_\varepsilon^k(\frac{j}{K_\varepsilon}) v^k(\frac{j}{K_\varepsilon},y) }, \text{ for some coefficients } \alpha_\varepsilon^k(\frac{j}{K_\varepsilon}) \in \mathbb{C},$$
(up to a constant $c_0 \in \mathbb{C}$ in the case $j=0$), so that $v_\varepsilon$ eventually reads: 
\begin{equation}\label{eq.decBlochveps}
 v_\varepsilon = c_0 + \sum\limits_{0\leq j \leq K_\varepsilon -1}{\sum\limits_{k\geq 1}{\alpha_\varepsilon^k(\frac{j}{K_\varepsilon}) v^k(\frac{j}{K_\varepsilon},\frac{x}{\varepsilon}) e^{2i\pi \frac{j}{K_\varepsilon} \cdot \frac{x}{\varepsilon}} }}.
 \end{equation}

\noindent \textit{Step $3$: Choice of an adequate test function in (\ref{quasiev}).}
Let us now consider a \textit{modulation} $z_\varepsilon\in H^1_\#({\mathcal Q}_\varepsilon)$ of $v_\varepsilon$ of the form: 
$$ z_\varepsilon(x) =  \sum\limits_{0\leq j \leq K_\varepsilon -1}{\sum\limits_{k\geq 1}{ \psi^k(\frac{j}{K_\varepsilon}) \alpha_\varepsilon^k(\frac{j}{K_\varepsilon}) v^k(\frac{j}{K_\varepsilon},\frac{x}{\varepsilon}) e^{2i\pi \frac{j}{K_\varepsilon} \cdot \frac{x}{\varepsilon}} }},$$
where the $\psi^k(y)$ are continuous and $Y$-periodic functions on $Y$, which additionnally satisfy: 
$$ \sup\limits_{k \geq 1}{\lvert\lvert \psi^k \lvert\lvert_{L^\infty(Y)}} < \infty.$$
Inserting $z_\varepsilon$ as test function in the approximate spectral problem (\ref{quasiev}),
using the Parseval identity (\ref{eq.Parsevalcomplete}) and the orthogonality (\ref{eq.orthovk}) between the $v^k(\eta,\cdot)$, $k=1,...$, we obtain: 
\begin{equation}\label{locbloch}
 \frac{1}{\varepsilon^2} (\varepsilon K_\varepsilon)^d \sum\limits_{0\leq j \leq K_\varepsilon -1}{\sum\limits_{k\geq 1}{ \psi^k(\frac{j}{K_\varepsilon}) \lvert \alpha_\varepsilon^k(\frac{j}{K_\varepsilon})\lvert^2 \left( \lambda_\varepsilon - \lambda^k(\frac{j}{K_\varepsilon}) \right) }}  \xrightarrow{\varepsilon \to 0} 0.
 \end{equation}

Let us now define the positive measures $\nu_\varepsilon^k$ on $Y$ by: 
$$ \nu_\varepsilon^k(\eta) = \frac{1}{\varepsilon^2} (\varepsilon K_\varepsilon)^d \sum\limits_{0\leq j \leq K_\varepsilon-1}{\lvert \alpha_\varepsilon^k(\frac{j}{K_\varepsilon}) \lvert^2 \delta_{\eta = \frac{j}{K_\varepsilon}}},$$
where $\delta_{\eta = \frac{j}{K_\varepsilon}}$ is the Dirac distribution at $\frac{j}{K_\varepsilon}$ on ${\mathcal D}(Y)$.
Using the fact that $\lvert\lvert \nabla v_\varepsilon \lvert\lvert_{L^2({\mathcal Q}_\varepsilon)} = 1$ together with the Parseval identity (\ref{eq.Parsevalcomplete}) yields: 
\begin{equation}\label{normblochmeas}
 \sum\limits_{k\geq 1}{\int_Y{d\nu^k_\varepsilon(\eta)}} = 1.
 \end{equation}
In particular, for each $k \geq 1$, the sequence $\left\{ \nu_\varepsilon^k \right\}_\varepsilon$ has bounded total variation as $\varepsilon \rightarrow 0$. 
Hence, upon diagonal extraction, there exists a subsequence (still denoted with $\varepsilon$) such that for any $k \geq 1$, $\nu_\varepsilon^k$ converges to a limit measure $\nu^k$ in the sense of measures, i.e., 
$$\forall \psi \in {\mathcal C}(Y), \:\: \int_{Y}{\psi \: d\nu_\varepsilon^k  }\quad \xrightarrow{\varepsilon \rightarrow 0} \quad \int_{Y}{\psi \:d\nu^k }.$$ 
Hence, from (\ref{normblochmeas}), we have: 
$$ 0 \leq  \sum\limits_{k\geq 1}{\int_Y{d\nu^k(\eta)}} \leq 1 .$$
Our purpose is now to prove that at least one of the limiting measures $\nu^k$ does not vanish.
To this end, it is obviously enough to prove that there exist $L \geq 1$, $\delta >0$ and a subsequence $\varepsilon \rightarrow 0$ such that: 
\begin{equation}\label{statebloch}
  \sum\limits_{k=1}^L{\int_Y{d\nu_\varepsilon^k(\eta)}}  \geq \delta.
  \end{equation}
Let us assume that this property does not hold; then, for any $L$ and $\delta$ there exists $\varepsilon_\delta >0$ such that: 
$$\forall \varepsilon < \varepsilon_\delta,\:\:  \sum\limits_{k= 1}^L{\int_Y{d\nu_\varepsilon ^k(\eta)}}  < \delta.$$
Using the approximate spectral equation (\ref{quasiev}) with $\varphi = v_\varepsilon$ as test function, as well as the decomposition (\ref{eq.decBlochveps}) and the Parseval identity (\ref{eq.Parsevalcomplete}), we obtain: 
$$\begin{array}{>{\displaystyle}cc>{\displaystyle}l} 
\lambda_\varepsilon\lvert\lvert \nabla u_\varepsilon \lvert\lvert_{L^2(\Omega)}^2 & = & \sum\limits_{0\leq j \leq K_\varepsilon -1}{\sum\limits_{k\geq 1}{ \lambda^k(\frac{j}{K_\varepsilon}) \lvert \alpha_\varepsilon^k(\frac{j}{K_\varepsilon})\lvert^2  }} + r_\varepsilon, \\ 
&= & \sum\limits_{k \geq 1}{\int_Y{\lambda^k(\eta)\:d\nu_\varepsilon^k(\eta)}} + r_\varepsilon,
\end{array}
$$
for a remainder $r_\varepsilon \to 0$ as $\varepsilon \to 0$. Rearranging, we see that, for any $L>0$: 
\begin{equation}\label{eqcontradic}
\begin{array}{>{\displaystyle}cc>{\displaystyle}l} 
 (\lambda_\varepsilon - \frac{1}{2})\lvert\lvert \nabla u_\varepsilon \lvert\lvert_{L^2(\Omega)}^2  - r_\varepsilon
 & =& \sum\limits_{k \geq 1}{\int_Y{(\lambda^k(\eta) - \frac{1}{2}) \:d\nu_\varepsilon^k(\eta)}}\\
 & =&  \sum\limits_{ k=1}^L{\int_Y{(\lambda^k(\eta) - \frac{1}{2}) \:d\nu_\varepsilon^k(\eta)}} +  \sum\limits_{k >L}{\int_Y{(\lambda^k(\eta) - \frac{1}{2}) \:d\nu_\varepsilon^k(\eta)}}.
 \end{array}
 \end{equation}
Let $\delta >0$ be given. We choose $L$ large enough so that: 
$$\sum\limits_{k >L}{\int_Y{(\lambda^k(\eta) - \frac{1}{2}) \:d\nu_\varepsilon^k(\eta)}} < \delta,$$
which is possible since, for each given $\eta \in Y$, $\lambda^k(\eta) \rightarrow \frac{1}{2}$, and all the mappings $\eta \mapsto \lambda^k(\eta)$, $k\geq 1$ are Lipschitz continuous with a constant $C$ independent of $k$ (see Theorem \ref{th.contev}). 

Now, owing to our assumption, there exists $\varepsilon_\delta$ such that, for $\varepsilon < \varepsilon_\delta$, 
 $$\sum\limits_{1\leq k \leq L}{\int_Y{(\lambda^k(\eta) - \frac{1}{2}) \:d\nu_\varepsilon^k(\eta)}} < \delta,$$
 so that the left-hand side of (\ref{eqcontradic}) actually converges to $0$, which is absurd since $\lambda \neq \frac{1}{2}$. Hence (\ref{statebloch}) is proved, 
 and there exists an index $k_0 \geq 1$ such that $\nu^{k_0}$ does not vanish. 
 
 Let us finally go back to (\ref{locbloch}) using functions $\psi^k$ equal to $0$ if $k \neq k_0$, and an arbitrary continuous and $Y$-periodic function $\psi$ if $k =k_0$. 
 Passing to the limit $\varepsilon \rightarrow 0$ imposes: 
 $$ \int_Y{\psi(\eta) \left( \lambda - \lambda^{k_0}(\eta)  \right) d\nu^{k_0}(\eta)} = 0.$$
Since $\psi$ is arbitrary, this shows that there exists $\eta \in Y$ and $k \geq 1$ so that $\lambda = \lambda^k(\eta)$. 
\end{proof}

\begin{remark}
The boundary layer spectrum $\sigma_{\partial \Omega}$ is composed of the limits of the sequences of eigenvalues of $T_\varepsilon$ 
such that an associated eigenvector sequence retains a `significant' fraction of energy (but not all of it) around the boundary of the macroscopic domain $\partial \Omega$. 
This characterization of $\sigma_{\partial \Omega}$ is weaker than in the problems tackled in \cite{allaireconcafs,allaireconca}, 
where the eigenvector sequences featured in the boundary layer spectrum are concentrated near $\partial \Omega$, and decay exponentially far from $\partial \Omega$. 
We believe that this different behavior is an inherent feature of the present problem and reveals a stronger influence of the boundary $\partial \Omega$ in our case that in those of \cite{allaireconcafs,allaireconca}. 
For instance, in their study of the latter situations, the authors crucially relied on the fact that eigenvector sequences $u_\varepsilon$ 
can be `localized' - i.e. that a truncation of such an eigenvector sequence in a region which retains some energy of the $u_\varepsilon$ gives an approximate eigenvector sequence for the problem -, a fact which clearly does not hold in our case. 

Note that the great sensitivity of the spectrum of $T_\varepsilon$ to how the inclusions interact with the macroscopic boundary $\partial \Omega$ is very reminiscent of the results contained in \cite{castro,moskowvog}.
\end{remark}

%%%%%%%%%%%%%%%%%%%%%%%%%%%%%%%%%%%%%%%%%%%%%%%%%%%%%%%
\section{Study of the conductivity equation}\label{secdirect}
%%%%%%%%%%%%%%%%%%%%%%%%%%%%%%%%%%%%%%%%%%%%%%%%%%%%%%%

In this section, we turn to the study of the voltage potential $u_\varepsilon \in H^1_0(\Omega)$ 
generated by a source $f \in H^{-1}(\Omega)$ in the physical setting of Section \ref{sec.setting}. 
Recall that $u_\varepsilon$ is solution to the conductivity system (\ref{eq.potgen}), which we rewrite below for the sake of convenience:
\begin{equation}\label{eq.pot}
\left\{
\begin{array}{cl}
-\text{\rm div}(A_\varepsilon(x)\nabla u_\varepsilon)  = f & \text{in } \Omega,\\
u_\varepsilon=0 & \text{on }Ê\partial \Omega,
\end{array}
\right. , \text{ where } A_\varepsilon(x) = \left\{Ê
\begin{array}{cl}
1 & \text{if }Êx \in \Omega \setminus \overline{\omega_\varepsilon},\\
a & \text{if }Êx \in \omega_\varepsilon. 
\end{array}
\right.
\end{equation}
Here, the set of inclusions $\omega_\varepsilon$ is that defined in (\ref{eq.omeps}) and it satisfies the properties (\ref{eq.assumom}): $\omega \Subset Y$ is of class ${\mathcal C}^2$. 
It is filled with a material with complex-valued conductivity $a \neq 0$. 
The situations where $a$ is either real-valued and positive, or complex-valued with non zero imaginary part fall into the classical homogenization framework of 
elliptic equations and are by now well understood (see e.g. \cite{blp,jikov}). 
In what follows, we are especially interested in the case where $a\in \mathbb{R}$, $a<0$.

In our study, the effective conductivity tensor $A^*$ \textit{formally} calculated as (\ref{eq.hommat}) by application of the usual homogenization formulae 
to the periodic distribution of inclusions $\omega_\varepsilon$ (i.e. without justification of their well-posedness or validity), 
plays a key role, as well as the \textit{formally} homogenized problem:
\begin{equation}\label{eq.pothom} 
\left\{
\begin{array}{cl}
-\text{\rm div}(A^*\nabla u)  = f & \text{in } \Omega\\
u=0 & \text{on }Ê\partial \Omega
\end{array}
\right. ,
\end{equation}
Hence, we start with a discussion about the definition and properties of $A^*$.

%%%%%%%%%%%%%%%%%%%%%%%%%%%%%%%%
\subsection{The cell functions $\chi_i$ and the homogenized tensor $A^*$}\label{sec.cellpb}~\\
%%%%%%%%%%%%%%%%%%%%%%%%%%%%%%%%%

The building blocks of the homogenized tensor $A^*$ are the \textit{cell functions} $\chi_i$ ($i=1,...,d$), 
defined as the solutions in $H^1_{\#}(Y) / \mathbb{R}$ to: 
\begin{equation}\label{def.chiiclass}
 \text{\rm div}_y(A(y)(e_i + \nabla_y \chi_i) = 0, \text{ where }ÊA(y) = \left\{
 \begin{array}{cl}
 1 & \text{if }Êy \in Y \setminus \overline{\omega}, \\
 a & \text{if }Êy \in \omega,
 \end{array}
 \right.
 \end{equation}
 an equation that we may equivalently rewrite as a transmission problem:
\begin{equation}\label{def.chiitrans}
 \left\{Ê
\begin{array}{cl}
- \Delta_y \chi_i = 0 & \text{in } (Y \setminus \overline{\omega}) \cup \omega, \\
\chi_i^- = \chi_i^+ & \text{on }Ê\partial \omega, \\
a (\frac{\partial \chi_i^-}{\partial n_y} + n_i) =  \frac{\partial \chi_i^+}{\partial n_y} + n_i & \text{on }Ê\partial \omega.
\end{array}
\right.,
\end{equation}
Recall that (\ref{def.chiiclass})-(\ref{def.chiitrans}) are to be understood in the variational sense: 
$$ \forall v \in H^1_\#(Y) / \mathbb{R}, \:\: \int_Y{A(y) (\nabla_y \chi_i + e_i) \cdot \nabla_y v \:dy} = 0,$$
where $e_i$ is the $i^{\text{th}}$ vector of the canonical basis of $\mathbb{R}^d$. \\

\subsubsection{Periodic layer potentials}~\\ 

Let us start by recalling some material about the $Y$-periodic Green function and the associated notions of surface potentials; 
the reader is referred to \cite{AmmariKang,AmmariPer} for proofs and further details.

Let $G^{\#}(y,z)$ be the $Y$-periodic Green function: for given $z \in Y$, $y \mapsto G^\#(y,z)$ satisfies: 
\begin{equation}\label{eq.Gper}
\left\{\begin{array}{c}
\Delta_yG^{\#}(y,z) = \delta_z - 1 \text{ on } Y,\\
y \mapsto G(y,z) \text{ is } Y-\text{periodic}, \\
\int_Y{G(y,z)\:dy} = 0,
\end{array} \right.
\end{equation}
where $\delta_z$ is the Dirac distribution centered at $z$. 
Observe that the additional `$-1$' in the right-hand side of the system (\ref{eq.Gper}) when compared to (\ref{eq.poisk})
is required so that it admits a $Y$-periodic solution. 
In this context, the \textit{periodic single layer potential} ${\mathcal S}_\omega^\# \phi \in H^1_\#(Y)$ is defined for density functions $\phi \in L^2_0(\partial \omega)$, where $L^2_0(\partial \omega)  := \left\{\phi \in L^2(\partial \omega) , \:\: \int_{\partial \omega}{\phi \:ds} =0Ê\right\}$. Its expression is:
$$ {\mathcal S}^{\#}_\omega\phi(y) = \int_{\partial \omega}{G^{\#}(y,z)\phi(z)\:ds(z)}.$$ 
Likewise, the \textit{periodic Neumann-Poincar\'e operator} ${\mathcal K}_\omega^{\#*} : L^2_0(\partial \omega) \to L^2_0(\partial \omega)$ is defined for $\phi \in L^2_0(\partial \omega)$ by: 
$$ {\mathcal K}^{\# *}_\omega\phi(y) = \int_{\partial \omega}{\frac{\partial G^{\#}}{\partial n_y}(y,z)\phi(z)\:ds(z)} \in L^2_0(\partial \omega).$$ 
The properties of these operators are quite similar to those encountered in the `classical' surface potential theory:

\begin{theorem}\label{thproppotper}
Under the assumption that $\omega \Subset Y$ is ${\mathcal C}^2$ regular, the operators ${\mathcal S}_\omega^\#$ and ${\mathcal K}_\omega^{\#*}$ satisfy the following properties: 
\begin{enumerate}[(i)]
\item The operator ${\mathcal K}^{\#*}_\omega : L^2_0(\partial \omega) \rightarrow L^2_0(\partial \omega)$ is compact.
\item For any $\phi \in L^2_0(\omega)$, the function ${\mathcal S}^{\#}_\omega\phi$ belongs to $H^1_\#(Y)$, 
and is harmonic in $\omega \cup (Y \setminus \overline{\omega})$. Besides, the following jump relations hold in the sense of traces in $H^{-1/2}(\partial \omega)$: 
\begin{equation}\label{jumpper}
\frac{\partial ({\mathcal S}^{\#}_\omega \phi)^\pm}{\partial n_y}   = \pm \frac{1}{2}\phi + {\mathcal K}^{\#*}_\omega\phi.
\end{equation}
\item The operator $(\lambda I - {\mathcal K}^{\#*}_\omega) : L^2_0(\partial \omega) \rightarrow L^2_0(\partial \omega)$ is invertible if $\lvert \lambda \lvert \geq \frac{1}{2}$. 
\end{enumerate}
\end{theorem}

Note also that, like in the case of Section \ref{sec.NPTD}, 
the operator ${\mathcal S}_\omega^\#: L^2_0(\partial \omega) \to H^1_\#(Y)$ extends as a bounded operator (still denoted by ${\mathcal S}_\omega^\#$) from $H^{-1/2}_0(\partial \omega)$ into $H^1_\#(Y)$, and ${\mathcal K}_\omega^{\#*}$ extends as a bounded operator $H^{-1/2}_0(\partial \omega) \to H^{-1/2}_0(\partial \omega)$. 
Here we have used the notation $H^{-1/2}_0(\partial \omega) = \left\{ u \in H^{-1/2}(\partial \omega), \:\: \int_{\partial \omega}{u \:ds} = 0 \right\}$. 

\begin{proposition}\label{prop.lpper}
Let $v \in H^1_\#(Y)$ be harmonic in $\omega$ and $Y \setminus \overline{\omega}$; then, there exists a unique pair $(\phi,c) \in H^{-1/2}_0(\partial \omega) \times \mathbb{R}$ such that: 
\begin{equation}\label{id.somcste} 
v = {\mathcal S}_\omega^\# \phi + c.
\end{equation}
\end{proposition}
\begin{proof}
Let us first assume that there exists $\phi \in H^{-1/2}_0(\partial \omega)$ such that (\ref{id.somcste}) holds. 
Using the jump relations (\ref{jumpper}), we immediately identify $\phi = \left[ \frac{\partial v}{\partial n} \right] \in H^{-1/2}_0(\partial \omega)$, which proves the uniqueness of $(\phi,c) \in H^{-1/2}_0(\partial \omega) \times \mathbb{R}$ satisfying (\ref{id.somcste}).  

Conversely, if $v \in H^1_\#(Y)$ is harmonic in $\omega$ and $Y \setminus \overline{\omega}$, let $\phi = \left[ \frac{\partial v}{\partial n} \right] \in H^{-1/2}_0(\partial \omega)$. 
Then $w = {\mathcal S}_\omega^\# \phi \in H^1_\#(Y)$ is 
one solution of: 
$$ \left\{ 
\begin{array}{cl}
-\Delta w = 0 & \text{in } \omega \text{ and } Y \setminus \overline{\omega}, \\
w^- = w^+ & \text{on } \partial \omega, \\
\left[ \frac{\partial w}{\partial n}\right] = \phi & \text{on } \partial \omega.
\end{array}
\right.$$ 
Now, an easy variational argument reveals that the above system has a unique solution up to constants. 
\end{proof}

Last but not least, in the line of the general Section \ref{sec.NPTD} (see notably Proposition \ref{prop.TDcompact}), 
let us mention the link between the periodic Neumann-Poincar\'e operator ${\mathcal K}_{\omega}^{\#*}$ and 
the periodic version of the Poincar\'e variational operator $T_0: H^1_\#(Y) / \mathbb{R} \to H^1_\#(Y) / \mathbb{R}$, introduced in (\ref{eq.defT0}): 
\begin{proposition}
Let $T_0 : H^1_\#(Y) / \mathbb{R} \to H^1_\#(Y) / \mathbb{R}$ be the operator that maps $u \in H^1_\#(Y) / \mathbb{R}$ into
 the unique element $T_0 u$ in $H^1_\#(Y) / \mathbb{R}$ such that: 
$$\forall v \in H^1_\#(Y) / \mathbb{R}, \:\: \int_Y{\nabla_y (T_0 u) \cdot \nabla_y v \:dy} = \int_\omega{\nabla_yu \cdot \nabla_yv \:dy}. $$
Then, $\sigma(T_0) = \left\{ 0\right\} \cup \sigma_{\text{\rm cell}} \cup \left\{Ê1\right\}$, and
a value $\lambda \in \mathbb{C}$ belongs to $\sigma_{\text{\rm cell}}$ if and only if $ (\frac{1}{2}-\lambda) \in \sigma({\mathcal K}_\omega^{*\#})$. 
\end{proposition}
\par
\smallskip

\subsubsection{Well-posedness of the cell problems}~\\

We now come to the question of the well-posedness of the cell problems (\ref{def.chiiclass}).

\begin{proposition}
Let $a \in \mathbb{C}^*$ be outside the exceptional set $\Sigma_\omega$ defined by:
\begin{equation}\label{eqlmom}
 \Sigma_\omega := \left\{ a \in \mathbb{R}, \: \: \frac{1}{2}\frac{a+1}{a-1} \in \sigma( {\mathcal K}^{\#*}_\omega) \right\} = \left\{a \in \mathbb{R}, \:\: \frac{1}{1-a} \in \sigma_{\text{\rm cell}} \right\}.
 \end{equation}
Then each cell problem (\ref{def.chiiclass}) has a unique solution $\chi_i \in H^1_\#(Y) / \mathbb{R}$ ($i=1,...,d$).
\end{proposition}
\begin{proof}
Using Proposition \ref{prop.lpper} together with the jump relations (\ref{jumpper}), it is easily seen that $\chi_i$ is solution to (\ref{def.chiiclass}) if and 
only if it is of the form 
\begin{equation}\label{eq.repchii} 
\chi_i = {\mathcal S}_\omega^\#\phi_i + c,
\end{equation}
where $c \in \mathbb{R}$ and $\phi_i \in H^{-1/2}_0(\partial \omega)$ is solution to the integral equation: 
\begin{equation}\label{defpot}
 \left( \lambda I - {\mathcal K}^{\#*}_\omega\right) \phi_i = n_i, \text{ where }Ê\lambda = \frac{1}{2}\frac{a+1}{a-1}.
 \end{equation}
 On the other hand, by definition, the above equation has a unique solution in $L^2_0(\partial \omega)$ provided $a \notin \Sigma_\omega$. 
\end{proof}

In particular, this proposition implies that, when the inclusion $\omega$ is of class ${\mathcal C}^2$, the cell functions
may fail to exist for \textit{at most} countably many values of $a < 0$.

As a consequence of the above results, when $a \notin \Sigma_\omega$, the \textit{homogenized matrix} with inclusion $\omega$
is well-defined via the formula: 
\begin{equation}\label{eq.hommat}
A^*_{ij} = \int_Y{A(y)(e_i + \nabla_y \chi_i) \cdot (e_j + \nabla_y \chi_j) \:dy}.
\end{equation}

\begin{remark}
The cell functions $\chi_i$ correspond to potentials $\phi_i$ via the equation (\ref{defpot}) which features
the operator $(\lambda I - {\mathcal K}_\omega^{\# *})$ and the \textit{particular} right-hand sides $n_i$. 
It may very well happen that such potentials (and thus such cell functions) exist even though $a \in \Sigma_\omega$ 
(there is then a compatibility relation delivered by the Fredholm theory). 
This is for instance the case when the inclusion $\omega \subset Y$ corresponds to a rank $1$ laminate (notice however that this case violates the assumption $\omega \Subset Y$): in the appendix, we prove that only $a=-1$ - which happens to correspond to 
$\lambda = 0$, the only element
in the essential spectrum of ${\mathcal K}_\omega^{\#*}$ -  
is associated to ill-posed cell problems (\ref{def.chiiclass}).
We do not know whether this fact
is general or not.
\end{remark}
\subsubsection{Ellipticity of the homogenized tensor for high contrasts}\label{seccoeffhomog}~\\

%We have just seen in Theorem \ref{thgapev} that the potential equation (\ref{eq.pothc}) is well-posed for values of the conductivity that 
%are either very large (close to $-\infty$) of very small (negative and close to $0$), in a uniform fashion with respect to $\varepsilon$. 
%
%Owing to Proposition \ref{prop.gencvue}, it follows that, for such values of the conductivity, upon extraction of a subsequence, 
%the potential $u_\varepsilon$ converges weakly in $H^1_0(\Omega)$ to the solution $u^*$ of the homogenized system, 
%\begin{equation}\label{eq.pot*}
%\left\{
%\begin{array}{cc}
%-\text{\rm div}(A^*\nabla u)  = f & \text{in } \Omega\\
%u=0 & \text{on }Ê\partial \Omega
%\end{array}
%\right. 
%\end{equation}
%
%A natural question arising in this view concerns the well-posedness of the limiting system (\ref{eq.pot*}). 
%The purpose of this section is to prove that, actually, the homogenized matrix $A^*$ is elliptic in both settings of very large and low conductivities $a$.  
%
%To achieve this, we rely on the material of Section \ref{sec.cellpb}. 
%Based on Theorem \ref{thproppotper}, provided $a$ is either large or small enough, 
%we represent each cell function $\chi_i$, defined by (\ref{defchii}), as a single layer potential: $ \chi_i = {\mathcal S}_\omega^{\#}(\phi_i),$
%where the potential $\phi_i \in L^2_0(\partial \omega)$ satisfies the equation:
%\begin{equation}\label{defpot}
% \left( \lambda I - {\mathcal K}^{\#*}_\omega\right) \phi = n_i, \text{ where }Ê\lambda = \frac{1}{2}\frac{a+1}{a-1}.
% \end{equation}
 
It is well-known that the homogenized tensor $A^*$ is positive definite when $a \in \mathbb{R}$, $a>0$; see for instance \cite{blp,jikov}.    
In this section, we prove that the same property holds in the case where $a <0$ has small or large modulus.  
To achieve this, we rely on an alternative expression for the homogenized coefficients: 
 
\begin{proposition}
The coefficients (\ref{eq.hommat}) of the homogenized tensor $A^*$ read: 
\begin{equation}\label{homogcoefphi}
 A^*_{ij} = \delta_{ij}Ê + \int_{\partial \omega}{y_i \phi_j \:ds(y)},
 \end{equation}
where the potential $\phi_j \in L^2_0(\partial \omega)$ is the solution to (\ref{defpot}).
\end{proposition}
\begin{proof}
Using integration by parts and the transmission conditions of the system (\ref{def.chiitrans}), we obtain: 
$$ 
\begin{array}{>{\displaystyle}c c>{\displaystyle}l}
\int_Y{A(y) (e_i + \nabla_y \chi_i) \cdot (e_j + \nabla_y \chi_j) \:dy} &= & a \int_\omega{(e_i + \nabla_y \chi_i) \cdot (e_j + \nabla_y \chi_j) \:dy} + \int_{Y\setminus \overline{\omega}}{(e_i + \nabla_y \chi_i) \cdot (e_j + \nabla_y \chi_j) \:dy} \\
&=& \int_{\partial Y}{\left(n_i + \frac{\partial \chi_i}{\partial n_y} \right)y_j \:ds(y)}. 
\end{array}
$$
From this point, another integration by parts easily shows that: 
$$\int_{\partial Y}{n_i y_j \:ds(y)} = \delta_{ij}.$$ 
On the other hand, successive integration by parts yield: 
$$ 
\begin{array}{>{\displaystyle}c c>{\displaystyle}l}
\int_{\partial Y}{\frac{\partial \chi_i}{\partial n_y} y_j \:ds(y)} &=&  \int_{Y \setminus \overline{\omega}}{\nabla_y \chi_i \cdot \nabla_y y_j \:dy} + \int_{\partial \omega}{\frac{\partial \chi_{i}^+}{\partial n_y} y_j \:ds(y)} \\Ê
&=&  - \int_{\partial \omega}{\chi_i n_j \:ds(y)} + \int_{\partial \omega}{\frac{\partial \chi_{i}^+}{\partial n_y} y_j \:ds(y)} \\Ê
&=&  - \int_{\omega}{\nabla_y \chi_i \cdot e_j \:ds(y)} + \int_{\partial \omega}{\frac{\partial \chi_{i}^+}{\partial n_y} y_j \:ds(y)} \\Ê
&=&  - \int_{\partial \omega}{\frac{\partial \chi_{i}^-}{\partial n_y} y_j \:ds(y)} + \int_{\partial \omega}{\frac{\partial \chi_{i}^+}{\partial n_y} y_j \:ds(y)} \\Ê
\end{array}
$$
and the desired result now follows from the representation (\ref{eq.repchii}) and the jump relations (\ref{jumpper}).
\end{proof}

\begin{corollary}\label{cor.A*wp}
There exist constants $0 < m < M < \infty$ such that, if the conductivity $a$ belongs to $(-\infty,-M) \cup (-m,0)$, the homogenized tensor $A^*$ defined in (\ref{eq.hommat}) 
is positive definite. 
\end{corollary}
\begin{proof}
In this proof, we denote by $A^*_{ij}(a)$ the coefficients of the homogenized tensor (\ref{eq.hommat}) 
when the conductivity inside the inclusion $\omega$ is $a$.  
Combining the representation formula (\ref{homogcoefphi}) with (\ref{defpot}), we see that
\begin{equation}\label{eq.A*cont}
A^*_{ij}(a) = F_{ij}(\frac{1}{2}\frac{a+1}{a-1}), \text{ where } F_{ij}(\lambda) = \delta_{ij}Ê + \int_{\partial \omega}{y_i \phi_j(\lambda) \:ds(y)}, \text{ and } \phi_j(\lambda) \text{ satisfies (\ref{defpot})}.  
\end{equation}
Clearly, the mappings $\lambda\mapsto F_{ij}(\lambda)$ are continuous on $\mathbb{C} \setminus \sigma({\mathcal K}^{\#*}_\omega)$. 
Moreover, $\sigma({\mathcal K}^{\#*}_\omega)$ is a \textit{closed} subset of $(-\frac{1}{2},\frac{1}{2})$ by Theorem \ref{thproppotper}, $(iii)$. 
Hence, there exists $0<\alpha < 1/2$ such that the $F_{ij}$'s are continuous on $\mathbb{C} \setminus (-\alpha,\alpha)$. 
Using Lemma \ref{lem.posdefA*} below and the fact that:
$$F_{ij}(-\frac{1}{2}) = \lim\limits_{a \to 0}{A^*_{ij}(a)}, \text{ and } F_{ij}(\frac{1}{2}) = \lim\limits_{a \to +\infty}{A^*_{ij}(a)},$$
we see that there exists $\beta >0$ such that: 
$$ \forall \xi \in \mathbb{R}^d, \:\: F(-\frac{1}{2}) \xi \cdot \xi \geq \beta \lvert\xi\lvert^2, \text{ and }F(\frac{1}{2}) \xi \cdot \xi \geq \beta \lvert\xi\lvert^2.$$
Since positive definiteness is an open condition, there exists (another) $0<\alpha < 1/2$ such that $F(\lambda)$ 
is positive definite for $\lambda \in (-1/2,-\alpha) \cup (\alpha,1/2)$. In view of (\ref{eq.A*cont}), this is the expected conclusion. 
\end{proof}

\begin{lemma}\label{lem.posdefA*}
There exists $\beta > 0$ such that, for any value $a \in (0,+\infty)$, the homogenized tensor $A^*$ satisfies: 
$$ \forall \xi \in \mathbb{R}^d, \:\: A^* \xi \cdot \xi \geq \beta \lvert \xi \lvert^2.$$
\end{lemma}

\begin{proof}
For given $a \in (0,+\infty)$ and $\xi \in \mathbb{R}^d$, 
the usual Lax-Milgram theory for the elliptic equation (\ref{def.chiiclass}) yields: 
\begin{equation}\label{eq.Axixi}
 A^* \xi \cdot \xi = \int_Y{A(y) \lvert \nabla_y w_\xi + \xi \lvert^2 \:dy} = \min\limits_{w \in H^1_\#(Y)} {\int_Y{A(y) \lvert \nabla_y w + \xi \lvert^2 \:dy}},
 \end{equation} 
where $w_\xi$ is the unique function in $H^1_\#(Y) / \mathbb{R}$ such that:
$$ - \text{\rm div}(A(y)(\nabla_y w +\xi)) = 0.$$
The dual variational principle associated to (\ref{eq.Axixi}) now reads (see \cite{kohnmilton} in this precise context):
\begin{equation}\label{eq.Axixidual}
 A^* \xi \cdot \xi = \max\limits_{\sigma \in L^2_\#(Y)^d \atop \text{\rm div}(\sigma) = 0} \left(\int_Y{\sigma \cdot \xi \:dy} - \frac{1}{2} \int_Y{A(y)^{-1} \lvert \sigma \lvert^2 \:dy}\right),
 \end{equation} 
 and we proceed to construct a adequate `test flux' $\sigma$ for this identity. To this end, since $\omega \Subset Y$, we may introduce $w_t(\xi)$, the unique function $w$ in $H^1_\#(Y \setminus \overline{\omega}) / \mathbb{R}$ satisfying: 
 \begin{equation}\label{eq.defwtestdual}
 \left\{ 
 \begin{array}{cl}
 -\text{\rm div} (\nabla_y w_t + \xi) = 0 & \text{in }ÊY \setminus \overline{\omega}, \\ 
\frac{\partial w_t}{\partial n} + \xi \cdot n = 0 & \text{on } \partial \omega.
 \end{array}
 \right. 
 \end{equation} 
 We then define our test flux $\sigma_t(\xi) \in L^2_\#(Y)$ as: 
 $$ \sigma_t(\xi) = \left\{ 
 \begin{array}{cl}
 \nabla_y w_t(\xi) + \xi & \text{in } Y \setminus \overline{\omega}, \\
 0 & \text{in } \omega.
 \end{array}
 \right.$$  
 Since $\sigma_t(\xi)$ is obviously divergence-free, inserting it in (\ref{eq.Axixidual}) yields the ($a$-independent) lower bound for $A^*\xi \cdot \xi$:
 $$ A\xi\cdot \xi \geq \int_Y{\sigma_t(\xi) \cdot \xi \:dy} - \frac{1}{2}\int_Y{A(y)^{-1}\sigma_t(\xi) \cdot \sigma_t(\xi) \:dy} = \int_{Y \setminus \overline{\omega}}{\lvert \nabla_y w_t(\xi) + \xi\lvert^2 \:dy}.$$
Let us now remark that the mapping $ \xi \mapsto  \int_{Y \setminus \overline{\omega}}{\lvert \nabla_y w_t(\xi) + \xi\lvert^2 \:dy}$ is continuous on $\mathbb{R}^d$ and 
homogeneous of order $2$. Hence, to conclude the proof of the lemma, it is enough to prove that, for any $\xi \in \mathbb{R}^d$, $\lvert \xi \lvert =1$, the integral $ \int_{Y \setminus \overline{\omega}}{\lvert \nabla_y w_t(\xi) + \xi\lvert^2 \:dy}$ does not vanish. 

To achieve this, we proceed by contradiction and assume that there exists $\xi\in \mathbb{R}^d$, $\lvert \xi \lvert =1$ such that: 
$$  \int_{Y \setminus \overline{\omega}}{\lvert \nabla_y w_t(\xi) + \xi\lvert^2 \:dy} = 0.$$
Then, $\nabla_y w_t(\xi) + \xi = 0$ a.e. on $Y \setminus \overline{\omega}$. Since this set is connected, we infer that there exists a constant $c \in \mathbb{R}$ such that $w_t(\xi)(y) = -\xi \cdot y + c$, which is impossible since $y \mapsto w_t(\xi)(y)$ is $Y$-periodic.   
\end{proof}

\begin{remark}
Corollary \ref{cor.A*wp} may fail if the assumption $\omega \Subset Y$ is removed. For instance, in the case of rank $1$ laminates, the calculations performed in Appendix \ref{sec.laminates} reveal that the homogenized tensor $A^*$ becomes hyperbolic as $a \to -\infty$ (i.e. it has sign-changing eigenvalues).
\end{remark}

%%%%%%%%%%%%%%%%%%%%%%%%%%%%%%
\subsection{General results for the direct problem}\label{sec.gendirect}~\\
%%%%%%%%%%%%%%%%%%%%%%%%%%%%%%

As we have mentioned, the behavior of solutions to the system (\ref{eq.pot})
when the conductivity $a$ inside the inclusions is positive is well-known. 
In the case where $a<0$, the system (\ref{eq.pot}) may
be ill-posed for some values of $\varepsilon$ and some of
source terms $f$. 
Nevertheless, our first result states that the possible limits of sequences $u_\varepsilon$ of solutions to this system are governed by the formally homogenized matrix $A^*$ defined in (\ref{eq.hommat}). 

The first result in this section expresses in our context the well-known fact that the homogenization mechanism applies as soon 
as the sequence $u_\varepsilon$ is bounded.

\begin{proposition}\label{prop.gencvue}
Let $a \in \mathbb{C}\setminus \Sigma_{\omega}$, and let $f \in H^{-1}(\Omega)$.  
Assume that there exists a sequence (indexed by $\varepsilon$) of solutions $u_\varepsilon$ to (\ref{eq.pot}) such that: 
\begin{equation}\label{eq.boundueps}
 \lvert\lvert \nabla u_\varepsilon \lvert\lvert_{L^2(\Omega)^d} \leq C,
 \end{equation}
for a constant $C>0$ independent of $\varepsilon$. 
Then, up to a subsequence (still indexed by $\varepsilon$), there exists $u_0 \in H^1_0(\Omega)$ such that: 
\begin{equation}\label{cvuepsu0}
 u_\varepsilon \xrightarrow{\varepsilon \to 0} u_0, \text{ weakly in } H^1_0(\Omega).
 \end{equation}
The function $u_0$ is a solution to the formally homogenized system (\ref{eq.pothom}). 
\end{proposition} 
\begin{proof}
Let $g \in H^1_0(\Omega)$ be the representative of $f \in H^{-1}(\Omega)$ supplied by the Riesz representation theorem: 
$$ \forall v \in H^1_0(\Omega), \:\: \int_\Omega{\nabla g \cdot \nabla v \:dx} = \langle f , v \rangle_{H^{-1}(\Omega), H^1_0(\Omega)},$$
so that the variational formulation for (\ref{eq.pot}) reads: 
\begin{equation}\label{eq.fvarpotriesz}
\forall v \in H^1_0(\Omega), \:\: \int_{\Omega}{A_\varepsilon(x) \nabla u_\varepsilon \cdot \nabla v \:dx} = \int_\Omega{\nabla g \cdot \nabla v \:dx}.
\end{equation}

Since $\lvert\lvert u_\varepsilon \lvert\lvert_{H^1_0(\Omega)} \leq C$, there exists $u_0 \in H^1_0(\Omega)$ such that the convergences (\ref{cvuepsu0}) hold, 
up to a subsequence which we still label by $\varepsilon$. 
Also, using Theorem \ref{thcompactunfold}, there exists $u_1(x,y) \in L^2(\Omega, H^1_\#(Y) / \mathbb{R})$ such that: 
$$ E_\varepsilon ( \nabla u_\varepsilon ) \to \nabla u_0 + \nabla_y u_1 \text{ weakly in } L^2(\Omega \times Y),$$
where $E_\varepsilon$ is the extension operator of Definition \ref{def.unfold}.   
Now, for an arbitrary $v \in H^1_0(\Omega)$, (\ref{eq.fvarpotriesz}) becomes: 
$$\begin{array}{>{\displaystyle}cc>{\displaystyle}l}
 \int_{\Omega}{A_\varepsilon(x) \nabla u_\varepsilon \cdot \nabla v\:dx} &=& \int_{{\mathcal O}_\varepsilon}{A(\frac{x}{\varepsilon})\nabla u_\varepsilon \cdot \nabla v \:dx} + \int_{{\mathcal B}_\varepsilon}{\nabla u_\varepsilon \cdot \nabla v\:dx}, \\
 &=& \int_{\Omega \times Y}{A(y) E_\varepsilon( \nabla u_\varepsilon ) \cdot E_\varepsilon( \nabla v)  \:dxdy} + \int_{{\mathcal B}_\varepsilon}{\nabla u_\varepsilon \cdot \nabla v\:dx}. 
 \end{array}$$
Notice that, with a slight abuse of notations, in the first line of the above formula, 
we have used the same notation for the conductivity tensor $A$ (which is defined on $Y$) and for its $Y$-periodic extension to $\mathbb{R}^d$. 
Using Lebesgue's dominated convergence theorem, in combination with the bound (\ref{eq.boundueps}), 
it is easily seen that the second integral in the right-hand side of the above formula vanishes as $\varepsilon \rightarrow 0$.
As for the first one, since $E_\varepsilon (\nabla v) \rightarrow \nabla v $ strongly in $L^2(\Omega \times Y)$ (see Proposition \ref{prop.unfold}, $(v)$), we see that:
$$ \int_{\Omega}{A_\varepsilon(x) \nabla u_\varepsilon \cdot \nabla v\:dx}  \xrightarrow{\varepsilon\rightarrow 0} \int_{\Omega \times Y}{A(y)(\nabla u_0 + \nabla_y u_1 )\cdot \nabla v \:dxdy}.$$
Therefore, taking limits in the variational problem (\ref{eq.fvarpotriesz}), we obtain: 
$$\int_{\Omega \times Y}{A(y)(\nabla u_0(x) + \nabla_y u_1(x,y) )\cdot \nabla v(x) \:dxdy}  = \int_\Omega{\nabla g \cdot \nabla v \:dx}.$$

By the same token, using a test function of the form $v_\varepsilon(x) := \varepsilon \phi(x) \psi(\frac{x}{\varepsilon})$ in (\ref{eq.fvarpotriesz}), 
for some given $\phi \in {\mathcal D}(\Omega)$ and $\psi\in H^1_\#(Y)$, and using Proposition \ref{prop.unfold}, we get: 
$$\int_{\Omega \times Y}{\phi(x)A(y)(\nabla u_0(x) + \nabla_y u_1(x,y) )\cdot \nabla_y \psi(y)) \:dxdy}  = \int_{\Omega \times Y}{\phi(x) \nabla g(x) \cdot \nabla_y \psi(y) \:dxdy} = 0.$$
All things considered, a standard density result yields that the pair $(u_0, u_1) \in {\mathcal H} := H^1_0(\Omega) \times L^2(\Omega, H^1_\#(Y)/\mathbb{R})$ 
is a solution to the problem: 
\begin{equation}\label{pb2scale}
 \forall (\varphi, \psi) \in {\mathcal H}, \:\: \int_{\Omega \times Y}{A(y)(\nabla u_0(x) + \nabla_y u_1(x,y) )\cdot (\nabla \varphi(x) + \nabla_y\psi(x,y) )  \:dxdy}  = \int_\Omega{\nabla g \cdot \nabla \varphi \:dx}.
 \end{equation}
In particular, this implies that, for any function $\psi \in L^2(\Omega, H^1_\#(Y) / \mathbb{R})$, 
$$\int_{\Omega \times Y}{A(y)(\nabla u_0(x) + \nabla_y u_1(x,y) )\cdot  \nabla_y \psi (x,y) \:dxdy}  = 0;$$
in other words,
\begin{equation}\label{eq.characu1}
 \int_Y{A(y)\left( \nabla u_0(x) + \nabla_y u_1(x,y)  \right) \cdot \nabla_y \psi(y) \:dy} =0, \text{ a.e. in } x \in \Omega, \text{ for any } \psi(y) \in H^1_\#(Y).
 \end{equation}
But for almost every $x \in \Omega$, since $a \notin \Sigma_\omega$, 
the variational problem (\ref{eq.characu1}) has a unique solution $u_1(x,\cdot) \in H^1_\#(Y) / \mathbb{R}$, which is: 
$$ u_1(x,y) = \sum\limits_{i=1}^d{\frac{\partial u_0}{\partial x_i}(x) \chi_i(y)},$$
where the $\chi_i$ are the cell functions defined by (\ref{def.chiiclass}). Recalling the expression (\ref{eq.hommat}) of the homogenized tensor $A^*$, (\ref{pb2scale}) rewrites:
$$ \forall v \in H^1_0(\Omega), \:\: \int_\Omega{A^*\nabla u_0 \cdot \nabla v \:dx} =  \int_\Omega{\nabla g \cdot \nabla v \:dx},$$
which is the desired result.
\end{proof}

\begin{remark}
\noindent\begin{itemize}
\item In the cases where $a \in \mathbb{R}$, $a>0$, or $a$ has a complex value with non zero imaginary part, 
the bound (\ref{eq.boundueps}) is automatically satisfied as a result of the standard a priori estimate for (\ref{eq.pot}).
\item Proposition \ref{prop.gencvue} holds in the more general situation when the source $f$ varies with $\varepsilon$, 
and more precisely, if $u_\varepsilon$ satisfies (\ref{eq.pot})
with $f$ replaced by $f_\varepsilon \in H^{-1}(\Omega)$, when the sequence $f_\varepsilon$ converges pointwise to some $f \in H^{-1}(\Omega)$, i.e. when 
$$\forall v \in H^1_0(\Omega), \:\: \langle f_\varepsilon, v \rangle_{H^{-1}(\Omega),H^1_0(\Omega)} \xrightarrow{\varepsilon \to 0} \langle f, v \rangle_{H^{-1}(\Omega),H^1_0(\Omega)}.$$
\item Likewise, the same result holds if the conductivity $a \notin \Sigma_\omega$ inside $\omega_\varepsilon$
is replaced by a sequence $a_\varepsilon$ which converges to $a$, i.e. if: 
$$ A_\varepsilon(x) = \left\{Ê
\begin{array}{cc}
1 & \text{if }Êx \in \Omega \setminus \overline{\omega_\varepsilon},\\
a_\varepsilon & \text{if }Êx \in \omega_\varepsilon. 
\end{array}
\right..$$
For instance, if $a\in \mathbb{R}$, $a<0$, and $a_\varepsilon= a+i\delta_\varepsilon$, where the real sequence $\delta_\varepsilon$ vanishes, 
the conclusion of Proposition \ref{prop.gencvue} holds and for any value $\varepsilon >0$, the system (\ref{eq.pot}) is well-posed.
\item  Proposition \ref{prop.gencvue} even holds without assuming that $\omega$ is smooth and $\omega \Subset Y$. In such a general context, 
the set $\Sigma_\omega$ may of course be larger than a mere sequence. 
\end{itemize}
\end{remark}
We now turn to the question of what information can be be gleaned from the homogenized system (\ref{eq.pothom}). 
The following proposition indicates that \textit{any} solution $u_0$ to (\ref{eq.pothom}) (there might be none, or many, 
depending on $A^*$ and the right-hand side $f$) can be attained as a limit of solutions $u_\varepsilon$ to (\ref{eq.pot}) for some appropriate right-hand sides.
Hence, there is no way to define a `natural' notion of solution to (\ref{eq.pothom}) as that arising from a limiting process of the form (\ref{eq.pot}). 

\begin{proposition}\label{prop.limitrecip}
Let $a \in \mathbb{C} \setminus \Sigma_\omega$, so that the homogenized tensor $A^*$ is well-defined by (\ref{eq.hommat}), and let $f\in H^{-1}(\Omega)$.
Let $u_0 \in H^1_0(\Omega)$ be one solution (if any) to the system: 
$$ \left\{ \begin{array}{cl}
-\text{\rm div}(A^*\nabla u_0) = f & \text{in } \Omega, \\
u_0 = 0 & \text{on } \partial \Omega.
\end{array}\right.$$ 
Let $a_\varepsilon \in \mathbb{C} \setminus \Sigma_\omega$ be any sequence such that $a_\varepsilon \to a$ and let: 
$$ A_\varepsilon(x) = \left\{
\begin{array}{cl}
1 & \text{if }Êx \in \Omega \setminus \overline{\omega_\varepsilon}, \\
a_\varepsilon & \text{if }Êx \in \omega_\varepsilon.
\end{array}
\right.$$
Then there exists a sequence $f_\varepsilon \in H^{-1}(\Omega)$ of sources converging pointwise to $f$, 
and a sequence $u_\varepsilon \in H^1_0(\Omega)$ of associated voltage potentials: 
\begin{equation}\label{eq.potnonat} 
\left\{ \begin{array}{cl}
-\text{\rm div}(A_\varepsilon(x)\nabla u_\varepsilon) = f_\varepsilon & \text{in } \Omega, \\
u_\varepsilon = 0 & \text{on } \partial \Omega.
\end{array}\right.
\end{equation}
such that $u_\varepsilon \to u_0$ weakly in $H^1_0(\Omega)$.
\end{proposition}
\begin{proof}
For the sake of simplicity, we only deal with the case where the sequence $a_\varepsilon \equiv a$ (the general case being no more difficult). 
The proof consists in the construction of a sequence $u_\varepsilon \in H^1_0(\Omega)$ in such a way that $u_\varepsilon \to u_0$ weakly in $H^1_0(\Omega)$, 
and that $f_\varepsilon := -\text{\rm div}(A_\varepsilon(x) \nabla u_\varepsilon) \in H^{-1}(\Omega)$ converges to $f$ in the sense of distributions. 
Relying on a classical intuition in homogenization theory, we define $u_\varepsilon$ as an oscillating sequence around $u_0$ using the cell functions $\chi_i$ introduced in Section \ref{sec.cellpb} (which are well-defined since $a \notin \Sigma_\omega$). 

Notice that, since the inclusion $\omega$ is of class ${\mathcal C}^2$, 
standard arguments in elliptic regularity theory using the well-posedness of the cell problem (see e.g. \cite{brezis}) 
lead to the fact that $\chi_i\in W^{1,\infty}(Y)$. 

Let us now proceed with the definition of $u_\varepsilon \in H^1_0(\Omega)$: we use a variation of the usual corrector formula for taking into account oscillations of the solutions to (\ref{eq.potnonat}) at the microscopic scale: 
\begin{equation}\label{eq.corueps} 
u_\varepsilon(x) = u_0(x) + \varepsilon \zeta_\varepsilon(x) \sum\limits_{i=1}^d{I_\varepsilon\left(\frac{\partial u_0}{\partial x_i} \right) \chi_i(\frac{x}{\varepsilon})}.
\end{equation}
In this formula, $\zeta_\varepsilon$ is a smooth cutoff function such that: 
\begin{equation}\label{propsetae}
 \zeta_\varepsilon(x) = 1 \text{ if } d(x, \partial \Omega) > \varepsilon, \: \zeta_\varepsilon(x) = 0 \text{ on } \partial \Omega, \text{ and } \lvert\lvert \nabla \zeta_\varepsilon \lvert\lvert_{L^\infty(\Omega)} \leq \frac{C}{\varepsilon}.
 \end{equation}
The operator $I_\varepsilon$ is defined according to \cite{griso,unfold}, and maps $L^p(\mathbb{R}^d)$ into $W^{1,\infty}(\mathbb{R}^d,\mathbb{R}^d)$ for $p \in [1,\infty]$;
it is used in our context to enforce that $u_\varepsilon$ belongs to $H^1_0(\Omega)$ even though the partial derivatives $\frac{\partial u_0}{\partial x_i}$ are only in $L^2(\Omega)$.
For $u \in L^p(\mathbb{R}^d)$, $I_\varepsilon u$ is inspired by the usual $\mathbb{Q}_1$ interpolation operator in the Finite Element theory: 
\begin{equation}\label{eq.Ie1}
 \forall \xi \in \mathbb{Z}^d, \:\: I_\varepsilon u(\varepsilon \xi) = \int_Y{u(\varepsilon \xi + \varepsilon y)\:dy}, \text{ and } I_\varepsilon u \text{ is a }\mathbb{Q}_1 \text{ function on each cell } Y_\varepsilon^\xi;
 \end{equation}
more precisely:
\begin{equation}\label{eq.Ie2}
\forall \xi \in \mathbb{Z}^d,\:\: \forall z \in Y, \:\: I_\varepsilon u(\varepsilon \xi + \varepsilon z) = \sum\limits_{\alpha \in \left\{0,1 \right\}^d}{I_\varepsilon u ( \varepsilon \xi + \varepsilon (\alpha_1,...,\alpha_d) ) \: p_{\alpha_1}(z_1)... p_{\alpha_d}(z_d)},
\end{equation}
where $p_0(t) := 1-t$ and $p_1(t) := t$.

Several mapping properties of $I_\varepsilon$ are reported in Lemma \ref{lem.mapIeps} below.

Let us now prove that $f_\varepsilon := -\text{\rm div}(A_\varepsilon \nabla u_\varepsilon) \in H^{-1}(\Omega)$ converges pointwise to $f$.  
Until the end of the proof, $r_\varepsilon$ stands for a sequence of real numbers (possible changing from line to line) converging to $0$ as $\varepsilon \to 0$. 
For an arbitrary test function $\varphi \in H^1_0(\Omega)$, one has:
$$
\begin{array}{>{\displaystyle}cc>{\displaystyle}l}
\langle -\text{\rm div}(A_\varepsilon \nabla u_\varepsilon) , \varphi \rangle_{H^{-1}(\Omega),H^1_0(\Omega)}&=& \int_\Omega{A_\varepsilon(x) \nabla u_\varepsilon \cdot \nabla \varphi \:dx} \\
&=& \int_\Omega{A_\varepsilon(x) \left(\nabla u_0 + \zeta_\varepsilon\sum\limits_{i=1}^d{I_\varepsilon\left( \frac{\partial u_0}{\partial x_i}\right) \nabla_y \chi_i(\frac{x}{\varepsilon}) } \right) \cdot \nabla \varphi \:dx} \\
&&+ \varepsilon \int_{\Omega}{A_\varepsilon(x) \left(\sum\limits_{i=1}^d{I_\varepsilon\left(\frac{\partial u_0}{\partial x_i} \right) \chi_i(\frac{x}{\varepsilon})} \right) \nabla \zeta_\varepsilon \cdot \nabla \varphi \:dx} \\
&& + \varepsilon \int_\Omega{ \zeta_\varepsilon A_\varepsilon(x) \left(\sum\limits_{i=1}^d{ \chi_i(\frac{x}{\varepsilon}) \nabla\left( I_\varepsilon\left(\frac{\partial u_0}{\partial x_i} \right)\right) } \right) \cdot \nabla \varphi \:dx}.
\end{array}
$$
Using (\ref{propsetae}) together with the facts that $I_\varepsilon \left(\frac{\partial u_0}{\partial x_i} \right) \in W^{1,\infty}(\mathbb{R}^d)$ and $\chi_i \in W^{1,\infty}(Y)$, the second integral in the right-hand side can be estimated by: 
$$ \left\lvert   \varepsilon \int_{\Omega}{A_\varepsilon(x) \left(\sum\limits_{i=1}^d{I_\varepsilon\left(\frac{\partial u_0}{\partial x_i} \right) \chi_i(\frac{x}{\varepsilon})} \right) \nabla \zeta_\varepsilon \cdot \nabla \varphi \:dx}\right\lvert \leq C \sqrt{\varepsilon}\lvert\lvert \nabla \varphi \lvert\lvert_{L^2(\Omega)^d}.$$
In a similar way, we see that:
$$
\begin{array}{cc>{\displaystyle}l}
\langle -\text{\rm div}(A_\varepsilon \nabla u_\varepsilon) , \varphi \rangle_{ H^{-1}(\Omega) , H^1_0(\Omega)} &=& \int_{{\mathcal O}_\varepsilon}{A(\frac{x}{\varepsilon}) \left(\nabla u_0 + \sum\limits_{i=1}^d{I_\varepsilon\left( \frac{\partial u_0}{\partial x_i}\right) \nabla_y \chi_i(\frac{x}{\varepsilon}) } \right) \cdot \nabla \varphi \:dx} \\ 
&&+ \int_{{\mathcal O}_\varepsilon}{\varepsilon A(\frac{x}{\varepsilon}) \left(\sum\limits_{i=1}^d{\chi_i(\frac{x}{\varepsilon}) \nabla \left( I_\varepsilon\left(\frac{\partial u_0}{\partial x_i} \right)\right) } \right) \cdot \nabla \varphi \:dx} + r_\varepsilon \lvert\lvert \nabla \varphi \lvert\lvert_{L^2(\Omega)^d}.
\end{array}
$$
Using Lemma \ref{lem.mapIeps} and the Lebesgue dominated convergence theorem, it follows that: 
$$
\langle -\text{\rm div}(A_\varepsilon \nabla u_\varepsilon) , \varphi \rangle_{H^{-1}(\Omega),H^1_0(\Omega)} 
= \int_{{\mathcal O}_\varepsilon}{A(\frac{x}{\varepsilon}) \left(\nabla u_0 + \sum\limits_{i=1}^d{\left( \frac{\partial u_0}{\partial x_i}\right) \nabla_y \chi_i(\frac{x}{\varepsilon}) } \right) \cdot \nabla \varphi \:dx} 
+ r_\varepsilon \lvert\lvert \nabla \varphi \lvert\lvert_{L^2(\Omega)^d}.
$$
Finally, rescaling using Proposition \ref{prop.unfold} yields:
$$
\begin{array}{cc>{\displaystyle}l}
\langle -\text{\rm div}(A_\varepsilon \nabla u_\varepsilon) , \varphi \rangle_{ H^{-1}(\Omega) , H^1_0(\Omega)}
&\hspace{-0.7cm}=&\hspace{-0.7cm} \int_{\Omega \times Y}{A(y)  \sum\limits_{i=1}^d{\left(e_i + \nabla_y \chi_i(y) \right)\cdot e_j\: E_\varepsilon\left( \frac{\partial u_0}{\partial x_i} \right)E_\varepsilon\left(\frac{\partial \varphi}{\partial x_j}\right) \:dxdy}} + r_\varepsilon \lvert\lvert \nabla \varphi \lvert\lvert_{L^2(\Omega)^d}. \\
&\xrightarrow{\varepsilon \rightarrow 0} & \int_{\Omega}{A^*\nabla u_0 \cdot \nabla \varphi \:dx}.
\end{array}
$$
Using similar calculations, it is easily verified that $u_\varepsilon \rightarrow u_0$ weakly in $H^1_0(\Omega)$, which ends the proof.
\end{proof}

In the course of the proof, we used the following convergence results.
\begin{lemma}\label{lem.mapIeps}
Let $p \in [1,\infty]$; there exists a constant $C >0$ independent of $\varepsilon$ such that, for any $u \in L^p(\mathbb{R}^d)$,
\begin{equation}\label{eq.estQe} 
\lvert\lvert I_\varepsilon u \lvert\lvert_{L^p(\mathbb{R}^d)} \leq C \lvert\lvert u \lvert\lvert_{L^p(\mathbb{R}^d)}, \text{ and } \lvert\lvert \nabla I_\varepsilon u \lvert\lvert_{L^p(\mathbb{R}^d)^d} \leq \frac{C}{\varepsilon} \lvert\lvert u \lvert\lvert_{L^p(\mathbb{R}^d)}.
\end{equation}
Moreover, if $p\in [1,\infty[$, for any $u \in L^p(\mathbb{R}^d)$, the following convergences hold:
\begin{equation} \label{eq.cvQe}
I_\varepsilon u \stackrel{\varepsilon\to 0}{\longrightarrow} u \text{ strongly in } L^p(\mathbb{R}^d), \text{ and }Ê \varepsilon \nabla (I_\varepsilon u) \stackrel{\varepsilon\to 0}{\longrightarrow} 0 \text{ strongly in } L^p(\mathbb{R}^d)^d.
\end{equation}
\end{lemma}
\begin{proof}
The verification of both estimates (\ref{eq.estQe}) is elementary (albeit a little tedious) given the definition (\ref{eq.Ie1} - \ref{eq.Ie2}) of $I_\varepsilon$. 

Using the density of ${\mathcal D}(\mathbb{R}^d)$ in $L^p(\mathbb{R}^d)$, it is enough to check the convergences (\ref{eq.cvQe}) 
in the particular case where $u \in {\mathcal D}(\mathbb{R}^d)$. Again, this is verified in an elementary way, using Taylor expansions from (\ref{eq.Ie1}) and (\ref{eq.Ie2}). 
See \cite{unfold}, Prop. 4.2 for details.
\end{proof}\par
\smallskip

%%%%%%%%%%%%%%%%%%%%%%%%%%%%%%%
\subsection{Partial identification of the limit spectrum in terms of the homogenized tensor}\label{sec.hommatlimspec}~\\
%%%%%%%%%%%%%%%%%%%%%%%%%%%%%%%%

The following proposition identifies the limit spectrum $\lim_{\varepsilon \rightarrow 0}{\sigma(T_\varepsilon)}$
as the set of those values of the conductivity $a$ for which there is a source $f$ causing the potential $u_\varepsilon$ to blow up.

\begin{proposition}\label{propcaracsteps}
Let $a \in \mathbb{C} \setminus \left\{Ê0 \right\}$. Then $a$ belongs to the limit spectrum $\lim_{\varepsilon \to 0}{\sigma(T_\varepsilon)}$
if and only if there exists $f \in H^{-1}(\Omega)$ and a sequence $f_\varepsilon \in H^{-1}(\Omega)$ with $f_\varepsilon \to f$ pointwise, such that the solution $u_\varepsilon \in H^1_0(\Omega)$ of (\ref{eq.pot}) with $f_\varepsilon$ as a source term
satisfies $\lvert\lvert \nabla u_\varepsilon \lvert\lvert_{L^2(\Omega)^d} \rightarrow +\infty$.
\end{proposition}

\begin{proof}
Let $g_\varepsilon \in H^1_0(\Omega)$ be supplied by the Riesz representation theorem so that: 
$$ \forall v \in H^1_0(\Omega), \:\: \int_\Omega{\nabla g_\varepsilon \cdot \nabla v \:dx} = \langle f_\varepsilon, v \rangle_{H^{-1}(\Omega),H^1_0(\Omega)}.$$
Since $f_\varepsilon\to f$ in the sense of distributions, it follows from the uniform boundedness principle that there exists $C >0$ independent of $\varepsilon$ such that: 
$ \lvert\lvert g_\varepsilon \lvert\lvert_{H^1_0(\Omega)} \leq C$. 

We recall that (\ref{eq.pot}) is equivalent to
$$ \left(\frac{1}{1-a}I - T_\varepsilon \right) u_\varepsilon = \frac{1}{1-a} \: g_\varepsilon.$$
Letting $v_\varepsilon := \frac{u_\varepsilon}{\lvert\lvert u_\varepsilon \lvert\lvert_{H^1_0(\Omega)}} $ and $h_\varepsilon := \frac{g_\varepsilon}{\lvert\lvert u_\varepsilon \lvert\lvert_{H^1_0(\Omega)}}  \in H^1_0(\Omega)$, this in turn is equivalent to: 
$$ \left(\frac{1}{1-a}I - T_\varepsilon \right) v_\varepsilon = h_\varepsilon, \text{ where } \lvert\lvert v_\varepsilon \lvert\lvert_{H^1_0(\Omega)} = 1,$$
and $h_\varepsilon \to 0$ strongly in $H^1_0(\Omega)$ if and only if $\lvert\lvert \nabla u_\varepsilon \lvert\lvert_{L^2(\Omega)^d} \to \infty$. 
The result now stems from the abstract Lemma \ref{lemweylasymp} below, which is a refinement of the usual sequential characterization of the spectrum of self-adjoint operators. 
\end{proof}

\begin{lemma}\label{lemweylasymp}
Let $T_\varepsilon : H \rightarrow H$ be a sequence of self-adjoint operators on a Hilbert space $H$. 
Then $\lambda \in \mathbb{C}$ belongs to the limit spectrum $\lim_{\varepsilon\rightarrow 0}{\sigma(T_\varepsilon)}$ if and only if there exists a subsequence, still denoted by $\varepsilon$, and $v_\varepsilon \in H$ such that:
\begin{equation}\label{eq.weyllim}
 \lvert\lvert v_\varepsilon \lvert\lvert = 1, \text{ and } \lvert\lvert \lambda v_\varepsilon - T_\varepsilon v_\varepsilon \lvert\lvert \xrightarrow{\varepsilon \rightarrow 0} 0.
 \end{equation}
\end{lemma}
\begin{proof}
Let $\lambda \in \lim_{\varepsilon\rightarrow 0}{\sigma(T_\varepsilon)}$;
by definition, there exists $\lambda_\varepsilon \in \sigma(T_\varepsilon)$ such that $\lambda_\varepsilon \rightarrow \lambda$. 
Now, for a given value $\varepsilon >0$, there exists a sequence $v_\varepsilon^k$, $k \in \mathbb{N}$ such that: 
$$ \lvert\lvert v_\varepsilon^k \lvert\lvert = 1, \text{ and } \lvert\lvert \lambda_\varepsilon v_\varepsilon^k - T_\varepsilon v_\varepsilon^k
 \lvert\lvert \xrightarrow{k\rightarrow \infty} 0.$$
In particular, for a given $\varepsilon >0$, there exists $w_\varepsilon \in H$ such that $\lvert\lvert w_\varepsilon \lvert\lvert =1$, and $\lvert\lvert \lambda_\varepsilon w_\varepsilon - T_\varepsilon w_\varepsilon \lvert\lvert \leq \varepsilon$.
Then, 
$$ \begin{array}{ccl}
\lvert\lvert \lambda w_\varepsilon - T_\varepsilon w_\varepsilon \lvert\lvert &\leq& \lvert\lvert \lambda_\varepsilon w_\varepsilon - T_\varepsilon w_\varepsilon \lvert\lvert  + \lvert\lambda - \lambda_\varepsilon \lvert, \\
&\leq& \varepsilon + \lvert\lambda - \lambda_\varepsilon \lvert;
\end{array}$$
that is, $w_\varepsilon$ satisfies (\ref{eq.weyllim}).

Conversely, let $\lambda \in \mathbb{C}$ be such that there exists a sequence $v_\varepsilon$ such that (\ref{eq.weyllim}) holds,
and assume that $\lambda \notin \lim_{\varepsilon\rightarrow 0}{\sigma(T_\varepsilon)}$.  
Then, there exists $\delta >0$ such that, for $\varepsilon >0$ small enough: 
$$ \forall \mu \in \sigma(T_\varepsilon), \:\: \lvert \mu - \lambda \lvert > \delta . $$
It follows from the spectral theorem for self-adjoint operators (see, e.g. \cite{rudin}, Th. $12.22$, and Proposition \ref{prop.speclsc}), that, for $\varepsilon >0$ small enough: 
$$ \forall v \in H, \:\: \lvert\lvert \lambda v - T_\varepsilon v \lvert\lvert \geq \delta \lvert\lvert v \lvert\lvert.$$
Using $v=v_\varepsilon$ in the above inequality yields $\lvert\lvert \lambda v_\varepsilon - T_\varepsilon v_\varepsilon 
\lvert\lvert \geq \delta$, in contradiction with the initial hypothesis.
\end{proof}

The following corollary now offers an interpretation of the limit spectrum in terms of the homogenized matrix $A^*$.

\begin{corollary}\label{corspeclim}
Let $a \in \mathbb{C} \setminus \Sigma_\omega$, and let $A^*$ be the corresponding homogenized matrix, given by (\ref{eq.hommat}). 
Assume that there exists $f \in H^{-1}(\Omega)$ such that the system: 
\begin{equation}\label{eq.sysA*}
 \left\{ 
\begin{array}{cl}
- \text{\rm div}(A^*\nabla u)  = f & \text{in } \Omega, \\
u = 0 & \text{on } \partial \Omega,
\end{array}
\right.
\end{equation}
does not have a solution. Then $a \in \lim_{\varepsilon\rightarrow 0}{\sigma(T_\varepsilon)}$. 
\end{corollary}
\begin{proof}
Let $\delta_\varepsilon$ be a sequence of positive numbers going to $0$ as $\varepsilon \rightarrow 0$, 
and let $u_\varepsilon \in H^1_0(\Omega)$ be the unique solution to the system: 
$$ \left\{ 
\begin{array}{cl}
- \text{\rm div}(B_\varepsilon\nabla u)  = f & \text{in } \Omega, \\
u = 0 & \text{on } \partial \Omega,
\end{array}
\right. , \text{ where } B_\varepsilon(x)  = \left\{ 
\begin{array}{cl}
1 & \text{if }Êx \in \Omega \setminus \overline{\omega_\varepsilon}, \\
a+ i\delta_\varepsilon & \text{if } x \in \omega_\varepsilon.
\end{array} 
\right. ,$$
as supplied by the classical Lax-Milgram theory. 

Then, one necessarily has $\lvert\lvert \nabla u_\varepsilon \lvert\lvert_{L^2(\Omega)^d} \rightarrow \infty$. 
Indeed, if a subsequence of $u_\varepsilon$ were to be bounded, Proposition \ref{prop.gencvue} would imply that 
the weak limit of this subsequence satisfies (\ref{eq.sysA*}). 

The desired conclusion is then a consequence of Proposition \ref{propcaracsteps}. 
\end{proof}
\begin{remark}
In particular, the conclusion of Corollary \ref{corspeclim} holds under the (stronger) hypothesis that there exists $u \neq 0$ in $H^1_0(\Omega)$ 
such that: 
$$ -\text{\rm div}(A^*\nabla u) = 0, \text{or equivalently } \int_\Omega{A^*\nabla u \cdot \nabla v \:dx} = 0 , \:\: \forall v \in H^1_0(\Omega).$$
This last point is a consequence of the self-adjointness of the operator $T^*: H^1_0(\Omega) \rightarrow H^1_0(\Omega)$ defined by:
$$ \forall v \in H^1_0(\Omega), \:\: \int_\Omega{\nabla (T^*u) \cdot \nabla v \:dx} = \int_{\Omega}{A^*\nabla u \cdot \nabla v \:dx},$$
which implies that: $\overline{\text{\rm Im}(T^*)} = \text{\rm Ker}(T^*)^\perp$.
\end{remark}

\begin{remark}
The conditions expressed in Proposition \ref{propcaracsteps} and Corollary \ref{corspeclim}
are conditions on the geometry of the inclusion pattern $\omega$ (which defines the homogenized tensor $A^*$) 
and on that of the macroscopic domain $\Omega$. See for instance \cite{john} for a study of the well-posedness of (\ref{eq.sysA*}) in the case where $A^*$ is hyperbolic.
\end{remark}

%%%%%%%%%%%%%%%%%%%%%%%%%%%%%%%
\subsection{Well-posedness for the conductivity equation with high contrast}\label{sec.hc}~\\
%%%%%%%%%%%%%%%%%%%%%%%%%%%%%%%

In this section, we take advantage of the material of Sections \ref{sec.boundev} and \ref{sec.cellpb} to investigate in closer detail the case of a high contrast
between the conductivity $1$ outside the set $\omega_\varepsilon$ of inclusions and the conductivity $a$ inside this set, i.e. the situations $a <0$, $a\to 0$ or $a \to -\infty$. 

Our main result in this direction is the following:

\begin{theorem}\label{thgapev}
There exists a constant $0 < \alpha$ such that, if the conductivity $a$ belongs to $(-\infty,-1/\alpha) \cup (-\alpha, 0)$, then
\begin{enumerate}[(i)]
\item For $0<\varepsilon$, the system (\ref{eq.pot}) for the voltage potential $u_\varepsilon$ is well-posed, i.e. it has a unique solution for any source $f \in H^{-1}(\Omega)$, and $u_\varepsilon$ depends continuously on $f$. 
\item The homogenized tensor $A^*$ defined by (\ref{eq.hommat}) is elliptic; in particular, the system (\ref{eq.pothom}) is well-posed.
\item For any source $f \in H^{-1}(\Omega)$, the unique solution $u_\varepsilon \in H^1_0(\Omega)$ to (\ref{eq.pot}) converges, weakly in $H^1_0(\Omega)$,
to the unique solution $u_*$ of (\ref{eq.pothom}).  
\end{enumerate}
\end{theorem}
\begin{proof}
\textit{(i)} As in Section \ref{sec.defTD}, (\ref{eq.pot}) is equivalent to: 
$$ \left( \frac{1}{1-a} I - T_\varepsilon \right) u_\varepsilon = \frac{1}{1-a} g,$$
where $g \in H^1_0(\Omega)$ represents $f \in H^{-1}(\Omega)$ in the sense of the Riesz theorem: 
$$ \forall v \in H^1_0(\Omega), \:\: \int_{\Omega}{\nabla g \cdot \nabla v \:dx} = \langle f , v \rangle_{H^{-1}(\Omega), H^1_0(\Omega)}.$$
Now, using Theorem \ref{th.evnp}, it is easily seen that there exists $\alpha >0$ small enough so that $\frac{1}{1-a}$ does not belong to the spectrum of $T_\varepsilon$ for $a \in (-\infty,-1/\alpha) \cup (-\alpha, 0)$; for such values of $a$, the operator $(\frac{1}{1-a} I - T_\varepsilon)^{-1}$
is bounded.\\

\noindent \textit{(ii)} follows immediately from Corollary \ref{cor.A*wp}, up to taking a smaller value $0<\alpha$.\\

\noindent \textit{(iii)} Owing to $(i)$, the sequence $u_\varepsilon$ is bounded in $H^1_0(\Omega)$. Using Proposition \ref{prop.gencvue}, 
there exists one solution $u_0 \in H^1_0(\Omega)$ to (\ref{eq.pothom}) such that, up to a subsequence $u_\varepsilon \to u_0$, 
weakly in $H^1_0(\Omega)$.
But we know from $(ii)$ that there is a unique solution $u_*$ to (\ref{eq.pothom}). Hence, $u_0=u_*$, and
by uniqueness of the limit, the whole sequence $u_\varepsilon$ converges to $u_0$.  
\end{proof}

The fact that (\ref{eq.pot}) is well-posed for values of the conductivity $a$ which are negative but `close to $0$' 
was already observed in the work of R. Bunuoiu and K. Ramdani \cite{ramdani}. 
It is quite noticeable that the same conclusion holds in the case where $a$ is negative, with a very large modulus, and that 
the homogenized tensor $A^*$ happens to be elliptic in both cases. These facts are new to the best of our knowledge.~\\

%%%%%%%%%%%%%%%%%%%%%%%%%%%%%%%%
\subsection{Application to a uniform convergence result for large conductivities in periodic homogenization}\label{sec.unifhomog}~\\
%%%%%%%%%%%%%%%%%%%%%%%%%%%%%%%%

In this section, we revisit a problem in `classical' periodic homogenization in light of the previous results. 
We assume that the conductivity $a$ inside the set of inclusions belongs to $(-\infty,-\alpha) \cup (\alpha,+\infty)$ for $\alpha>0$ large enough 
so that (\ref{eq.pot}) is well-posed for any source $f \in H^{-1}(\Omega)$, and the homogenized tensor $A^*$ given by (\ref{eq.hommat}) is elliptic; see Sections \ref{sec.hc} and \ref{seccoeffhomog}. 

We aim to understand if the convergence of the homogenization process $\varepsilon \rightarrow 0$ is \textit{uniform} in the values of $a \in (-\infty,-\alpha) \cup (\alpha,+\infty)$. 
To this end, we use the following notations: for $\varepsilon >0$ and a source $f \in H^{-1}(\Omega)$, $u_\varepsilon^a$ and $u_*^a\in H^1_0(\Omega)$ are the unique solutions to (\ref{eq.pot}) and (\ref{eq.pothom}) respectively when $a$ is the value of the conductivity inside the set of inclusions $\omega_\varepsilon$. 
Like in Section \ref{sec.defTD}, it is convenient to observe that (\ref{eq.pot}) may be equivalently written: 
\begin{equation}\label{eq.Teuag}
 \left( \frac{1}{1-a} I - T_\varepsilon \right) u_\varepsilon‰ = \frac{1}{1-a} g,
 \end{equation}
where $g \in H^1_0(\Omega)$ is the representative of $f$ supplied by the Riesz representation theorem. 

In our developments, a useful intermediate result is the solution $u_\varepsilon^\infty \in H^1_0(\Omega)$ to (\ref{eq.pot}) in the case where the inclusions are filled with a material with infinite conductivity, that is $a = +\infty$. Following for instance \cite{baolilin}, $u_\varepsilon^\infty$ is formally defined as the unique solution in $\text{\rm Ker}(T_\varepsilon)$ to the variational problem: 
\begin{equation}\label{eq.FVuinfty}
\forall v \in \text{\rm Ker}(T_\varepsilon), \:\: \int_{\Omega \setminus \overline{\omega_\varepsilon}}{\nabla u_\varepsilon \cdot \nabla v \:dx} = \int_\Omega{\nabla g \cdot \nabla v \:dx}.
\end{equation}
Let us recall (see Section \ref{sec.evTD}) that $\text{\rm Ker}(T_\varepsilon)$ is the subspace of $H^1_0(\Omega)$ 
composed of functions which are constant in each inclusion $\omega_\varepsilon^\xi$, $\xi \in \Xi_\varepsilon$. 

Our first result is an estimate of the difference between $u_\varepsilon^a$ and $u_\varepsilon^\infty$ which is uniform with respect to $\varepsilon$.

% On ne peut pas mettre n'importe quoi pour cette constante, ni se rapprocher trop de 0, au vu de l'estimation sur le dŽnominateur.

\begin{proposition}\label{propcvunifa}
There exist constants $\alpha >0$ and $C>0$, which are independent of $g$, $a$ and $\varepsilon$ such that: 
$$ \forall a \in (-\infty, -\alpha) \cup (\alpha,\infty), \:\: \lvert\lvert u_\varepsilon^a - u_\varepsilon^\infty \lvert\lvert_{H^1_0(\Omega)} \leq \frac{C}{\lvert a\lvert}\lvert\lvert g \lvert\lvert_{H^1_0(\Omega)}.$$ 
\end{proposition}
\begin{proof}
We observe that the variational formulation (\ref{eq.FVuinfty}) actually expresses that
$u_\varepsilon^\infty$ is the orthogonal projection of $g$ on $\text{\rm Ker}(T_\varepsilon)$ in the sense of the $H^1_0(\Omega)$ inner product (\ref{eq.H10prod}).

For $a \in (-\infty, -\alpha) \cup (\alpha,\infty)$, let us decompose $g$ according to the orthogonal direct sum (\ref{TDortho}) of $H^1_0(\Omega)$: 
$$ g = g_0 + g_1 + \sum\limits_{k\geq 1}{\langle g, h_k \rangle_{H^1_0(\Omega)} h_k },$$
where $g_0 \in \text{\rm Ker}(T_\varepsilon)$ equals $u_\varepsilon^\infty$ owing to the previous remark, where
$g_1 \in \text{\rm Ker}(I-T_\varepsilon)$, and where the $h_k\in \mathfrak{h}$ are the orthonormal
eigenvectors of the compact, self-adjoint operator $T_\varepsilon: \mathfrak{h} \rightarrow \mathfrak{h}$, whose sequence of eigenvalues
$\lambda_k\in (0,1)$ converges to $\frac{1}{2}$.
Likewise, we decompose the solution $u_\varepsilon^a$ to (\ref{eq.Teuag}): 
$$ u_\varepsilon^a = u_0 + u_1 + \sum\limits_{k\geq 1}{\beta_k h_k },$$
with $u_0 \in  \text{\rm Ker}(T_\varepsilon)$, $u_1 \in \text{\rm Ker}(I-T_\varepsilon)$ and coefficients $\beta_k\in \mathbb{R}$ (we omit the dependence of $u_0$, $u_1$ and the $\beta_k$ on $\varepsilon$ and $a$ for notational simplicity). 
Using (\ref{eq.Teuag}), we readily identify: 
$$u_0 = u_\varepsilon^\infty, \:\: u_1 = \frac{g_1}{a}, \text{ and } \beta_k = \frac{\langle g, h_k \rangle_{H^1_0(\Omega)}}{1 - (1-a)\lambda_k},$$
and so
$$ \lvert\lvert u_\varepsilon^a - u_\varepsilon^\infty \lvert\lvert_{H^1_0(\Omega)}^2 = \frac{1}{a^2}\vert\lvert g_1 \lvert\lvert_{H^1_0(\Omega)}^2 + \sum\limits_{k\geq 1}{\left(\frac{\langle g, h_k \rangle_{H^1_0(\Omega)}}{1 - (1-a)\lambda_k}\right)^2}.$$
Now using the bound supplied by Theorem \ref{th.evnp} on the eigenvalues $\lambda_k$, $k\geq 1$, one has:
$$\lvert 1 - (1-a)\lambda_k \lvert \geq (\lvert a \lvert - 1) \lambda_k
-1 \geq C\lvert a \lvert,$$
for some constant $C >0$, provided $\lvert a \lvert$ is large enough. Eventually, we obtain
$$ \lvert\lvert u_\varepsilon^a - u_\varepsilon^\infty \lvert\lvert_{H^1_0(\Omega)}^2  \leq \frac{1}{a^2}\vert\lvert g_1 \lvert\lvert_{H^1_0(\Omega)}^2 + \frac{C}{a^2} \lvert\lvert g \lvert\lvert_{H^1_0(\Omega)}^2,$$
where $C$ is a constant (different from that of the previous inequality) independent of $\varepsilon$, $g$ and $a$; this terminates the proof.
\end{proof}

We now prove a convergence estimate of the homogenization process 
which is \textit{uniform} in the value of the conductivity $a$ inside the inclusions, provided that this value is large enough (in modulus). 
This result holds in the $L^2(\Omega)$ norm, and cannot possibly hold in the $H^1(\Omega)$ norm, since correctors would have to be introduced for this purpose (see e.g. \cite{blp}). 
Recall also that error estimates are usually difficult to establish in homogenization since they generally call for the introduction and analysis of so-called 
\textit{boundary layer terms}, accounting for the interaction of the periodic structure with the boundary $\partial \Omega$ of the macroscopic domain. 

\begin{theorem}
There exists a constant $\alpha>0$ such that: 
$$ \sup\limits_{a \in (-\infty,-\alpha) \cup (\alpha,\infty)}{\lvert\lvert u_\varepsilon^a - u_*^a \lvert\lvert_{L^2(\Omega)}} \xrightarrow {\varepsilon \rightarrow 0} 0.$$
\end{theorem}
\begin{proof}
Assume that the result does not hold. 
Then there exists $\delta >0$ and two sequences $a_k \in (-\infty, -\alpha) \cup (\alpha,\infty)$  and $\varepsilon_k \rightarrow 0$ such that: 
\begin{equation}\label{eq.conthchom}
\lvert\lvert u_{\varepsilon_k}^{a_k} - u_*^{a_k} \lvert\lvert_{L^2(\Omega)} \geq \delta. 
\end{equation}
We may assume that $a_k$ does not change signs, and without loss of generality, we consider the case $a_k \in (\alpha,\infty)$.
By standard results in homogenization, it is clear that (\ref{eq.conthchom}) is only possible when $a_k \rightarrow \infty$.   

The triangle inequality yields:
\begin{equation*}\label{eq.triaidhc}
\lvert\lvert u_{\varepsilon_k}^{a_k} - u_*^{a_k} \lvert\lvert_{L^2(\Omega)} \leq \lvert\lvert u_{\varepsilon_k}^{a_k} - u_{\varepsilon_k}^{\infty} \lvert\lvert_{L^2(\Omega)} + \lvert\lvert u_{\varepsilon_k}^{\infty} - u_*^{\infty} \lvert\lvert_{L^2(\Omega)} + \lvert\lvert u_*^{\infty} - u_*^{a_k} \lvert\lvert_{L^2(\Omega)}.
\end{equation*}
In this inequality, 
\begin{itemize}
\item The first term in the right-hand side is controlled by $C / a_k$ as a consequence of Proposition \ref{propcvunifa}.
\item The second term estimates the difference between the voltage potential $u_\varepsilon^\infty$ and the homogenized solution $u_*^\infty$
in the particular case of the infinite conductivity problem. It can be proved, for instance in a similar way as in the proof of Proposition \ref{prop.gencvue} 
that this difference converges to $0$; see also \cite{ciorainfty} for a complete study of this problem. 
\item The last term in the right-hand side converges to $0$ owing to the conclusions of Section \ref{seccoeffhomog}, 
whereby the homogenized tensor $A^*$ (and thus the potential $u_*^a$) depends continuously on $a \in (-\infty,-\alpha) \cup (\alpha,\infty)$.
\end{itemize}
This shows that $\lvert\lvert u_{\varepsilon_k}^{a_k} - u_*^{a_k} \lvert\lvert_{L^2(\Omega)}$ converges to $0$ as $k \to \infty$, 
which contradicts our initial hypothesis.
\end{proof}

\begin{remark}
These results can be appraised in connection with the very deep results obtained by M. Briane and his collaborators concerning the homogenization of 
elliptic equations with unbounded coefficients; see for instance \cite{brianecasa}. In particular, in \cite{brbrcd}, the authors prove a similar result for (positive) unbounded $a$, in 
the two-dimensional setting, in a much larger context than ours, using a completely different proof.
\end{remark}\par
\medskip

\noindent \textbf{Acknowledgements} The authors were partially supported by the AGIR-HOMONIM grant from Universit\'e Grenoble-Alpes,
and by the Labex PERSYVAL-Lab (ANR-11-LABX-0025-01).

\appendix
%%%%%%%%%%%%%%%%%%%%%%%%%%%%%%%%%%%%%%%%%%%%%%%%%%%%%%%
\section{A closer study of the particular case of rank $1$ laminates}\label{sec.laminates}
%%%%%%%%%%%%%%%%%%%%%%%%%%%%%%%%%%%%%%%%%%%%%%%%%%%%%%%

In this appendix, we focus on an interesting particular geometry of microstructures $\omega \subset Y$, that of \textit{rank $1$ laminates}, 
which is one of the few amenable to explicit calculations.
Note that this situation violates some of the prevailing assumptions of this article, notably the fact that $\omega \Subset Y$;
it is therefore not surprising that some of the general results established in the previous sections do not hold in the present case.\\

The `macroscopic' domain $\Omega$ at stake is the two-dimensional square $(0,1)^2$. 
It is filled with $N^2$ identical cells, homothetic to the unit periodicity cell $Y = (0,1)^2$. 
The rescaled inclusion pattern $\omega$ in each cell is:
$$ \omega := \left\{ y=(y_1,y_2) \in Y, \: 0<y_1<\theta \right\}, \text{ and so } Y \setminus \overline{\omega} = \left\{y=(y_1,y_2) \in Y, \: \theta<y_1<1 \right\}, $$
where $\theta\in (0,1)$ is the volume of $\omega$. Accordingly, the total set of inclusions $\omega_N \subset \Omega$ reads: 
$$ \omega_N := \bigcup_{j \in \mathbb{N}^2 \atop 0 \leq j \leq N-1}{\frac{1}{N}(j + \omega)};$$
see Figure \ref{figcarre} for an illustration.

\begin{figure}[!ht]
\centering
\includegraphics[width=0.5 \textwidth]{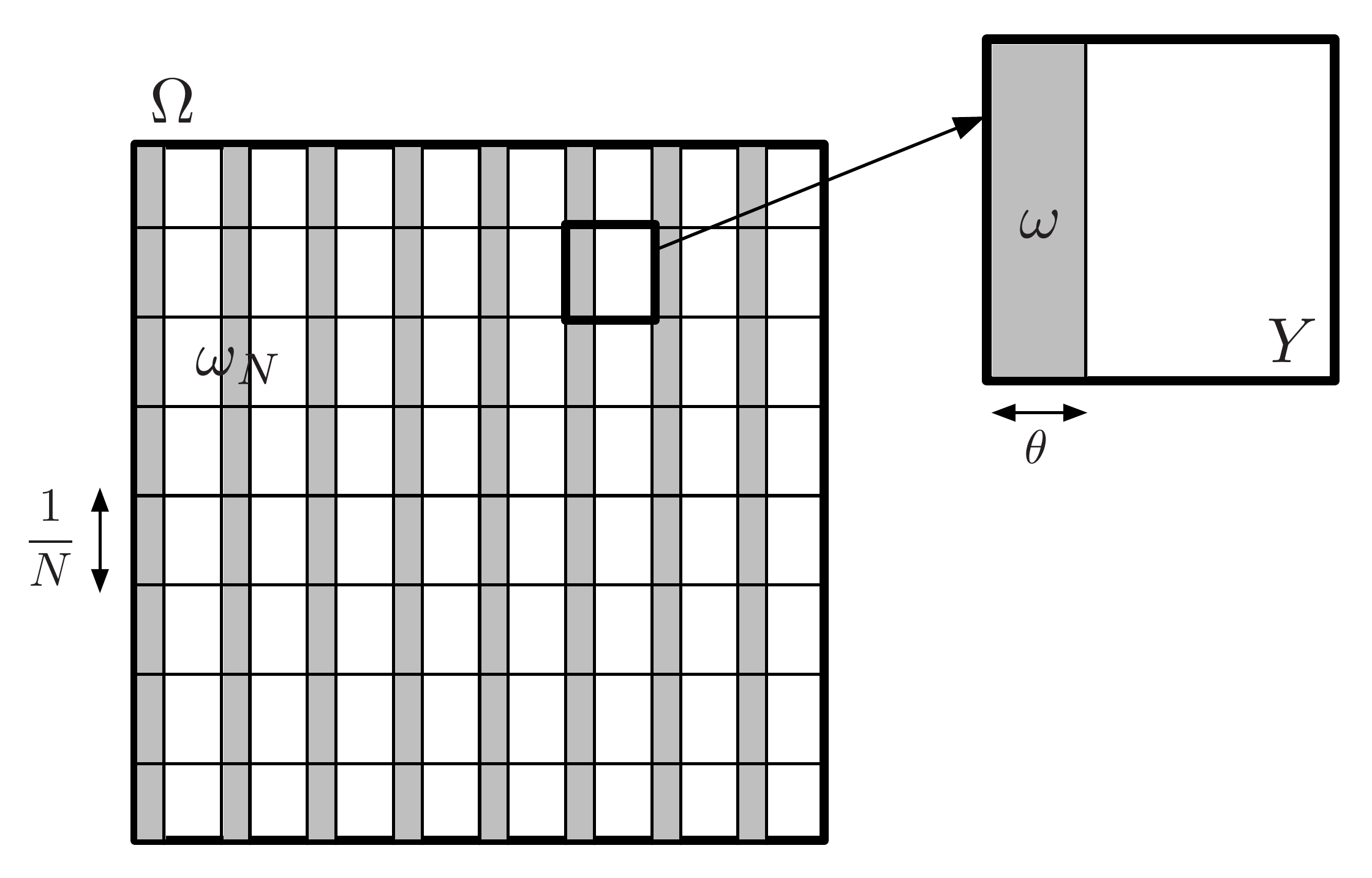}
\caption{\textit{Situation where the inclusion $\omega \subset Y$ is a rank $1$ laminate.}}
\label{figcarre}
\end{figure}

To avoid effects caused by the interactions between cells and the boundary $\partial \Omega$, 
we impose periodic boundary conditions on $\partial \Omega$. In this context, the voltage potential $u_N \in H^1_\#(\Omega)/\mathbb{R}$ associated to 
a source $f \in (H^1_\#(\Omega) / \mathbb{R})^*$ and to the distribution of conductivity equal to $a \in \mathbb{C}$ in $\omega_N$ and $1$ in $Y \setminus \overline{\omega_N}$
is solution to 
\begin{equation}\label{eq.potapp}
-\text{\rm div}(A_N \nabla u) = f \text{ in } \Omega, \text{ where } A_N(x) = \left\{ 
\begin{array}{cl}
a & \text{if }Êx \in \omega_N, \\
1 & \text{otherwise} 
\end{array}
\right. .
\end{equation}
We are interested in
 the spectrum of the Poincar\'e variational operator $T_N$, which maps an arbitrary function $u \in H_\#^1(\Omega) / \mathbb{R}$ to the unique element $T_N u \in H_\#^1(\Omega) / \mathbb{R}$ satisfying:  
$$ \forall v \in H^1_\#(\Omega)/ \mathbb{R}, \:\: \int_{\Omega}{\nabla (T_N u)\cdot \nabla v \:dx} = \int_{\omega_N}{\nabla u \cdot \nabla v\:dx},$$
and in particular in the identification of the limit spectrum
\begin{equation}\label{eq.limspecbloch}
 \lim\limits_{N \to \infty}{\sigma(T_N)} = \left\{\lambda \in \mathbb{C}, \:\: \exists \: N_j \to \infty, \: \lambda_{N_j} \in \sigma(T_{N_j}),\:\: \lambda_{N_j} \xrightarrow{N_j \to \infty} \lambda  \right\}.
 \end{equation}
\par
\smallskip 

\subsection{Study of the homogenization process for the operator $T_N$}~\\

The main tool in our analysis is the discrete Bloch decomposition \cite{aguirreconca}, which has already been used without proof several times
in this article. Although it is quite classical, we sktech the proof for completeness.

\begin{theorem}\label{thbloch}
Let $u$ be a function in $L^2_\#(\Omega)$. Then, there exists a unique collection $\left\{ u_j \right\}$, 
indexed by $j \in \mathbb{N}^2$, $0 \leq j \leq N-1$, composed of $N^2$ complex-valued functions in $L^2_\#(Y)$ such that the following identity holds: 
\begin{equation}\label{eq.decblochN}
  u(x) = \sum\limits_{0\leq j \leq N-1}{u_j(Nx)\: e^{2i\pi j \cdot x}}, \text{ a.e. } x \in \Omega.
  \end{equation}
Furthermore, the Parseval identity holds: for $u,v \in L^2_\#(\Omega)$, with coefficients $\left\{ u_j \right\}_{0\leq j \leq N-1}$, $\left\{ v_j \right\}_{0\leq j \leq N-1}$, one has: 
\begin{equation}\label{eq.ParsevalN}
 \int_\Omega{u(x)\overline{v(x)}\:dx}= \sum\limits_{0\leq j\leq N-1}{\int_Y{u_j(y)\overline{v_j(y)}\:dy}}.
 \end{equation}
\end{theorem}
\begin{proof}
Let $u \in L^2_\#(Y)$ be given, and assume that there exist $N^2$ functions $u_j \in L^2_\#(Y)$ such that (\ref{eq.decblochN}) holds.  
Then, for arbitrary $j^\prime \in \mathbb{N}^2$, $0 \leq j^\prime \leq N-1$, 
$$ u(x + \frac{j^\prime}{N}) = \sum\limits_{0 \leq j \leq N-1}{u_j(Nx)e^{2i\pi j \cdot (x + \frac{j^\prime}{N})}};$$
so as to isolate a particular index $0 \leq j^0 \leq N-1$, we multiply both sides in the previous identity by $e^{-2i\pi j^0 \cdot (x + \frac{j^\prime}{N})}$, then sum over $j^\prime$ to obtain:
\begin{equation}\label{eq.dblesumbloch}
 \sum\limits_{0 \leq j \leq N-1 \atop 0 \leq j^\prime \leq N-1}{u_j(Nx) e^{2i\pi(j- j^0) \cdot (x + \frac{j^\prime}{N})}} = \sum\limits_{0 \leq j \leq N-1}{u(x + \frac{j^\prime}{N}) e^{-2i\pi j^0 \cdot (x + \frac{j^\prime}{N})}}.
 \end{equation}
Using the fact that, for any $N^{\text{th}}$ root 
$r = e^{2\pi i j/N}$ 
of $1$, 
\begin{equation}\label{Nthroot}
 \sum\limits_{0\leq k \leq N-1}{r^k } = \left\{ 
\begin{array}{cc}
N & \text{if }Êr = 1, \\
0 & \text{otherwise}
\end{array}
\right. ,
\end{equation}
The relation
(\ref{eq.dblesumbloch}) simplifies into: 
\begin{equation}\label{eq.invbloch}
 u_{j^0}(y)  =\frac{1}{N^2}\sum\limits_{0\leq j^\prime \leq N-1}{u(\frac{y}{N} + \frac{j^\prime}{N})e^{-2i\pi j^0 \cdot \left(\frac{y}{N} + \frac{j^\prime}{N}\right)} },
 \end{equation}
a formula which clearly defines a function in $L^2_\#(Y)$. Conversely, one easily proves that the $u_j$ defined in (\ref{eq.invbloch}) satisfy (\ref{eq.decblochN}), which proves the first statement.

To verify the Parseval identity (\ref{eq.ParsevalN}), let $u,v \in L^2_\#(\Omega)$ be decomposed as:
$$ u(x) = \sum\limits_{0\leq j \leq N-1}{u_j(Nx)\: e^{2i\pi j \cdot x}}, \text{ and }Êv(x) = \sum\limits_{0\leq j^\prime \leq N-1}{
v_{j^\prime}
(Nx)\: e^{2i\pi j^\prime \cdot x}}.$$
A simple calculation yields: 
$$ 
\begin{array}{>{\displaystyle}c c>{\displaystyle}l}
\int_\Omega{u(x)\overline{v(x)}\:dx} &=& \sum\limits_{0 \leq j, j^\prime \leq N-1}{\int_{\Omega}{u_j(Nx)\overline{v_{j^\prime}(Nx)} \: e^{2i\pi(j-j^\prime)\cdot x}\:dx}}\\
&=& \sum\limits_{0 \leq j, j^\prime,k \leq N-1}{\int_{\frac{1}{N}(k+Y)}{u_j(Nx)\overline{v_{j^\prime}(Nx)} \: e^{2i\pi(j-j^\prime)\cdot x}\:dx}}\\
&=& \frac{1}{N^2}\sum\limits_{0 \leq j, j^\prime,k \leq N-1}{\int_{Y}{u_j(y)\overline{v_{j^\prime}(y)} \: e^{2i\pi(j-j^\prime) \cdot(\frac{k}{N} + \frac{y}{N})}\:dy}}\\
&=& \sum\limits_{0 \leq j \leq N-1}{\int_{Y}{u_j(y)\overline{v_{j}(y)}\:dy}},
\end{array}
$$ 
where we have again made use of (\ref{Nthroot}) to pass from the third line to the last. 
\end{proof}
Let us now consider the operator $B$, from $L^2_\#(Y)^{N^2}$ into $L^2_\#(\Omega)$ which maps a collection $\left\{ u_j \right\}_{0\leq j \leq N-1}$ of coefficients to the function $u \in L^2_\#(\Omega)$ defined by:
$$ u(x) = \sum\limits_{0\leq j \leq N-1}{u_j(Nx)\: e^{2i\pi j \cdot x}}, \text{ a.e. } x \in \Omega.$$ 
Equipping both spaces with their natural inner products, Theorem \ref{thbloch} states that $B$ is a bijective isometry, whose inverse: 
$ {B}^{-1}:L^2_\#(\Omega) \to L^2_\#(Y)^{N^2}$ is also its adjoint operator.

If $u$ belongs to $H^1_\#(\Omega)$,
(\ref{eq.invbloch}) implies that the coefficients $u_j$ of its Bloch decomposition actually belong to $H^1_\#(Y)$, and that the Bloch decomposition of $\nabla u$ reads:
$$ \nabla u(x) = N \sum\limits_{0\leq j \leq N-1}{(\nabla_y + 2i\pi \frac{j}{N}) u_j(Nx)\: e^{2i\pi j \cdot x}}, \text{ a.e. } x \in \Omega.$$
Using the fact that the Bloch decomposition of the constant function $u \equiv 1 \in H^1_\#(\Omega)$ has coefficients: 
$$u_0(y) = 1, \text{ and } u_j(y) = 0 \text{ if } j \neq 0,$$
$B$ induces an invertible operator (still denoted by $B$) $(H^1_\#(Y)/\mathbb{C}) \times H^1_\#(Y)^{N^2-1} \rightarrow H^1_\#(\Omega) /\mathbb{C}$. 

The following proposition easily follows from the previous remarks, and in particular from the Parseval identity (\ref{eq.ParsevalN}).

\begin{proposition}
For any $\eta \in \overline{Y}$, $\eta \neq 0$, define $T_{\eta}: H^1_\#(Y) \to H^1_\#(Y)$ by   
\begin{equation}\label{eq.Tetaapp}
 \forall v \in H^1_\#(Y), \:\: \int_{Y}{(\nabla_y + 2i\pi \eta)(T_{\eta} u) \cdot \overline{(\nabla_y + 2i\pi \eta) v }\:dy} = \int_{\omega}{(\nabla_y + 2i\pi \eta) u \cdot \overline{(\nabla_y + 2i\pi \eta) v }\:dy },
 \end{equation}
and define $T_0 : H^1_\#(Y) / \mathbb{C} \to H^1_\#(Y) / \mathbb{C}$ by 
\begin{equation}\label{eq.T0app}
 \forall v \in H^1(Y) / \mathbb{C}, \:\: \int_Y{\nabla_y (T_0 u) \cdot \overline{\nabla_y v}\:dy} = \int_\omega{\nabla_y u \cdot \overline{\nabla_y v}\:dy}.
 \end{equation}
The operator $B$ diagonalizes $T_N$, i.e. the self-adjoint operator $B^* T_N B$ maps a collection $\left\{ u_j \right\}_{0 \leq j \leq N-1} \in (H^1_\#(Y)/\mathbb{C}) \times H^1_\#(Y)^{N^2-1}$ to $\left\{ T_{\frac{j}{N}} u_j \right\}_{0 \leq j \leq N-1} \in (H^1_\#(Y)/\mathbb{C}) \times H^1_\#(Y)^{N^2-1}$. As
a consequence, the spectrum $\sigma(T_N)$ is the union of the spectra $\sigma(T_{\frac{j}{N}})$:
$$ \sigma(T_N) = \bigcup_{0 \leq j \leq N-1}{\sigma(T_{\frac{j}{N}})}.$$ 
\end{proposition}

Therefore, the study of the spectrum of $T_N$ boils down to that of the spectra of the operators $T_\eta$ defined by (\ref{eq.Tetaapp}) and (\ref{eq.T0app}).
Let us now take advantage of the particular geometry of $\omega$ to simplify the problem further:
We decompose functions $u \in H^1_\#(Y)$ (or $u \in H^1_\#(Y) / \mathbb{C}$) by using partial Fourier series in the variable $y_2$: 
$$ u(y) =  \sum\limits_{n=-\infty}^{+\infty}{a_n(y_1) e^{2i\pi ny_2}}, \text{ a.e. } y \in Y,$$
where the $a_n \in H^1_\#(0,1)$ (and $a_0 \in H^1_\#(0,1) / \mathbb{C}$ if $u \in H^1_\#(Y) / \mathbb{C}$). 

After some elementary calculations, the operators $T_\eta$ are diagonalized by this Fourier decomposition, i.e. the spectrum of $T_N$ reads: 
$$ \sigma(T_N) = \bigcup_{0 \leq j \leq N-1 \atop n \in \mathbb{N}}{\sigma(T_{\frac{j}{N}}^n)},$$
where, for any $\eta = (\eta_1,\eta_2) \in \overline{Y}$, $T_\eta^n : H^1_\#(0,1) \to H^1_\#(0,1)$ is defined by: 
\begin{multline}\label{eq.Tetan}
\forall v \in H^1_\#(0,1) \in, \:\: \int_{0}^1{\left(((T_\eta^n u)^\prime + 2i\pi \eta_1 (T_\eta^n u))\overline{(v^\prime + 2i\pi \eta_1 v)} + 4\pi^2(n+\eta_2)^2 u\overline{v}\right)\:dy}  
\\ = \int_{0}^\theta{\left((u^\prime + 2i\pi \eta_1 u)\overline{(v^\prime + 2i\pi \eta_1 v)} + 4\pi^2(n+\eta_2)^2 u\overline{v}\right)\:dy},
\end{multline}
and $T_0^0 : H^1_\#(0,1)/ \mathbb{C} \to H^1_\#(0,1) / \mathbb{C}$ is given by: 
\begin{equation}\label{eq.T00}
\forall v \in H^1_\#(0,1)/\mathbb{C}, \:\: \int_{0}^1{((T_0^0 u)^\prime)\overline{v^\prime}\:dy}  
\\ = \int_{0}^\theta{u^\prime\overline{v^\prime}\:dy}.
\end{equation}

It is proved in the same way as in Section \ref{sec.NPTD} that the spectrum of each operator $T_\eta^n$
consists of a sequence of eigenvalues in $[0,1]$ with $\frac{1}{2}$ as unique accumulation point, and we now proceed to identify these eigenvalues.\\

Let us first study the eigenvalues of the operator $T_\eta^n$ in the case where either $\eta_2 \neq 0$ or $n \neq 0$.
A value $\beta \in \mathbb{C}$ is an eigenvalue for $T_\eta^n$ as defined in (\ref{eq.Tetan}) if there exists $u \in H^1_\#(0,1)$, 
$u\neq 0$ such that:
 \begin{equation}\label{eq.stdyevTetan}
 -\left( \frac{\partial }{\partial y_1} + 2i\pi \eta_1 \right)\left(A_\beta(y_1)\left( \frac{\partial u}{\partial y_1} + 2i\pi \eta_1 u \right) \right) + 4\pi^2 A_\beta(y_1)(n + \eta_2)^2 u =0,
 \end{equation}
 where:
$$A_\beta(y) = \left\{Ê
\begin{array}{cc}
\beta -1 & \text{if }Êy_1 < \theta, \\
\beta  & \text{if }Ê\theta < y_1 < 1.  \\
\end{array}
\right. $$
Assuming $\beta \notin \left\{0 ,1\right\}$, (\ref{eq.stdyevTetan}) is equivalent to:
\begin{equation}\label{edobloch}
 u^{\prime\prime}(z) + 4i \pi\eta_1 u^\prime(z) - 4\pi^2(\eta_1^2 + (n+\eta_2)^2) u(z) = 0 \text{ a.e. }  z\in (0,\theta) \text{ and } z \in (\theta,1), 
 \end{equation}
complemented with the transmission conditions at $z=0$ and $z = \theta$:
\begin{equation}\label{edoblochtc0}
 u(0^+) = u(1^-), \:\:\:\:  \beta (u^\prime + 2i\pi \eta_1u)(1^-)  = (\beta-1) (u^\prime + 2i\pi \eta_1u)(0^+),
 \end{equation}
\begin{equation}\label{edoblochtctheta}
 u(\theta^-) = u(\theta^+), \:\:\:\: (\beta-1) (u^\prime + 2i\pi \eta_1u)(\theta^-) = \beta (u^\prime + 2i\pi \eta_1u)(\theta^+) .
 \end{equation}

The ordinary differential equation (\ref{edobloch}) has discriminant $\Delta = 16\pi^2(n + \eta_2)^2$, 
and the associated characteristic equation has two solutions 
$$r_1 = -2i \pi\eta_1 - 2\pi(n + \eta_2), \:\: r_2 = -2i\pi\eta_1 + 2\pi(n + \eta_2),$$ 
which are distinct since $n+\eta_2 \neq 0$. 
Therefore, there exist $4$ coefficients $A,B,C,D \in \mathbb{C}$ such that: 
$$ u(z) = Ae^{r_1z} + Be^{r_2z} \text{ for }Êz \in (0,\theta), \:\: u(z) = Ce^{r_1z} + De^{r_2z} \text{ for }Êz \in (\theta,1). $$
The fact that $u \in H^1_\#(0,1)$ imposes that: 
$$ A + B = Ce^{r_1} + De^{r_2}, \text{ and } Ae^{r_1\theta} + Be^{r_2\theta} = Ce^{r_1\theta} + De^{r_2\theta},$$
to be complemented with the transmission conditions: 
$$(\beta-1) (r_1+2i\pi\eta_1) A+ (\beta-1) (r_2+2i\pi\eta_1) B = \beta(r_1+2i\pi\eta_1)Ce^{r_1} + \beta (r_2+2i\pi\eta_1)De^{r_2} $$
$$(\beta-1) (r_1+2i\pi\eta_1) Ae^{r_1\theta} + (\beta-1) (r_2+2i\pi\eta_1) Be^{r_2\theta} = \beta(r_1+2i\pi\eta_1)Ce^{r_1\theta} + \beta (r_2+2i\pi\eta_1)De^{r_2\theta} $$

As a consequence, (\ref{edobloch},\ref{edoblochtc0},\ref{edoblochtctheta}) has a non trivial solution provided the following 
determinant vanishes:
\begin{equation}\label{eq.detbloch}
\left\lvert 
\begin{array}{cccc}
1 & 1 & -e^{r_1} & -e^{r_2} \\
e^{r_1\theta} & e^{r_2\theta} & -e^{r_1\theta} & -e^{r_2\theta} \\
(\beta-1) (r_1+2i\pi\eta_1) &  (\beta-1) (r_2+2i\pi\eta_1) & -\beta(r_1+2i\pi\eta_1) e^{r_1} & -\beta (r_2+2i\pi\eta_1)e^{r_2}\\
(\beta-1) (r_1+2i\pi\eta_1) e^{r_1\theta} &  (\beta-1) (r_2+2i\pi\eta_1) e^{r_2\theta} & -\beta(r_1+2i\pi\eta_1) e^{r_1\theta} & -\beta (r_2+2i\pi\eta_1)e^{r_2\theta}
\end{array}
\right\lvert = 0,
 \end{equation}
which is a quadratic equation for $\beta$. After tedious calculations, (\ref{eq.detbloch}) simplifies into: 
$$ \beta^2 - \beta + \gamma = 0, \text{ where } \gamma := \frac{1}{4} \frac{\cosh(2\pi (n + \eta_2)) - \cosh((2\pi (n + \eta_2))(2\theta-1))}{\cosh(2\pi (n + \eta_2)) - \cos(2\pi\eta_1)}.$$
The discriminant of this second order equation reads: 
\begin{equation}\label{eq.discbloch}
 \Delta_\eta^n = \frac{\cosh((2\pi (n + \eta_2))(2\theta-1)) -\cos(2\pi\eta_1)}{\cosh(2\pi (n + \eta_2)) - \cos(2\pi\eta_1)}.
 \end{equation}
We observe that $\Delta_\eta^n \in (0,1)$, leading to two distinct eigenvalues $\beta_\eta^{n\pm} = (1 \pm \sqrt{\Delta_\eta^n})/2$, which are symmetric with respect to $\frac{1}{2}$.\\

As for the eigenvalues of $T_\eta^n$ in the case where $\eta_2 = 0$ and $n=0$, simple calculations show that $\sigma(T_\eta^n) = \left\{ 0, 1\right\}$ if $\eta_1 \neq 0$
and that $\sigma(T_0^0) = \left\{0,1-\theta,1\right\}$. All in all, we have proved that the spectrum of $T_N$ is: 
$$ \sigma(T_N) = \left\{0,1-\theta,1\right\} \cup \left\{ \frac{1}{2}(1\pm \sqrt{\Delta_{\frac{j}{N}}^n})\right\}_{0 < j \leq N-1 \atop n \in \mathbb{N}},$$ 
where $\Delta_\eta^n$ is defined by (\ref{eq.discbloch}).
This allows for the identification of the limit spectrum (\ref{eq.limspecbloch}) as:
$$ \lim\limits_{N\to \infty}{\sigma(T_N)} = \left\{0,1-\theta,1\right\} \cup \left\{ \frac{1}{2}(1\pm \sqrt{\Delta_{\eta}^n})\right\}_{\eta \in \overline{Y}, n \in \mathbb{N}} = [0,1].$$
Hence, in the present situation of rank $1$ laminates, the limit spectrum of $T_N$ is the whole interval $[0,1]$, in sharp contrast with the situation where $\omega \Subset Y$, 
as studied in Section \ref{sec.limspec}.\\

\subsection{Analysis of the homogenized tensor $A^*$}~\\

As we have seen in Section \ref{sec.gendirect} (see in particular Propositions \ref{prop.gencvue} and \ref{prop.limitrecip}, which are straightforwardly adapted to this context), 
the asymptotic behavior of the voltage potential $u_N$ associated to a conductivity $a \in \mathbb{C}$ inside the set of inclusions $\omega_N$, solution to (\ref{eq.potapp}), is partly driven by the formally homogenized tensor $A^*$ whose components $A^*_{ij}$ are given by ($i,j=1,2$): 
\begin{equation}\label{eq.A*app}
 A^*_{ij} = \int_Y{A(y) (e_i + \nabla_y \chi_i) \cdot (e_j + \nabla_y \chi_j) \:dy}, \text{ where } A(y) = \left\{ 
\begin{array}{cl}
a & \text{if } y \in \omega, \\
1 & \text{otherwise} ,
\end{array}
\right. 
\end{equation}
and the cell functions $\chi_i \in H^1_\#(Y) / \mathbb{R}$ ($i=1,2$) are solutions to: 
\begin{equation}\label{eq.cellpbapp}
 -\text{\rm div}_y(A(y)(e_i + \nabla_y \chi_i)) = 0.
\end{equation}
As we have seen in Section \ref{sec.cellpb}, both cell problems (\ref{eq.cellpbapp}) have a unique solution, 
provided the conductivity $a$ lies outside the exceptional set $\Sigma_\omega$ given by: 
$$ \Sigma_\omega = \left\{ a \in \mathbb{C}, \:\: \frac{1}{1-a} \in \sigma(T_0)\right\},$$
where $T_0 : H^1_\#(Y) / \mathbb{R} \to H^1_\#(Y) / \mathbb{R}$ is (the real version of that ) defined by (\ref{eq.T0app}):
\begin{equation}\label{eq.T0appcell}
 \forall v \in H^1_\#(Y) / \mathbb{R}, \:\: \int_Y{\nabla_y (T_0 u) \cdot \nabla_y v\:dy} = \int_\omega{\nabla_y u \cdot \nabla_y v\:dy}.
 \end{equation}
On the contrary, in the case $a \in \Sigma_\omega$, there may be multiple solutions to (\ref{eq.cellpbapp}) (or none), 
but when this happens, it is easily seen that (\ref{eq.A*app}) is independent of which of these solutions is used.\\

Calculations similar to those performed in the previous section (based on a partial Fourier decomposition) lead to an explicit characterization of the set $\Sigma_\omega$ in the present context of rank $1$ laminates: 
$$ \Sigma_\omega := \left\{Ê-\frac{\theta}{1-\theta}, 0\right\} \cup \left\{Êa_n^\pm\right\}_{n \in \mathbb{N}^*}, $$
where the $a_n^\pm$ read: 
$$ a^n_{\pm} = \frac{2(1 \pm 2\sinh(\pi n)\sinh(\pi n (2\theta-1)) ) - \cosh(2\pi n (2\theta-1)) - \cosh(2\pi n)}{\cosh(2\pi n) - \cosh(2\pi n (2\theta-1))}.$$

Let us now turn to the cell problems (\ref{eq.cellpbapp}) in this setting for an arbitrary conductivity $a \in \mathbb{C}$. 
Relying again on a partial Fourier decomposition in the $y_2$ variable, it is easy to see that, 
if (\ref{eq.cellpbapp}) has solutions, one of them is necessarily of the form: 
\begin{equation}\label{depdy1}
 \chi_i(y) = u_i(y_1), \text{ for some function } u_i \in H^1_\#(0,1) / \mathbb{R}, \: i = 1,2.
 \end{equation}
 
Now, easy calculations reveal that: 
\begin{itemize}
\item If $a \neq -\frac{\theta}{1-\theta}$, (\ref{eq.cellpbapp}) has (possibly non unique) solutions $\chi_i$ for $i=1,2$: 
$$ \chi_1(y) = \left\{ 
\begin{array}{cl}
Ay_1 + B & \text{if } y_1 < \theta, \\
Cy_1 + D & \text{if } y_1 > \theta, \\
\end{array}
\right. ,\text{ and } \chi_2(y) = 0,$$
where the coefficients $A$ and $C$ read: 
$$A = \frac{1-a}{a - \frac{\theta}{1-\theta}}, \: C = A \frac{\theta} {\theta-1}. $$
Then, the homogenized tensor (\ref{eq.A*app}) equals: 
$$ A^* = \left( \begin{array}{cc}
\lambda_{\theta,a}^- & 0, \\
0 & \lambda_{\theta,a}^+
\end{array}
\right), \text{ where }Ê\lambda_{\theta,a}^- = \left(\frac{\theta}{a} + 1-\theta\right)^{-1}, \text{ and }\lambda_{\theta,a}^+ = a\theta + (1-\theta) .$$
Note that $A^*$ is invertible, and becomes degenerate when $a$ gets close to $-\frac{\theta}{1-\theta}$.
The behaviors of the mappings $a \mapsto \lambda_{\theta,a}^\pm$ change depending on whether the volume fraction $\theta$ is larger or smaller than $\frac{1}{2}$, 
as can be seen on Figure \ref{figvpa}.
In the case where $\theta < \frac{1}{2}$, three regimes are to be distinguished: 
\begin{itemize}
\item When $-\frac{\theta}{1-\theta} < a < 0 $, $A^*$ has eigenvalues with opposite signs. 
\item When $-\frac{1-\theta}{\theta} <a<-\frac{\theta}{1-\theta}$, $A^*$ is positive definite . 
\item When $a < -\frac{1-\theta}{\theta}$, $A^*$ has again eigenvalues with opposite signs. 
\end{itemize}
This behavior differs in the case where $\theta > \frac{1}{2}$: 
\begin{itemize}
\item When $-\frac{1-\theta}{\theta}< a < 0 $, $A^*$ has eigenvalues with opposite signs. 
\item When $-\frac{\theta}{1-\theta} <a<-\frac{1-\theta}{\theta}$, $A^*$ is negative definite . 
\item When $a < -\frac{\theta}{1-\theta}$, $A^*$ has again eigenvalues with opposite signs. 
\end{itemize}

These results are in sharp contrast with the case of inclusions $\omega \Subset Y$, dealt with in Section \ref{secdirect}. 

\item In the case $a = -\frac{\theta}{1-\theta}$, the cell problem (\ref{eq.cellpbapp}) has infinitely many solutions of the form (\ref{depdy1}) for $i=2$, 
and none for $i=1$. 
\end{itemize}

It is remarkable that the only value of $a$ for which the cell problems do not have solutions is $a = -\frac{\theta}{1-\theta}$, corresponding 
to the essential spectrum of the operator defined in (\ref{eq.T0appcell}). We do not know whether this fact holds for general inclusion patterns
$\omega \subset Y$ 
or is particular to rank $1$ laminates.

\begin{figure}[!ht]
\centering
\includegraphics[width=0.49 \textwidth]{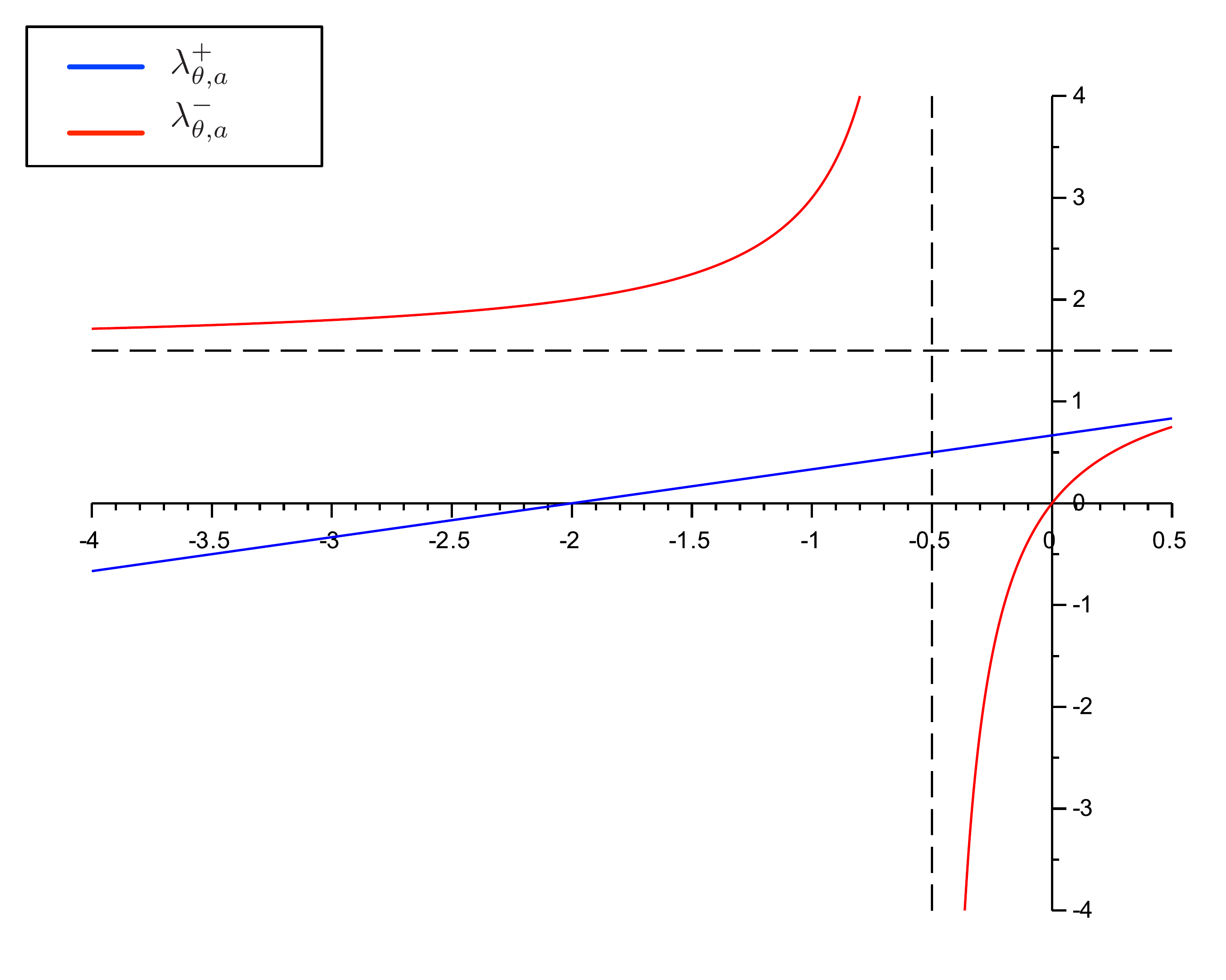} \hspace{0.1cm} \includegraphics[width=0.49 \textwidth]{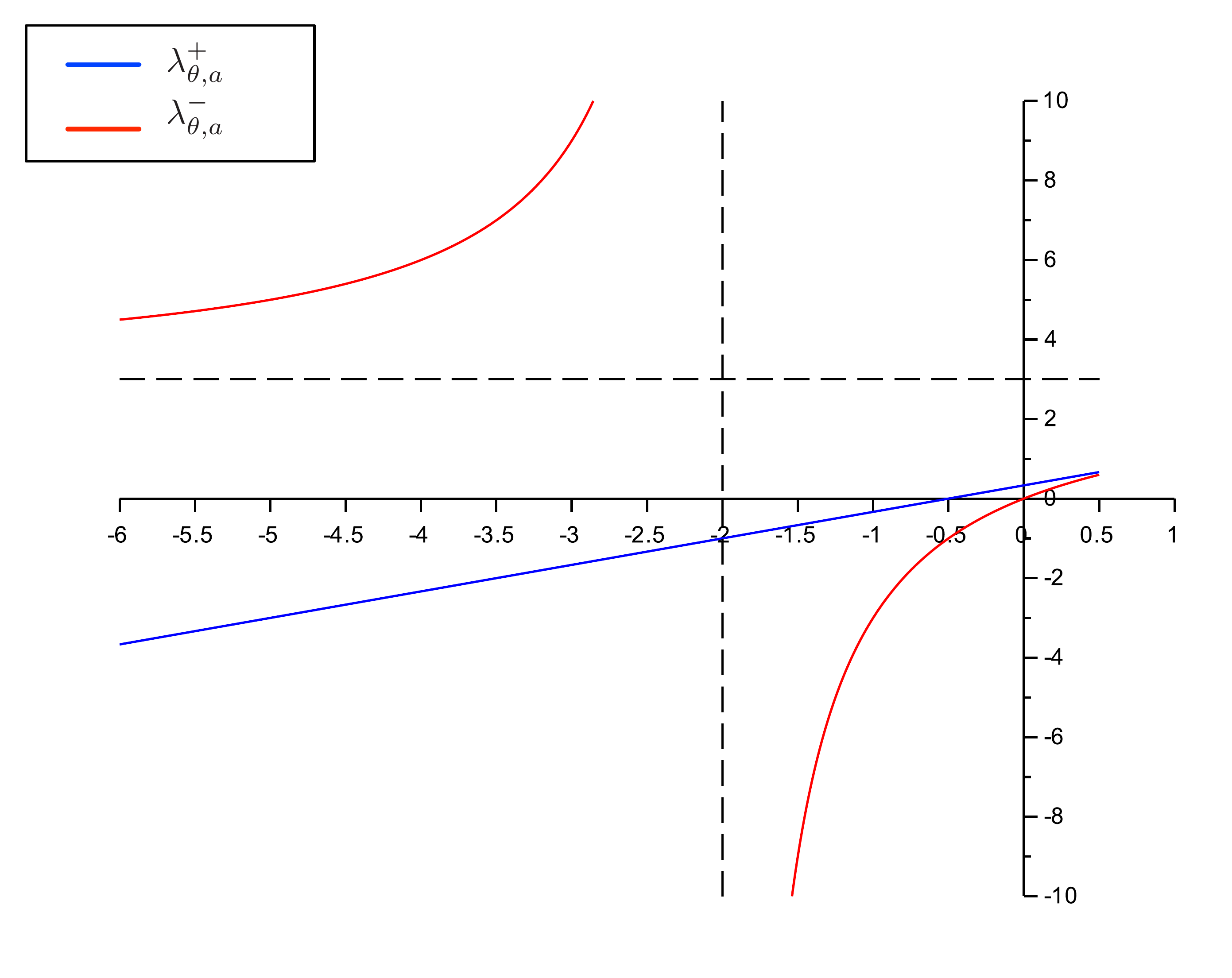}
\caption{\textit{Behavior of the eigenvalues $a \mapsto \lambda_{\theta,a}^\pm$ of the homogenized matrix $A^*$ for the volume fractions (left) $\theta = \frac{1}{3}$, 
and (right) $\theta = \frac{2}{3}$.}}
\label{figvpa}
\end{figure}

%%%%%%%%%%%%%%%%%%%%%%%%%%%%%%%%%%%%%%%%%%%%%%%%%%%%%%%


\begin{thebibliography}{00}
%%%%%%%%%%%%%%%%%%%%%%%%%%%%%%%%%%%%%%%%%%%%%%%%%%%%%%%
\bibitem{adams}{\sc R. A. Adams and J.F. Fournier},
{\em Sobolev spaces},
Academic Press, 2nd Ed., (2003).

\bibitem{aguirreconca} {\sc F. Aguirre and C. Conca},
{\em Eigenfrequencies of a tube bundle immersed in a fluid}, Appl. Math. Optim., 18, (1988), pp.1--38.

\bibitem{allaire2s}{\sc G. Allaire}, {\em Homogenization and two-scale convergence}, SIAM J. Math. Anal. 23, 6, (1992), pp.~1482--1518. 

\bibitem{allairehomog}{\sc G. Allaire}, {\em Shape optimization by the homogenization method}, Springer Verlag, New York, (2001). 

\bibitem{allairebrianevanni} {\sc G. Allaire, M. Briane and M. Vanninathan},
{\em A comparison between two-scale asymptotic expansions and Bloch wave expansions for the homogenization of periodic structures}, SEMA journal, 73(3),(2016), pp.~237--259.

\bibitem{allaireconcafs} {\sc G. Allaire and C. Conca},
{\em Bloch wave homogenization for a spectral problem in fluid-solid structures}, Arch. Rat. Mech. Anal., 135, (1996), pp.197--257.

\bibitem{allaireconca} {\sc G. Allaire and C. Conca},
{\em Bloch wave homogenization and spectral asymptotic analysis}, J. Math. Pures et Appli., 77, (1998), pp.153--208.

\bibitem{AmmariCiraolo} {\sc H. Ammari, G. Ciraolo, H. Kang, H. Lee and G.W. Milton}, {\em Spectral theory of a Neumann-Poincar\'e-type operator and analysis of cloaking due to anomalous localized resonance}, Arch. Ration. Mech. An. 208 (2013), pp.~667--692.

\bibitem{AmmariCiraoloII} {\sc H. Ammari, G. Ciraolo, H. Kang, H. Lee and G.W. Milton}, {\em Spectral analysis of a Neumann-PoincarŽ-type operator and analysis of cloaking due to anomalous localized resonance II}, Contemporary Mathematics, 615, (2014), pp.~1--14.

\bibitem{AmmariDengMillien} {\sc H. Ammari, Y. Deng and P. Millien}, {\em Surface Plasmon Resonance of Nanoparticles and
Applications in Imaging}, Arch. Ration. Mech. An. 220, (2016), pp.~109--153.

\bibitem{AmmariKang} {\sc H. Ammari and H. Kang},
{\em Polarization and Moment Tensors; With Applications to Inverse Problems and Effective Medium Theory},
Springer Applied Mathematical Sciences, 162, (2007).

\bibitem{AmmariKangLee} {\sc H. Ammari, H. Kang, and H. Lee}, {\em Layer Potential Techniques in Spectral Analysis}, Mathematical Surveys and Monographs, Vol. 153, American Mathematical Society, Providence RI, (2009).

\bibitem{AmmariKangLim}{\sc H. Ammari, H. Kang, and M. Lim}, {\em Gradient estimates for solutions to the conductivity problem}, Math. Ann., 332, (2005), pp.~277--286.

\bibitem{AmmariMillienRuizZhang} {\sc H. Ammari, P. Millien, M. Ruiz and H. Zhang}, {\em Mathematical analysis of plasmonic nanoparticles: the scalar
case},  arXiv:1506.00866, (2016).

\bibitem{AmmariRuizYuZhang} {\sc H. Ammari, M. Ruiz, S. Yu, and H. Zhang}, {\em Mathematical analysis of plasmonic resonances for nanoparticles: the full Maxwell equations}, arXiv:1511.06817, (2016).

\bibitem{AmmariPer} {\sc H. Ammari, H. Kang and K. Touibi},
{\em Boundary Layer Techniques for Deriving the Effective Properties of Composite Materials},
Asymptot. Anal. 41(2), (2005), pp.~119--140.

\bibitem{AmmariSeo} {\sc H. Ammari and J. K Seo}, {\em An accurate formula for the reconstruction of conductivity inhomogeneities}, Adv. Appl. Math., 30(4), (2003), pp.~679--705.

\bibitem{ando} {\sc K. Ando and H. Kang}, {\em Analysis of plasmon resonance on smooth domains using spectral properties of the Neumann-PoincarŽ operator}, Journal of Mathematical Analysis and Applications, 435, (2016), pp.~162--178.

\bibitem{baolilin}{\sc E. S. Bao, Y. Li and B. Yin} {\em Gradient Estimates for the Perfect
Conductivity Problem}, Arch. Rational Mech. Anal., 193, (2009), pp.~195--226.

\bibitem{blp} {\sc A. Bensoussan, J.-L. Lions and G. Papanicolau},
{\em Asymptotic analysis of periodic structures},
North Holland, (1978).

\bibitem{Tcoerc} {\sc A.-S. Bonnet-Ben Dhia, P. Ciarlet Jr. and C.-M. Zw\"olf}, {\em Time harmonic wave diffraction problems in materials with sign-shifting coefficients}, J. Comput. Appl. Math. 234, (2007), pp.~1912Ð1919, Corrigendum 2616 (2010). 

\bibitem{bonhm} {\sc E. Bonnetier and H.-M. Nguyen}, {\em Superlensing using hyperbolic metamaterials: the scalar case}, submitted (2016). 

\bibitem{bontrikigrad} {\sc E. Bonnetier and F. Triki},
{\em Pointwise bounds on the gradient and the spectrum of the NeumannÐPoincarŽ operator: the case of 2 discs},
Contemp. Math., 577, (2012), pp.~81--92.

\bibitem{bontriki} {\sc E. Bonnetier and F. Triki},
{\em On the spectrum of the Poincar\'e variational problem for two close-to-touching inclusions in 2d},
Arch. Rational Mech. Anal., 209, (2013), pp.~541--567.

\bibitem{bontrikitsou} {\sc E. Bonnetier, F. Triki and C.H. Tsou},
{\em Eigenvalues of the Neumann-PoincarŽ operator for two inclusions with contact of order m: A numerical study},
 to appear in Journal of Computational Mathematics, (2017).
 
\bibitem{bouchitteSchweizer}{\sc G. Bouchitt\'e and B. Schweizer}, {\em Cloaking of small objects by anomalous localized resonance}, Quart.
J. Mech. Appl. Math., 63, (2010), pp.~437--463.

\bibitem{brbrcd} {\sc A. Braides, M. Briane and J. Casado-Diaz},
{\em Homogenization of non-uniformly bounded periodic diffusion energies in dimension two},
Nonlinearity, 22, (2009), pp.~1459--1480.

\bibitem{brianecasa}{\sc M. Briane and J. Casado-Diaz}, {\em Uniform convergence of sequences of solutions of two-dimensional linear elliptic equations with unbounded coefficients}, J. Differential Equations, 245, (2008), pp.~2038--2054.

\bibitem{brezis} {\sc H. Brezis},
{\em Functional Analysis, Sobolev Spaces and Partial Differential Equations}, Springer (2000).

\bibitem{ramdani} {\sc R. Bunoiu and K. Ramdani},
{\em Homogenization of materials with sign changing coefficients},
submitted (2015).

\bibitem{castro}
{\sc C. Castro and E. Zuazua}, {\em Une remarque sur l'analyse asymptotique spectrale en homog\'en\'eisation}, 
C. R. Acad. Sci. Paris, Ser. I 335 pp.99-104 (2002).

\bibitem{unfoldcras}
{\sc D. Cioranescu, A. Damlamian and G. Griso}, {\em 
Periodic unfolding and homogenization}, 
C. R. Acad. Sci. Paris, Ser. I 322, (1996), pp.1043-1047.

\bibitem{unfold} {\sc D. Cioranescu, A. Damlamian and G. Griso}
{\em The periodic unfolding method in homogenization},
SIAM J. Math. Anal., 40(4), (2008), pp.~1585--1620.

\bibitem{ciorainfty} {\sc D. Cioranescu, A. Damlamian and T. Li}
{\em Periodic Homogenization for Inner Boundary Conditions with Equi-valued Surfaces: The Unfolding Approach},
in Partial Differential Equations: Theory, Control and Approximation
(In Honor of the Scientific Heritage
of Jacques-Louis Lions), (2014).

\bibitem{concaplanchard}
{\sc C. Conca, J. Planchard and M. Vanninathan}, {\em Fluids and Periodic Structures}, 
RMA, 38, J. Wiley \& Masson, (1995).

\bibitem{concavan}
{\sc C. Conca and M. Vanninathan}, {\em Homogenization of periodic structures via Bloch decomposition}, 
SIAM J. Appl. Math., 57, (1997), pp.~1639--1659.

\bibitem{CostabelStephan} {\sc M. Costabel and E. Stephan}, {\em A direct boundary integral equation method for transmission problems}, 
Journal of Mathematical Analysis and Applications, 106, (1985), pp.~367--413.

\bibitem{elsayed} {\sc I. H. El-Sayed, X. Huang and M. A. El-Sayed}, {\em Surface Plasmon Resonance Scattering and Absorption
of anti-EGFR Antibody Conjugated Gold Nanoparticles
in Cancer Diagnostics: Applications in Oral Cancer}, 
Nano Lett., 5 (5), (2005), pp.~829--834.

\bibitem{folland}{\sc G. B. Folland}, {\em Introduction to Partial Differential Equations}, 2nd Edition,
Princeton University Press (1995).

\bibitem{gerard} {\sc P. G\'erard}, {\em Mesures semi-classiques et ondes de Bloch}, in S\'eminaire \'Equations aux D\'eriv\'ees Partielles
1990-1991, volume 16, Ecole Polytechnique, Palaiseau, (1991).

\bibitem{grieser} {\sc D. Grieser}, {\em The plasmonic eigenvalue problem}, Rev. Math. Phys., 26 (2014), 1450005.

\bibitem{griso} {\sc G. Griso}, {\em Analyse asymptotique de structures r\'eticul\'ees}, Th\`ese de l'Universit\'e Pierre et Marie Curie (Paris VI), (1996).

\bibitem{jikov}{\sc V. V. Jikov, S. M. Kozlov, and O. A. Oleinik}, {\em Homogenization of differential operators and integral functionals}, Springer-Verlag, Berlin, (1994).

\bibitem{john} {\sc F. John}, {\em The Dirichlet problem for a hyperbolic equation}, American Journal of Mathematics, 63,1, (1941), pp.~141--154.

\bibitem{kangln} {\sc H. Kang}, {\em Layer potential approaches to interface problems}, In Inverse Problems and Imaging: Panoramas et synth\`eses, 44. Soci\'et\'e Math\'ematique de France, (2013).

\bibitem{khavinson} {\sc D. Khavinson, M. Putinar and H.S. Shapiro},
{\em On Poincar\'e's variational problem in potential theory},
Arch. Rational Mech. Anal., 185, (2007), pp.~143--184.

\bibitem{kohnmilton}{\sc R.V. Kohn and G.W. Milton}, {\em On bounding the effective conductivity of anisotropic composites. Homogenization and Effective Moduli of Materials
and Media}, Eds. J.L. Ericksen, D. Kinderlehrer, R. Kohn, and J.-L. Lions,
IMA Volumes in Mathematics and its Applications, 1, Springer Verlag, (1986), pp.~97--125.

\bibitem{kuchment} {\sc P. Kuchment}, {\em Floquet Theory for Partial Differential Equations}, Birkh\"auser, Basel (1993).

\bibitem{lipton} {\sc R. Lipton and R. Viator}, {\em Bloch waves in crystals and periodic high contrast media}, submitted, (2016).

\bibitem{maier} {\sc S. A. Maier},
{\em Plasmonics: Fundamentals and Applications},
Springer, (2007).

\bibitem{manley} {\sc P. Manley, S. Burger, F. Schmidt and M. Schmid}, {\em Design Principles for Plasmonic Nanoparticle Devices}, in Progress in Nonlinear Nano-Optics
Part of the series Nano-Optics and Nanophotonics, (2015), pp.~223--247.

\bibitem{mclean} {\sc W. Mc Lean},
{\em Strongly Elliptic Systems and Boundary Integral Equations},
Cambridge University Press, Cambridge (2000).

\bibitem{moskowvog} {\sc S. Moskow and M.S. Vogelius}, {\em First-order corrections to the homogenised eigenvalues of a periodic composite medium. A convergence proof}, Proc. Roy. Soc., Edinburgh, 127, (1997), pp.~1263--1299.

\bibitem{MiltonNicorovici}{\sc N.A. Nicorovici, R.C. McPhedran and G.M. Milton}, {\em Optical and dielectric properties of partially resonant composites}, Phys. Rev. B 49, (1994),
pp.~8479--8482.

\bibitem{nguyen}{\sc H.-M. Nguyen}, {\em Cloaking using complementary media in the quasistatic regime}, Ann. I. H. Poincar\'e (C) Non Linear Analysis, 32 (2015), pp.~471-484.

\bibitem{hmrev}{\sc H.-M. Nguyen}, {\em Negative index materials and their applications: recent mathematics progress}, to appear in Chinese Annals of Mathematics, (2016).

\bibitem{nguetseng}{\sc G. Nguetseng}, {\em A general convergence result for a functional related to the theory of homogenization}, SIAM J. Math. Anal, 20:3 (1989), pp.~608--623.

\bibitem{planchard}{\sc J. Planchard}, {\em Global behaviour of large elastic tube bundles immersed in a fluid}, Computational Mechanics, 2, (1987), pp.~105--118.

\bibitem{otomori}{\sc M. Otomori, T. Yamada, K. Izui, S. Nishiwaki and J. Andkj\ae r}, {\em Topology optimization of hyperbolic metamaterials for an optical hyperlens}, Struct. Multidisc. Optim., (2016), DOI 10.1007/s00158-016-1543-x.

\bibitem{patching}{\sc S. G. Patching}, {\em Surface plasmon resonance spectroscopy for characterisation of membrane proteinÐligand interactions and its potential for drug discovery}, Biochimica et Biophysica Acta (BBA) - Biomembranes, 1838, 1, (2014), pp.~43--55.

\bibitem{poddubny} {\sc A. Poddubny, I. Iorsh, P. Belov, and Y. Kivshar}, {\em Hyperbolic metamaterials}, Nature Photon. 7, (2013), pp.~948--957.

\bibitem{reedsimon4} {\sc M. Reed and B. Simon},
{\em Methods of Modern Mathematical Physics, IV, Analysis of operators}, 
Academic Press, New York, (1978).

\bibitem{rudin} {\sc W. Rudin}, {\em Functional Analysis. 2nd ed.}, International Series in Pure and Applied Mathematics. New York, NY: McGraw-Hill, (1991).

\bibitem{sauterschwab} {\sc S. A. Sauter and C. Schwab},
{\em Boundary element methods}, 
Springer Series in Computational Mathematics, (2010).

\bibitem{shekhar}{\sc P. Shekhar, J. Atkinson, and Z. Jacob}, {\em Hyperbolic metamaterials: fundamentals and applications}, Nano Convergence, 1, (2014), pp.1--14.

\bibitem{trivau} {\sc F. Triki and M. Vauthrin}, {\em Mathematical modelization of the Photoacoustic effect generated by the heating of metallic nanoparticles},Êsubmitted, (2017). 

\bibitem{wilcox} {\sc C. Wilcox}, {\em Theory of Bloch waves}, J. Anal. Math., 33, (1978), pp.~146--167.

\end{thebibliography}
\end{document}